\documentclass{article}
\usepackage[utf8]{inputenc} 
\usepackage[T1]{fontenc}    
\usepackage{amsmath, amsfonts, amsthm, graphicx, microtype, hyperref}
\usepackage{grffile}
\usepackage{multirow}
\usepackage{subcaption}

\usepackage{iclr}
\iclrfinalcopy
\newcommand{\bal}{\bar{\lambda}}
\newcommand{\ubal}{\underline{\lambda}}

\newcommand{\bP}{\mathbf{P}}
\newcommand{\bJ}{\mathbf{J}}
\newcommand{\bF}{\mathbf{F}}
\newcommand{\bx}{\mathbf{x}}
\newcommand{\bSig}{\boldsymbol{\Sigma}}
\newcommand{\bsig}{\boldsymbol{\sigma}}

\newcommand{\bT}{\mathbf{T}}
\newcommand{\bz}{\mathbf{z}}
\newcommand{\ba}{\mathbf{a}}
\newcommand{\bW}{\mathbf{W}}
\newcommand{\bM}{\mathbf{M}}

\newcommand{\bb}{\mathbf{b}}
\newcommand{\bB}{\mathbf{B}}
\newcommand{\bv}{\mathbf{v}}
\newcommand{\Exp}{\mathbb{E}}
\newcommand{\bThet}{\boldsymbol{\Theta}}

\newcommand{\txb}{\tilde{\mathbf{x}}}
\newcommand{\tyb}{\tilde{\mathbf{y}}}

\newcommand{\comb}{\mathbf{x}^\bullet}

\newcommand{\com}{x^\bullet}

\newcommand{\comy}{y^\bullet}
\newcommand{\comyb}{\by^\bullet}
\newcommand{\bH}{\mathbf{H}}

\newcommand{\tlim}{\lim_{t\rightarrow \infty}}
\newcommand{\bg}{\mathbf{g}}
\newcommand{\bA}{\mathbf{A}}
\newcommand{\bI}{\mathbf{I}}

\newcommand{\by}{\mathbf{y}}
\newcommand{\bu}{\mathbf{u}}
\newcommand{\bQ}{\mathbf{Q}}
\newcommand{\bzeta}{\boldsymbol{\zeta}}
\newcommand{\beps}{\boldsymbol{\epsilon}}

\newtheorem{thm}{Theorem}[section]
\newtheorem{lem}{Lemma}[section]
\newtheorem{cor}{Corollary}[section]
\theoremstyle{remark}
\newtheorem{rmk}{Remark}[section]

\theoremstyle{definition}
\newtheorem{definition}{Definition}[section]

\theoremstyle{definition}
\newtheorem{assumption}{Assumption}[section]

\title{A continuous-time analysis of distributed stochastic gradient}
\author{%
  Nicholas M. Boffi\\
  John A. Paulson School of Engineering and Applied Sciences\\
  Harvard University\\
  Cambridge, MA 02138\\
  \texttt{boffi@g.harvard.edu} \\
   \And
   Jean-Jacques E. Slotine\\
   Nonlinear Systems Laboratory\\
   Massachusetts Institute of Technology\\
   Cambridge, MA 02139\\
   \texttt{jjs@mit.edu}
}

\begin{document}
\maketitle

\begin{abstract}
    We analyze the effect of synchronization on distributed stochastic gradient algorithms. By exploiting an analogy with dynamical models of biological quorum sensing -- where synchronization between agents is induced through communication with a common signal -- we quantify how synchronization can significantly reduce the magnitude of the noise felt by the individual distributed agents and by their spatial mean. This noise reduction is in turn associated with a reduction in the smoothing of the loss function imposed by the stochastic gradient approximation. Through simulations on model non-convex objectives, we demonstrate that coupling can stabilize higher noise levels and improve convergence.  We provide a convergence analysis for strongly convex functions by deriving a bound on the expected deviation of the spatial mean of the agents from the global minimizer for an algorithm based on quorum sensing, the same algorithm with momentum, and the Elastic Averaging SGD (EASGD) algorithm. We discuss extensions to new algorithms that allow each agent to broadcast its current measure of success and shape the collective computation accordingly. We supplement our theoretical analysis with numerical experiments on convolutional neural networks trained on the CIFAR-10 dataset, where we note a surprising regularizing property of EASGD even when applied to the non-distributed case. This observation suggests alternative second-order in-time algorithms for non-distributed optimization that are competitive with momentum methods.
\end{abstract}

\section{Introduction}
\label{sec:intro}
Stochastic gradient descent (SGD) and its variants have become the de-facto algorithms for large-scale machine learning applications such as deep neural networks \citep{bottou-2010, Goodfellow-et-al-2016, LeCun2015, Mallat2016}. SGD is used to optimize finite-sum loss functions, where a stochastic approximation to the gradient is computed using only a random selection of the input data points. Well-known results on almost-sure convergence rates to global minimizers for strictly convex functions and to stationary points for non-convex functions exist under sufficient regularity conditions \citep{Bottou1999, Robbins1971}. Classic work on iterate averaging for SGD \citep{pol-jud} and other more recent extensions \citep{Bach1, Bach2, Bach3, Bach4} can improve convergence under a set of reasonable assumptions typically satisfied in the machine learning setting. Convergence proofs rely on a suitably chosen decreasing step size; for constant step sizes and strictly convex functions, the parameters ultimately converge to a distribution peaked around the optimum.

For large-scale machine learning applications, parallelization of SGD is a critical problem of significant modern research interest \citep{large-scale-dist, Recht2013, hogwild, Soatto3}. Recent work in this direction includes the Elastic Averaging SGD (EASGD) algorithm, in which $p$ distributed agents coupled through a common signal optimize the same loss function. EASGD can be derived from a single SGD step on a global variable consensus objective with a quadratic penalty, and the common signal takes the form of an average over space and time of the parameter vectors of the individual agents \citep{EASGD, boyd2010}. At its core, the EASGD algorithm is a system of identical, coupled, discrete-time dynamical systems. And indeed, the EASGD algorithm has exactly the same structure as earlier mathematical models of synchronization~\citep{russo-slot,soon-jo} inspired by \emph{quorum sensing} in bacteria~\citep{bass1, bass2}. In these models, which have typically been analyzed in continuous-time, 
the dynamics of the common (quorum) signal can be arbitrary \citep{russo-slot}, and in fact may simply consist of a weighted average of individual signals. Motivated by this immediate analogy, we present here a continuous-time analysis of distributed stochastic gradient algorithms, of which EASGD is a special case. A significant focus of this work is the interaction between the degree of synchronization of the individual agents, characterized rigorously by a bound on the expected distance between all agents and governed by the coupling strength, and the amount of noise induced by their stochastic gradient approximations.

The effect of coupling between identical continuous-time dynamical systems has a rich history. In particular, synchronization phenomena in such coupled systems have been the subject of much mathematical \citep{Wang2005}, biological \citep{russo-slot}, neuroscientific \citep{tab-slot}, and physical interest \citep{PhysRevE.78.011108}. In nonlinear dynamical systems, synchronization has been shown to play a crucial role in \emph{protection} of the individual systems from independent sources of noise \citep{tab-slot}. The interaction between synchronization and noise has also been posed as a possible source of regularization in biological learning, where quorum sensing-like mechanisms could be implemented between neurons through local field potentials \citep{bouv-slot}. Given the significance of stochastic gradient \citep{Zhang2018} and externally injected \citep{grad-noise} noise in regularization of large-scale machine learning models such as deep networks \citep{understanding-DL}, it is natural to expect that the interplay between synchronization of the individual agents and the noise from their stochastic gradient approximations is of central importance in distributed SGD algorithms.

Recently, there has been renewed interest in a continuous-time view of optimization algorithms \citep{Wilson2016, Wibisono2016, Wibisono2015, Betancourt2018}. Nesterov's accelerated gradient method \citep{nest-mom} was fruitfully analyzed in continuous-time in \citet{Su2015}, and a unifying extension to other algorithms can be found in \citet{Wibisono2016}. Continuous-time analysis has also enabled \emph{discrete-time} algorithm development through classical discretization techniques from numerical analysis \citep{ZhangRK}. This paper further adds to this line of work by deriving new results with the mathematical tools afforded by the continuous-time view, such as stochastic calculus and nonlinear contraction analysis \citep{Lohmiller-Slot}.

The paper is organized as follows. In Sec.~\ref{sec:prelim}, we provide some necessary mathematical preliminaries: a review of SGD in continuous-time, a continuous-time limit of the EASGD algorithm, a review of stochastic nonlinear contraction theory, and a statement of some needed assumptions. In Sec.~\ref{sec:sync_noise}, we demonstrate that the effect of synchronization of the distributed SGD agents is to reduce the magnitude of the noise felt by each agent and by their spatial mean. We derive this for an algorithm where all-to-all coupling is implemented through communication with the spatial mean of the distributed parameters, and we refer to this algorithm as quorum SGD (QSGD). In the appendix, a similar derivation is presented with arbitrary dynamics for the quorum variable, of which EASGD is a special case. In Sec.~\ref{sec:kleinberg}, we connect this noise reduction property with a recent analysis in \citet{kleinberg}, which shows SGD can be interpreted as performing gradient descent on a smoothed loss in expectation. We use this derivation to garner intuition about the qualitative performance of distributed SGD algorithms as the coupling strength is varied, and we verify this intuition with simulations on model nonconvex loss functions in low and high dimensions. In Sec.~\ref{sec:conv}, we provide new convergence results for QSGD, QSGD with momentum, and EASGD for a strongly convex objective. In Sec.~\ref{sec:dnn}, we explore the properties of EASGD and QSGD on deep neural networks, and in particular, test the stability and performance of variants proposed throughout the paper. We also propose a new class of second-order in time algorithms motivated by the EASGD algorithm with a single agent, which consists of standard SGD coupled in feedback to the output of a nonlinear filter of the parameters. We close with some concluding remarks in Sec.~\ref{sec:conc}.

\section{Mathematical preliminaries}
\label{sec:prelim}
\noindent In this section, we provide a brief review of the necessary mathematical tools employed in this work.
\subsection{Convex optimization}
For the convergence proofs in Sec.~\ref{sec:conv}, and for synchronization of momentum methods, we will require a few standard definitions from convex optimization.
\begin{definition}{(Strong Convexity)}
A function $f \in \mathcal{C}^2(\mathbb{R}^n, \mathbb{R})$ is $l$-strongly convex with $l > 0$ if its Hessian is uniformly lower bounded by $l \bI$ with respect to the positive semidefinite order, $\nabla^2 f(\bx) > l \bI$ for all $\bx \in \mathbb{R}^n$.
\end{definition}
\begin{definition}{(L-Smoothness)}
A function $f \in \mathcal{C}^2(\mathbb{R}^n, \mathbb{R})$ is $L$-smooth with $L > 0$ if its Hessian is uniformly upper bounded by $L \bI$ with respect to the positive semidefinite order, $\nabla^2 f(\bx) < L \bI$ for all $\bx \in \mathbb{R}^n$.
\end{definition}
\subsection{Stochastic gradient descent in discrete-time}
\label{ssec:sgd}
Minibatch SGD has been essential for training large-scale machine learning models such as deep neural networks, where empirical risk minimization leads to finite-sum loss functions of the form 
\begin{equation}
    f(\bx) = \frac{1}{N}\sum_{i=1}^N l(\bx, \by^i)\nonumber.
\end{equation}
Above, $\by^i$ is the $i^{th}$ input data example and the vector $\bx$ holds the model parameters. In the typical machine learning setting where $N$ is very large, the gradient of $f$ requires $N$ gradient computations of $l$, which is prohibitively expensive. 

To avoid this calculation, a stochastic gradient is computed by taking a random selection $\mathcal{B}$ of size $b < N$, typically known as a minibatch. It is simple to see that the stochastic gradient
\begin{equation}
    \hat{\mathbf{g}}(\bx) = \frac{1}{b}\sum_{\by \in \mathcal{B}}\nabla l(\bx, \by)\nonumber
\end{equation}
is an unbiased estimator of the true gradient. The parameters are updated according to the iteration
\begin{equation}
    \bx_{t+1} = \bx_t - \eta \hat{\mathbf{g}}(\bx)\nonumber.
\end{equation}
By adding and subtracting the true gradient, the SGD iteration can be rewritten
\begin{equation}
    \bx_{t+1} = \bx_t - \eta \nabla f(\bx_t) - \frac{\eta}{\sqrt{b}} \bzeta_t,
    \label{eqn:sgd_disc}
\end{equation}
where $\bzeta_t \sim N(0, \bSig(\bx_t))$ is a data-dependent noise term. $\bzeta_t$ can be taken to be Gaussian under a central limit theorem argument, assuming that the size of the minibatch is large enough \citep{3fac, Mandt-cont}. $\bSig(\bx)$ is then given by the variance of a single-element stochastic gradient
\begin{equation}
    \bSig(\bx) = \frac{1}{N}\sum_{i=1}^N\left[\left(\nabla l\left(\bx, \by^i\right) - \nabla f\left(\bx\right)\right)\left(\nabla l\left(\bx, \by^i\right) - \nabla f\left(\bx\right)\right)^T\right]\nonumber.
\end{equation}

\subsection{Stochastic gradient descent in continuous-time}
\label{ssec:sgd_cont}
A significant difficulty in a continuous-time analysis of SGD is formulating an accurate stochastic differential equation (SDE) model. Recent works have proved rigorously \citep{Hu2017a, Feng2018, Li2018} that the sequence of values $\bx(k\eta)$ generated by the SDE
\begin{equation}
    d\bx = \left(-\nabla f(\bx) - \frac{1}{4}\eta\nabla\left\Vert\nabla f(\bx)\right\Vert^2\right)dt + \sqrt{\frac{\eta}{b}}\mathbf{B}(\bx)d\bW\nonumber
\end{equation}
approximates the SGD iteration with weak error $\mathcal{O}(\eta^2)$, where $\mathbf{W}$ is a Wiener process, $\Vert\cdot\Vert$ denotes the Euclidean 2-norm\footnote{\normalsize For the remainder of this paper, unless otherwise specified, we will use $\Vert \cdot \Vert$ to denote the 2-norm.}, and where $\bB\bB^T = \bSig$. Dropping the small term proportional to $\eta$ reduces the weak error to $\mathcal{O}(\eta)$ \citep{Hu2017a}. This leads to the SDE
\begin{equation}
    d\bx = -\nabla f(\bx)dt  + \sqrt{\frac{\eta}{b}}\bB(\bx)d\bW.
    \label{eqn:sgd_cont_typical}
\end{equation}
Equation (\ref{eqn:sgd_cont_typical}) has appeared in a number of recent works \citep{Mandt2017, Mandt2016, Mandt-cont, soatto1, soatto2, 3fac}, and is generally obtained by making the replacement $\eta \rightarrow dt$ and $\sqrt{\eta}\bzeta \rightarrow \bB d\bW$ in (\ref{eqn:sgd_disc}), as a sort of reverse Euler-Maruyama discretization \citep{kloe-plat}.

\subsection{EASGD in continuous-time}
\label{ssec:EASGD}
Following \citep{EASGD}, we provide a brief introduction to the EASGD algorithm, and convert the resulting sequences to continuous-time. We imagine a distributed optimization setting with $p \in \mathbb{N}$ agents and a single master. We are interested in solving a stochastic optimization problem
\begin{equation}
    \min_{\bx} F(\bx) = \Exp_{\bzeta}\left[f(\bx, \bzeta)\right]\nonumber,
    \label{eqn:obj}
\end{equation}
where $\bx \in \mathbb{R}^n$ is the vector of parameters and $\bzeta$ is a random variable representing the stochasticity in the objective. This is equivalent to the distributed optimization problem \citep{boyd2010}
\begin{equation}
    \min_{\bx^1, \hdots, \bx^p, \txb}\sum_{i=1}^p \left(\Exp_{\bzeta^i}[f(\bx^i, \bzeta^i)] + \frac{k}{2}\Vert \bx^i - \txb \Vert^2\right),
    \label{eqn:dist_obj}
\end{equation}
where each $\bx^i$ is a local vector of parameters and $\txb$ is the quorum variable. The quadratic penalty ensures that all local agents remain close to $\txb$, and $k$ sets the coupling strength. Smaller values of $k$ allow for more exploration, while larger values ensure a greater degree of synchronization. Intuitively, the interaction between agents mediated by $\txb$ is expected to help individual trajectories escape local minima, saddle points, and flat regions, unless they all fall into the same deep or wide minimum together.

We assume the expectation in (\ref{eqn:obj}) is approximated by a sum over input data points, and that the stochastic gradient is computed by taking a minibatch of size $b$. After taking an SGD step, the updates for each agent and the quorum variable become
\begin{align*}
    \bx^i_{t+1} &= \bx^i_t - \eta \nabla f(\bx^i_t) + \eta k \left(\txb_t - \bx^i_t\right) - \frac{\eta}{\sqrt{b}}\bzeta^i,\\
    \txb_{t+1} &= \txb_t + \eta p k \left(\comb_t - \txb_t\right),
\end{align*}
where $\comb_t = \frac{1}{p}\sum_{i=1}^p \bx^i_t$ and $\Exp\left[\bzeta^i \left(\bzeta^i\right)^T\right] = \bSig(\bx^i_t)$. Transferring to the continuous-time limit, these equations become,
\begin{align}
    \label{eqn:EASGD_cont_x}
    d\bx^i &= \left(-\nabla f(\bx^i) + k \left(\txb - \bx^i \right)\right)dt + \sqrt{\frac{\eta}{b}}\bB(\bx^i) d\bW^i,\\
    d\txb &= k p \left(\comb - \txb\right)dt,
    \label{eqn:EASGD_cont_quor}
\end{align}
with $\bB \bB^T = \bSig$. Note that in (\ref{eqn:EASGD_cont_quor}), the dynamics of $\txb$ represent a simple low-pass filter of the center of mass (spatial mean) variable $\comb$. In the limit of large $p$, the dynamics of this filter will be much faster than the SGD dynamics, and the continuous-time EASGD system can be approximately replaced by
\begin{equation}
    d\bx^i = \left(-\nabla f(\bx^i) + k \left(\comb - \bx^i\right)\right)dt + \sqrt{\frac{\eta}{b}}\bB(\bx^i) d\bW^i.
    \label{eqn:QSGD}
\end{equation}
We refer to (\ref{eqn:QSGD}) as Quorum SGD (QSGD), and it will be a significant focus of this work.

\subsection{Background on nonlinear contraction theory}
The main mathematical tool used in this work is nonlinear contraction theory, a form of incremental stability for nonlinear systems. In particular, we specialize to the case of time- and state-independent metrics; further details can be found in \citet{Lohmiller-Slot}.
\begin{definition}{(Contraction)}
The nonlinear dynamical system
\begin{equation}
    \dot{\bx} = \mathbf{f}(\bx, t),
    \label{eqn:gen_dyn}
\end{equation}
with $\bx \in \mathbb{R}^n$ and $\mathbf{f} \in \mathcal{C}^1(\mathbb{R}^n \times \mathbb{R}, \mathbb{R}^n)$ is said to be contracting with rate $\lambda > 0$ and invertible metric transformation $\bThet \in \mathbb{R}^{n \times n}$ if the symmetric part of the generalized Jacobian
\begin{equation}
    \left(\bThet \nabla \mathbf{f}(\bx, t) \bThet^{-1}\right)_s \leq -\lambda \bI
    \label{eqn:contr_jac}
\end{equation}
is uniformly negative definite for all $\bx \in \mathbb{R}^n$ and all $t \in \mathbb{R}$. Above, subscript $s$ denotes the symmetric part of a matrix, $\mathbf{A}_s = \frac{1}{2}\left(\mathbf{A} + \mathbf{A}^T\right)$. Equivalently, the system is said to be contracting in the corresponding metric $\bM = \bThet^T \bThet$.
\end{definition}

If condition (\ref{eqn:contr_jac}) is satisfied, all trajectories exponentially converge to one another regardless of initial conditions. That is, for two solutions $\bx_1(t)$ and $\bx_2(t)$ of (\ref{eqn:gen_dyn}),
\begin{equation}
    \Vert \bx_1(t) - \bx_2(t) \Vert_{\bM} \leq e^{-\lambda t}\Vert \bx_1(0) - \bx_2(0)\Vert_{\bM}
    \label{eqn:contr_forget}
\end{equation}
where $\Vert\bx\Vert_{\bM} = \sqrt{\bx^T \bM \bx}$. Intuitively, because of the property (\ref{eqn:contr_forget}), a nonlinear system is called contracting if differences in system trajectories due to initial conditions and temporary disturbances are exponentially forgotten. This behavior is proved differentially, by considering the time evolution of the squared Euclidean norm of the virtual displacement $\delta \bz = \bThet \delta \bx$, which formally obeys the differential equation $\delta\dot{\bz} = \bThet \nabla \mathbf{f}(\bx)\bThet^{-1}\delta \bz$ \citep{Lohmiller-Slot}. As an immediate and powerful corollary, if the system is contracting and a single trajectory is known, then all trajectories must converge to the single known trajectory exponentially. 

In this work, we will interchangeably refer to $\mathbf{f}$, the system, and the generalized Jacobian as contracting depending on the context. In particular, for stochastic differential equations, we will refer to $\mathbf{f}$ as contracting if the deterministic system is contracting. Two specific robustness results for contracting systems needed for the derivations in this work are summarized below.

\begin{lem}
\label{lem:pert}
     Consider the dynamical system (\ref{eqn:gen_dyn}), and assume that it is contracting with metric transformation $\bThet$ and contraction rate $\lambda$. Let $\chi = \Vert \bThet^{-1}\Vert \Vert \bThet \Vert$ denote the condition number of $\bThet$, where $\Vert \bThet \Vert = \sup_{\Vert \by \Vert = 1} \Vert \bThet \by \Vert$ denotes the induced matrix 2-norm. Consider the perturbed dynamical system
    \begin{equation}
        \dot{\bx} = \mathbf{f}(\bx, t) + \beps(\bx, t).
        \label{eqn:gen_dyn_pert}
    \end{equation}
    Then, for a solution $\bx_1$ of (\ref{eqn:gen_dyn}) and a solution $\bx_2$ of (\ref{eqn:gen_dyn_pert}), with $R = \Vert \bThet\left(\bx_1 - \bx_2\right)\Vert$,
    \begin{equation}
        \dot{R} + \lambda R \leq \Vert \bThet \beps(\bx, t)\Vert
        \label{eqn:rob_basic}
    \end{equation}
    Furthermore, if $\Vert \beps \Vert \leq A e^{-a t} + B$ with $A, B \in \mathbb{R}$ and $a \in \mathbb{R}^+$, then after exponential transients of rates $a$ and $\lambda$,
    \begin{equation}
        \Vert \bx_1(t) - \bx_2(t) \Vert \leq \frac{\chi B}{\lambda}
        \label{eqn:rob}
    \end{equation}
\end{lem}
\begin{proof}
See point (vii) of ``linear properties of generalized contraction analysis'' in \citet{Lohmiller-Slot} for the derivation of (\ref{eqn:rob_basic}). From (\ref{eqn:rob_basic}), $\dot{R} + \lambda R \leq \Vert \bThet \Vert \Vert \beps \Vert \leq \Vert \bThet \Vert\left(A e^{-at} + B\right)$. Convolving $e^{-\lambda t}$ with the right-hand side yields the inequality 
\begin{equation}
    R(t) \leq \Vert \bThet \Vert \left(\frac{B}{\lambda} + \frac{A e^{-\lambda t}}{a-\lambda} - \frac{B e^{-\lambda t}}{\lambda} -\frac{A e^{-a t}}{a-\lambda}\right)
\end{equation}
Noting that $\Vert \bx_1(t) - \bx_2(t)\Vert = \Vert \bThet^{-1} \bThet \left(\bx_1 - \bx_2\right)\Vert \leq \Vert \bThet^{-1}\Vert \Vert \bThet\left(\bx_1 - \bx_2\right)\Vert = \Vert \bThet^{-1}\Vert R$ yields the result (\ref{eqn:rob}).
\end{proof}

\begin{thm}
\label{thm:stoch}
 Consider the stochastic differential equation
    \begin{equation}
        d\bx = \mathbf{f}(\bx, t)dt + \bsig(\bx, t)d\bW,
        \label{eqn:stoch_contr}
    \end{equation}
    with $\bx \in \mathbb{R}^n$ and where $\bW$ denotes an $n$-dimensional Wiener process. Assume that there exists a positive definite metric $\bM = \bThet^T\bThet$ such that $\bx^T \bM \bx \geq \beta \Vert \bx \Vert^2$ with $\beta > 0$, and that $\mathbf{f}$ is contracting in this metric. Further assume that $Tr\left(\bsig(\bx, t)^T\bM\bsig(\bx, t)\right) \leq C$ where $C \in \mathbb{R}^+$. Then, for two trajectories $\ba(t)$ and $\bb(t)$ driven by independent sources of noise with stochastic initial conditions given by a probability distribution $p(\bzeta_1, \bzeta_2)$,
    \begin{align}
        \Exp\left[\Vert \ba(t) - \bb(t)\Vert^2\right] &\leq \frac{1}{\beta}\left(\Exp\left[\left(\Vert \ba(0) - \bb(0)\Vert_{\bM}^2 - \frac{C}{\lambda}\right)^+\right]e^{-2\lambda t} + \frac{C}{\lambda}\right)\nonumber
    \end{align}
    where $(\cdot)^+$ denotes the unit ramp (or ReLU) function. The expectation on the left-hand side is over the noise $d\bW(s)$ for all $s < t$, and the expectation on the right-hand side is over the distribution of initial conditions.
\end{thm}
See \citet{stoch-cont}, Thm.~2 for a proof of Thm.~\ref{thm:stoch}. A corollary that will be useful in Sec.~\ref{sec:conv} is as follows.
\begin{cor}
\label{cor:stoch_det}
Assume that the conditions of Thm.~\ref{thm:stoch} are satisfied. Then, for a trajectory $\bx_{nf}(t)$ of (\ref{eqn:gen_dyn}) and a trajectory $\bx(t)$ of (\ref{eqn:stoch_contr}),
\begin{equation}
    \Exp\left[\Vert \bx(t) - \bx_{nf}(t)\Vert^2\right] \leq \frac{1}{\beta}\left(\Exp\left[\left(\Vert \bx(0) - \bx_{nf}(0)\Vert_{\bM}^2 - \frac{C}{2\lambda}\right)^+\right]e^{-2\lambda t} + \frac{C}{2 \lambda}\right)\nonumber.
\end{equation}
\end{cor}
Cor.~\ref{cor:stoch_det} is obtained by following the proof of Thm.~2 in \citet{stoch-cont} with the restriction that one system is deterministic. To reduce the appearance of decaying exponential terms, in applications of Thm.~\ref{thm:stoch}, Cor.~\ref{cor:stoch_det}, and other related contraction-based bounds, we will simply state the final constant and the corresponding rate of exponential transients. The conditions of Thm.~\ref{thm:stoch} are worthy of their own definition.
\begin{definition}{(Stochastic contraction)}
If the conditions of Thm.~\ref{thm:stoch} are satisfied, the system (\ref{eqn:stoch_contr}) is said to be stochastically contracting in the metric $\bM$ (or with metric transformation $\bThet$) with bound $C$ and rate $\lambda$.
\end{definition}

In this work, we will also make use of an extension of contraction known as partial contraction originally introduced in \citet{Wang2005}. The procedure is summarized below.

\begin{thm}
    \label{thm:part_cont}
    Consider the nonlinear dynamical system (\ref{eqn:gen_dyn}) -- not assumed to be contracting -- and consider a contracting auxiliary system of the form
    \begin{equation}
        \dot{\by} = \mathbf{g}(\by, \bx, t)
        \label{eqn:virt_g}
    \end{equation}
    with the requirement that $\mathbf{g}(\bx, \bx, t) = \mathbf{f}(\bx, t)$.\footnote{\normalsize For example, say $\mathbf{f}(\bx, t) = -\bP(\bx) \bx$ with $\bP(\bx)$ a symmetric and uniformly positive definite matrix. Then $\mathbf{g}(\by, \bx, t) = -\bP(\bx)\by$ satisfies this restriction requirement. The $\by$ system is also contracting in $\by$, as the symmetric part of the Jacobian $\bJ_s = -\bP(\bx) < 0$ uniformly. On the other hand, the $\mathbf{f}(\bx, t)$ system has Jacobian $\frac{\partial f_i}{\partial x_j} = -P_{ij}(\bx) - \sum_k \frac{\partial P_{ik}(\bx)}{\partial x_j}x_k$, which has symmetric part with unknown definiteness without further assumptions on $\bP$.} Assume a single trajectory $\by(t)$ of (\ref{eqn:virt_g}) is known. Then all trajectories of (\ref{eqn:gen_dyn}) converge to $\by(t)$.
\end{thm}
\begin{proof}
    By assumption, (\ref{eqn:virt_g}) is contracting, and so all trajectories converge to $\by(t)$. Because $\mathbf{g}(\bx, \bx, t) = \mathbf{f}(\bx, t)$, any solution $\bx(t)$ of (\ref{eqn:gen_dyn}) is also a solution of (\ref{eqn:virt_g}), and hence must converge to $\by(t)$.
\end{proof}
We will commonly refer to the auxiliary $\by$ system in the above theorem as a \emph{virtual system}, and $\mathbf{f}$ is said to be partially contracting. Thm.~\ref{thm:part_cont} enables the application of contraction to systems which in themselves are not contracting, but can be embedded in a virtual system which is. 

This notion also extends to stochastic systems through the use of stochastic contraction. If a stochastically contracting system
\begin{equation}
    d\by = g(\by, \bx, t)dt + \boldsymbol{\Xi}(\by, \bx, t)d\bW
    \label{eqn:stoch_virt_sys}
\end{equation}
can be found such that $\mathbf{g}(\bx, \bx, t) = \mathbf{f}(\bx, t)$ and $\boldsymbol{\Xi}(\bx, \bx) = \boldsymbol{\sigma}(\bx, t)$, then trajectories of (\ref{eqn:stoch_virt_sys}) can be compared to trajectories of (\ref{eqn:gen_dyn}) through the application of Cor.~\ref{cor:stoch_det} or (\ref{eqn:stoch_contr}) through the application of Thm.~\ref{thm:stoch}.

\subsection{Assumptions}
\label{ssec:assump}
We require two main assumptions about the objective function $f(\bx)$, both of which have been employed in previous work analyzing synchronization and noise in nonlinear systems \citep{tab-slot}. The first is an assumption on the nonlinearity of the \emph{components} of the gradient.

\begin{assumption}
\label{asmpt:hess}
Assume that the Hessian matrix of each component of the negative gradient has bounded maximum eigenvalue, $\nabla^2 \left[\left(-\nabla f(\bx)\right)_j\right] \leq Q \bI$ for all $j$.
\end{assumption}

The second assumption is a condition on the robustness of the distributed gradient flows studied in this work to small, potentially stochastic perturbations.

\begin{assumption}
\label{asmpt:conv}
Consider two dynamical systems 
\begin{align}
    \dot{\bx} &= -\nabla f(\bx) + k (\bz - \bx),
    \label{eqn:asmp2_x}\\
    d\by&= \left(-\nabla f(\by) + k (\bz - \by) + \mathbf{P}_{l}\right)dt + \beta_q d\bW,
    \label{eqn:asmp2_y}
\end{align}
where $\mathbf{P}_{l}$ is a continuous-time stochastic process dependent on a parameter $l$ and $\beta_q \in \mathbb{R}$ is a real coefficient dependent on a parameter $q$. Denote by $\bx(t)$ the solution to (\ref{eqn:asmp2_x}) and by $\by_{l, q}(t)$ the solution to (\ref{eqn:asmp2_y}) with the same initial condition, $\bx(0) = \by_{l, q}(0)$.  We assume that $\lim_{l\rightarrow \infty}\mathbb{E}\left(\Vert\mathbf{P}_l\Vert\right) = 0$ and $\lim_{q \rightarrow \infty}\beta_q = 0$ implies that $\lim_{l\rightarrow \infty}\lim_{q\rightarrow \infty}\Vert \bx - \by_{l, q}\Vert = 0$ almost surely.
\end{assumption}

Continuous dependence of trajectories on parameters of the dynamics in the sense of Assumption \ref{asmpt:conv} can be characterized for deterministic systems through continuity assumptions on the dynamics --  see, for example, Section 3.2 in \citet{Khalil} -- here we assume a natural stochastic extension. Assumption \ref{asmpt:conv} has been verified for FitzHugh-Nagumo oscillators where $\bP_l$ is a white noise process \citep{Tuckwell1998}, and validated in simulation for more complex nonlinear oscillators \citep{tab-slot}. We remark that $\mathbb{E}\left[\Vert \bP\Vert\right] \rightarrow 0$ implies that $\Vert \bP \Vert \rightarrow 0$ almost surely, and hence that $\bP \rightarrow \mathbf{0}$ almost surely.

\section{Synchronization and noise}
\label{sec:sync_noise}
In this section, we analyze the interaction between synchronization of the distributed QSGD agents and the noise they experience. We begin with a derivation of a quantitative measure of synchronization that applies to a class of distributed SGD algorithms involving coupling to a common external signal with no communication delays. We then present the section's primary contribution, which will serve as a basis for the theory in the remainder of the paper, as well as for the intuition for various experiments.

\subsection{A measure of synchronization}
\label{ssec:sync_measure}
We now present a simple theorem on synchronization in the deterministic setting, which will allow us to prove a bound on synchronization in the stochastic setting using Thm.~\ref{thm:stoch}.
\begin{thm}
    \label{thm:sync}
    Consider the coupled gradient descent system
    \begin{equation}
        \dot{\bx}^i = -\nabla f(\bx^i) + k (\bz - \bx^i),
        \label{eqn:dist_GD}
    \end{equation}
    where $\bz$ represents a common external signal. Let $\bal$ denote the maximum eigenvalue of $-\nabla^2 f(\bx)$. For $k > \bal$, the individual $\bx^i$ trajectories synchronize exponentially with rate $k - \bal$ regardless of initial conditions.
\end{thm}
\begin{proof}
    Consider the auxiliary virtual system 
    \begin{equation}
        \dot{\by} = -\nabla f(\by) + k(\bz - \by),
        \label{eqn:virt_sync}
    \end{equation} 
    where $\bz$ is an external input. Note that with $\by = \bx^i$, we recover (\ref{eqn:dist_GD}) -- i.e., (\ref{eqn:virt_sync}) admits the trajectories of each agent $\bx^i$ as particular solutions. The Jacobian of (\ref{eqn:virt_sync}) is given by
    \begin{equation}
        \bJ = -\nabla^2 f(\by) - k \bI.
        \label{eqn:sync_jac}
    \end{equation}
    Equation (\ref{eqn:sync_jac}) is symmetric and negative definite for $k > \bal$ for any external input $\bz$. Because the individual $\bx^i$ are particular solutions of this virtual system, contraction implies that for all $i$ and $j$, $\Vert \bx^i - \bx^j\Vert \rightarrow 0$ exponentially. The contraction rate is given by $k - \bal$.
\end{proof}
This theorem motivates a definition.
\begin{definition}{(Global exponential synchronization)}
We will say the agents in a distributed algorithm globally exponentially synchronize if they all converge to one another exponentially regardless of initial conditions.
\end{definition}
Thm.~\ref{thm:sync} gives a simple condition on the coupling gain $k$ for synchronization of the individual agents in (\ref{eqn:dist_GD}). Because $\bz$ can represent any input, Thm.~\ref{thm:sync} applies to any dynamics of the quorum variable: with $\bz = \comb$, it applies to the QSGD algorithm, and with $\bz = \txb$, it applies to the EASGD algorithm. Under the assumption of a contracting \emph{deterministic} system, we can use the stochastic contraction results in Thm.~\ref{thm:stoch} to bound the expected distance between individual agents in the stochastic setting.
\begin{lem}
    \label{lem:sync_nomom}
    Assume that $k > \bar{\lambda}$ and that $Tr(\bB\bB^T) = Tr(\bSig) < C$ uniformly. Then, after exponential transients of rate $2(k - \bal)$, 
    \begin{equation}
        \mathbb{E}\left[\sum_{i}\Vert \bx^i - \comb\Vert^2\right] \leq \frac{\left(p-1\right)C\eta}{2 b\left(k - \bal\right)}
        \label{eqn:sync_com}
    \end{equation}
    where each $\bx^i$ is a solution of (\ref{eqn:EASGD_cont_x}) or (\ref{eqn:QSGD}).
\end{lem}
\begin{proof}
    Consider the systems for $i = 1, \hdots, p$
    \begin{equation}
        d\bx^i = \left(-\nabla f(\bx^i) + k\left(\bz - \bx^i\right)\right)dt + \sqrt{\frac{\eta}{b}}\bB(\bx^i) d\bW^i
        \label{eqn:stoch_gen}
    \end{equation}
    which reproduces (\ref{eqn:EASGD_cont_x}) with $\bz = \txb$ and (\ref{eqn:QSGD}) with $\bz = \comb$. Each solution $\bx^i$ to (\ref{eqn:stoch_gen}) is a solution of the stochastic virtual system
    \begin{equation}
        d\by = \left(-\nabla f(\by) + k\left(\bz - \by\right)\right)dt + \sqrt{\frac{\eta}{b}}\bB(\by) d\bW,
        \nonumber
    \end{equation}
    which has contracting deterministic part under the assumptions of the lemma and by Thm.~\ref{thm:sync}. For fixed $i$ and $j$, applying the results of Thm.~\ref{thm:stoch} in the Euclidean metric leads to
    \begin{equation}
        \Exp\left[\Vert\bx^i - \bx^j\Vert^2\right] \leq \frac{C\eta}{b(k - \bal)}
        \label{eqn:sync_bound_indiv}
    \end{equation}
    after exponential transients of rate $2(k-\bal)$. Summing (\ref{eqn:sync_bound_indiv}) over $i$ and $j$ leads to 
     \begin{equation}
         \mathbb{E}\left[\sum_{i < j}\Vert\bx^i - \bx^j\Vert^2\right] \leq \frac{p \left(p-1\right)\eta C}{2 b \left(k - \bal\right)}\nonumber.
        \label{eqn:sync_bound_total}
    \end{equation}
    Finally, as in \citet{tab-slot}, we can rewrite
    \begin{equation}
        \sum_{i<j}\Vert \bx^i - \bx^j\Vert^2 = p \sum_i \Vert \bx^i - \comb\Vert^2\nonumber,
    \end{equation}
    which proves the result. $\square$
\end{proof}
We will refer to (\ref{eqn:sync_com}) as a synchronization condition.

\subsection{Reduction of noise due to synchronization}
\label{ssec:sync_noise_analysis}
We now provide a mathematical characterization of how synchronization reduces the amount of noise felt by the individual QSGD agents. The derivation follows the mathematical procedure first employed in \citet{tab-slot} in the study of neural oscillators.
\begin{thm}{(The effect of synchronization on stochastic gradient noise)}
\label{thm:sync_noise}
Let $\comb_{k, p}(t)$ denote the center of mass trajectory of the continuous-time QSGD system (\ref{eqn:QSGD}) with coupling gain $k$ and $p$ agents. In the simultaneous limits $k \rightarrow \infty$ and $p \rightarrow \infty$, the difference between $\comb_{k, p}(t)$ and a trajectory of the noise-free dynamics
\begin{equation}
    \dot{\bx}_{\text{nf}} = - \nabla f(\bx_{\text{nf}})
\end{equation}
tends to zero, $\lim_{k \rightarrow \infty}\lim_{p \rightarrow \infty} \Vert \comb_{k, p} - \bx_{\text{nf}}\Vert \rightarrow 0$ almost surely, with $\bx_{\text{nf}}(0) = \comb_{k, p}(0)$.
\end{thm}
\begin{proof}
Summing the stochastic dynamics (\ref{eqn:QSGD}) over $p$, we find
\begin{equation}
    d\comb = \left[-\frac{1}{p}\sum_i \nabla f(\bx^i)\right]dt + \sqrt{\frac{\eta}{bp^2}}\sum_i \bB(\bx^i) d\bW^i.
    \label{eqn:com_ito_1}
\end{equation}
To make clear the dependence of the dynamics on $\comb$, we define the disturbance term
\begin{equation}
    \beps = -\frac{1}{p}\sum_i \nabla f(\bx^i) + \nabla f(\comb)\nonumber,
\end{equation}
so that we can rewrite (\ref{eqn:com_ito_1}) as
\begin{equation}
    d\comb = \left[-\nabla f(\comb) + \beps\right]dt + \sqrt{\frac{\eta}{bp^2}}\sum_i \bB(\bx^i) d\bW^i.
    \label{eqn:com_ito_2}
\end{equation}
Each term $\sqrt{\frac{\eta}{bp^2}}\bB(\bx^i) d\bW^i$ is a Gaussian random variable with covariance $\frac{\eta}{bp^2}\bSig(\bx^i)$, and each $d\bW^i$ is independent of all other $d\bW^j$. Hence the sum over the noise terms in (\ref{eqn:com_ito_2}) can also be written as a single Gaussian random variable with covariance $\frac{\eta}{bp^2}\sum_i\bSig(\bx^i)$,
\begin{equation}
    \sqrt{\frac{\eta}{bp^2}}\sum_i \bB(\bx^i) d\bW^i = \sqrt{\frac{\eta}{bp^2}}\bT d\bW.
    \label{eqn:avg_noise}
\end{equation}
where $\bT = \bT(\bx^1, \hdots, \bx^p)$ and $\bT\bT^T = \sum_i \bSig(\bx^i)$. (\ref{eqn:avg_noise}) leads to an additional simplification of (\ref{eqn:com_ito_2}),
\begin{equation}
    d\comb = \left[-\nabla f(\comb) + \beps\right]dt + \sqrt{\frac{\eta}{bp^2}}\bT d\bW.
    \label{eqn:com_ito_3}
\end{equation}
(\ref{eqn:com_ito_3}) shows that the effect of the additive noise is eliminated as the number of agents $p\rightarrow \infty$\footnote{\normalsize Indeed, the covariance $\frac{\eta}{bp^2}\sum_i \bSig(\bx^i) \leq \frac{\eta}{bp}\bar{\bSig}$ where $\bar{\bSig} = \max_i \bSig(\bx^i)$ and the $\max$ and $\leq$ are with respect to the positive semidefinite order. The covariance $\frac{\eta}{bp}\bar{\bSig}$ tends to zero as $p\rightarrow \infty$, so that Gaussian random variables drawn from a distribution with this covariance will become increasingly concentrated around zero with increasing $p$. Because the true covariance $\frac{\eta}{bp^2}\bT\bT^T$ is less positive semidefinite, random variables drawn from the true distribution will too become concentrated around zero as $p\rightarrow \infty$.}. We now let $\bF_j$ denote the gradient of $\left(-\nabla f(\bx)\right)_j$, and we let $\bH_j$ denote its Hessian. We apply the Taylor formula with integral remainder to $\left(-\nabla f(\bx)\right)_j$,
\begin{align}
    \left(-\nabla f(\bx^i)\right)_j &+ \left(\nabla f(\comb)\right)_j - \bF_j^T(\comb)(\bx^i - \comb)\nonumber\\
    &= \int_0^1 (1-s)\left(\bx^i - \comb\right)^T\bH_j\left((1-s)\bx^i + s \comb\right)\left(\bx^i - \comb\right).
    \label{eqn:bound}
\end{align}
Summing (\ref{eqn:bound}) over $i$ and applying the assumed bound $\bH_j \leq Q \bI$ leads to the inequality
\begin{equation}
    \left|p \left(\nabla f(\comb)\right)_j - \sum_i \left(\nabla f(\bx^i)\right)_j\right| \leq \frac{Q}{2}\sum_i\Vert \bx^i - \comb\Vert^2\nonumber.
\end{equation}
The left-hand side of the above inequality is $p|\beps_j|$. Squaring both sides and summing over $j$ provides a bound on $p^2\Vert \beps\Vert^2$. Taking a square root of this bound, we find
\begin{equation}
    \Vert \beps \Vert \leq \frac{\sqrt{n}Q}{2p}\sum_i \Vert \bx^i - \comb\Vert^2,
    \nonumber
\end{equation}
where the factor of $\sqrt{n}$ originates from the sum over the components of $\beps$. Performing an expectation over the noise $d\bW(s)$ for all $s < t$ and using the synchronization condition in (\ref{eqn:sync_com}), we conclude that after exponential transients of rate $2(k-\bal)$,
\begin{equation}
    \mathbb{E}\left[\left\Vert\beps\right\Vert\right] \leq \frac{(p-1) \sqrt{n} Q C \eta}{4p \left(k - \bal\right) b}.
    \label{eqn:eps_bound}
\end{equation}
The bound in (\ref{eqn:eps_bound}) depends on the synchronization rate of the agents $k - \bal$, the dimensionality of space $n$, the bound on the third derivative of the objective $Q$, and the bound on the noise strength $\frac{\eta}{b}C$. In the limit of large $p$, the dependence on $p$ becomes negligible. The expected effect of the disturbance term $\beps$ tends to zero as the coupling gain $k$ tends to infinity, corresponding to the fully synchronized limit. 

By Assumption \ref{asmpt:conv} and Thm.~\ref{thm:stoch}, as $k \rightarrow \infty$ and $p \rightarrow \infty$, the difference between trajectories of (\ref{eqn:com_ito_3}) and the unperturbed, noise-free system tends to zero almost surely, as the effects of both the stochastic disturbance $\beps$ and the additive noise term are eliminated in this simultaneous limit.
\end{proof}

\subsection{Discussion}
\label{ssec:sync_noise_disc}
Thm.~\ref{thm:sync_noise} demonstrates that for distributed SGD algorithms, roughly speaking, the noise strength is set by the ratio parameter $\frac{\eta}{bp}$ at the expense of a distortion term which tends to zero with synchronization. Whether this noise reduction is a benefit or a drawback for non-convex optimization depends on the problem at hand. 

If the use of a stochastic gradient is purely as an approximation of the true gradient -- for example, due to single-node or single-GPU memory limitations -- then synchronization can be seen as improving this approximation and eliminating undesirable noise while simultaneously parallelizing the optimization problem. The analysis in this section then gives rigorous bounds on the magnitude of noise reduction. The $\beps$ term could be measured in practice to understand the empirical size of the distortion, and $k$ could be increased until $\beps$ tends approximately to zero and the noise is reduced to a desired level.

On the other hand, many studies have reported the importance of stochastic gradient noise in deep learning, particularly in the context of generalization performance \citep{Poggio2017, Zhu2018, soatto1, understanding-DL}. Furthermore, large batches are known to cause issues with generalization, and this has been hypothesized to be due to a reduction in the noise magnitude due to a higher $b$ in the ratio $\frac{\eta}{b}$ \citep{Keskar2016}. In this context, reduction of noise may be undesirable, and one may only be interested in parallelization of the problem. The above analysis then suggests choosing $k$ high enough such that the quorum variable represents a meaningful average of the parameters, but low enough that the noise in the SGD iterations is not reduced. Indeed, in Sec.~\ref{sec:dnn}, we will find the best generalization performance for low values of $k$ which still result in convergence of the quorum variable. For deep networks, the level of synchronization for a given value of $k$ will be both architecture and dataset-dependent. 

We remark that the condition in Thm.~\ref{thm:sync} is merely a sufficient condition for synchronization, and synchronization may occur for significantly lower values of $k$ than predicted by contraction in the Euclidean metric. However, independent of when synchronization exactly occurs, so long as there is a fixed upper bound as in (\ref{eqn:sync_com}), the results in this section will apply with the corresponding estimate of $\Exp[\Vert\beps\Vert]$.

\subsection{Extension to multiple learning rates}
Our analysis can be extended to the case when each individual agent has a \emph{different} learning rate $\eta_i$ (or equivalently, different batch size), and thus a different noise level. In effect, this is because each agent still follows the same dynamics, though with different integration errors, and at a different rate. In this case, the synchronization condition (\ref{eqn:sync_com}) is modified to
\begin{equation}
    \Exp\left[\sum_i \Vert \bx^i - \comb \Vert^2\right] \leq \frac{C}{2 p b (k - \bar{\lambda)}}\sum_{i < j}\left(\eta_i + \eta_j\right),
    \nonumber
\end{equation}
so that
\begin{equation}
    \Exp\left[\Vert \beps \Vert\right] \leq \frac{\sqrt{n} Q C}{4 p^2 b (k - \bal)}\sum_{i < j}\left(\eta_i + \eta_j\right).
\end{equation}
The noise term $\sum_i \sqrt{\frac{\eta^i}{bp^2}}\bB(\bx^i) d\bW^i$ becomes a sum of $p$ independent Gaussians each with covariance $\frac{\eta^i}{b p^2}\bSig(\bx^i)$, and can be written as a single Gaussian random variable $\sqrt{\frac{1}{bp^2}}\bT d\bW$ with $\bT\bT^T = \sum_i \eta^i \bSig(\bx^i)$. An analogous argument as given in Sec.~\ref{ssec:sync_noise_analysis} shows that the effect of this additive noise will tend to zero as $p\rightarrow \infty$. This could allow, for example, for multiresolution optimization, where agents with larger learning rates may help avoid sharper local minima, saddle points, and flat regions of the parameter space, while agents with finer learning rates may help converge to robust local minima which generalize well. Standard learning rate schedules can also be applied agent-wise using the validation loss of individual agents, rather than decreasing all learning rates using the validation loss of the quorum variable.

\subsection{Extension to momentum methods}
Our analysis can also be extended to momentum methods, modeled using the differential equation \citep{Su2015}
\begin{equation*}
    \ddot{\bx}^i + \mu(t)\dot{\bx}^i + \nabla f(\bx^i) = 0,
\end{equation*}
in component-wise form
\begin{align*}
    \dot{\bx}_1^i &= \bx_2^i,\\
    \dot{\bx}_2^i &= -\nabla f(\bx_1^i) - \mu(t)\bx_2^i.
\end{align*}
Coupling the agents in both position and velocity leads to the dynamics,
\begin{align}
    \label{eqn:qsgd_mom_1}
    \dot{\bx}_1^i &= \bx_2^i + k_1 (\comb_1 - \bx_1^i),\\
    \label{eqn:qsgd_mom_2}
    \dot{\bx}_2^i &= -\nabla f(\bx_1^i) - \mu(t)\bx_2^i + k_2 (\comb_2 - \bx_2^i),
\end{align}
where $\comb_l = \frac{1}{p}\sum_j \bx_l^j$. 
\begin{lem}
  \label{lem:qsgd_mom_sync}
  Consider the QSGD with momentum system given by (\ref{eqn:qsgd_mom_1}) and (\ref{eqn:qsgd_mom_2}). Assume that $f$ is $\underline{\lambda}$-strongly convex and $\bar{\lambda}$-smooth. For $k_1 > \frac{1}{4\left(\inf_t \mu(t) + k_2\right)}\max\left(\left(1-\bal\right)^2, \left(1-\ubal\right)^2\right)$, the individual $\bx^i$ systems globally exponentially synchronize with rate $\xi$, where
  \begin{equation}
      \xi \geq \frac{k_1 + \inf_t \mu(t) + k_2}{2} - \sqrt{\left(\frac{k_1 - (\inf_t \mu(t) + k_2)}{2}\right)^2 + \frac{\max\left(\left(1-\bal\right)^2, \left(1-\ubal\right)^2\right)}{4}}
      \label{eqn:qsgd_mom_sync_bound}
  \end{equation}
\end{lem}
\begin{proof}
    The virtual system
    \begin{align}
        \label{eqn:mom_virt_1}
        \dot{\by}_1 &= \by_2 + k_1 (\comb_1 - \by_1),\\
        \dot{\by}_2 &= -\nabla f(\by_1) - \mu(t)\by_2 + k_2 (\comb_2 - \by_2),
        \label{eqn:mom_virt_2}
    \end{align}
    has system Jacobian
    \begin{equation*}
        \mathbf{J} = \begin{pmatrix} -k_1\mathbf{I} & \mathbf{I} \\ -\nabla^2 f(\by_1) & -\left(\mu(t) + k_2\right)\mathbf{I} \end{pmatrix},
    \end{equation*}
    and will be contracting for $\left(\inf_{t}\mu(t) + k_2\right)k_1 > \sup_\bx\left(\sigma^2\left(\frac{1}{2}\left(-\nabla^2 f(\bx) + \mathbf{I}\right)\right)\right)$, where $\sigma^2(\cdot)$ denotes the largest squared singular value \citep{Wang2005}. Because $\bI - \nabla^2f$ is symmetric, the square singular values are simply the square eigenvalues. This leads to the condition $\left(\inf_{t}\mu(t) + k_2\right)k_1 > \frac{1}{4}\max\left((1-\bal^2), (1-\ubal)^2\right)$, which may be rearranged to yield the condition in the theorem.
    
    (\ref{eqn:mom_virt_1}) and (\ref{eqn:mom_virt_2}) also admit the $\bx_l^i$ as particular solutions, so that the agents globally exponentially synchronize with a rate $\xi = |\lambda_{\text{max}}(\bJ)|$. The lower bound on $\xi$ can be obtained by application of the result in \citet{slot-modular}, Example 3.8.
\end{proof}
Hence, a bound similar to (\ref{eqn:sync_com}) can be derived just as in Lemma \ref{lem:sync_nomom}. Because the $\comb_1$ dynamics are linear, and because the $\comb_2$ dynamics are only nonlinear through the gradient of the loss, Assumption \ref{asmpt:hess} does not need to be modified. For $\inf_t \mu(t) > 0$, $k_2$ can be set to zero, so that coupling is only through the position variables.

\section{An alternative view of distributed stochastic gradient descent}
\label{sec:kleinberg}
In this section, we connect the above discussion of synchronization and noise reduction with the analysis in~\citet{kleinberg}, which interprets SGD as performing gradient descent on a smoothed loss in expectation. Specifically, we show that the reduction of noise due to synchronization can be viewed as a reduction in the smoothing of the loss function. This provides further geometrical intuition for the effect of synchronization on distributed SGD algorithms. It furthermore sheds light as to why one may want to use low values of $k$ to prevent noise reduction in learning problems involving generalization, where optimization of the empirical risk rather than the expected risk introduces spurious defects into the loss function that may be removed by sufficient smoothing.

Defining the auxiliary sequence $\by_t = \bx_t - \eta \nabla f(\bx_t)$  and comparing with (\ref{eqn:sgd_disc}) shows that $\bx_{t+1} = \by_t - \frac{\eta}{\sqrt{b}} \bzeta_t$, yielding
\begin{equation}
    \by_{t+1} = \by_t - \eta \nabla f\left(\by_t - \frac{\eta}{\sqrt{b}} \bzeta_t\right) - \frac{\eta}{\sqrt{b}} \bzeta_t \nonumber,
    \label{eqn:y_disc}
\end{equation}
so that
\begin{equation}
    \Exp_{\bzeta_t}\left[\by_{t+1}\right] = \by_t - \eta \nabla \Exp_{\bzeta_t}\left[f\left(\by_t - \frac{\eta}{\sqrt{b}} \bzeta_t\right)\right] \nonumber.
    \label{eqn:y_disc_exp}
\end{equation}
This demonstrates that the $\by$ sequence performs gradient descent on the loss function convolved with the $\frac{\eta}{\sqrt{b}}$-scaled noise in expectation\footnote{\normalsize In \citet{kleinberg}, the authors group the factor of $\sqrt{1/b}$ with the covariance of the noise.}. Using this argument, it is shown in \citet{kleinberg} that SGD can converge to minimizers for a much larger class of functions than just convex functions, though the convolution operation can disturb the locations of the minima.

\subsection{The effect of synchronization on the convolution scaling}
The analysis in Sec.~\ref{sec:sync_noise} suggests that synchronization of the $\bx^i$ variables should reduce the convolution prefactor for a $\by$ variable related to the center of mass, and we now make this intuition more precise for the QSGD algorithm. We have that
\begin{equation}
    \Delta \bx^i_t = - \eta \nabla f(\bx^i_t) + \eta k (\comb_t - \bx^i_t) - \frac{\eta}{\sqrt{b}} \bzeta_t^i\nonumber,
\end{equation}
so that
\begin{align}
    \Delta \comb_t &= -\eta \nabla f(\comb_t) + \eta \beps_t - \frac{\eta}{\sqrt{bp}}\bzeta_t\nonumber,
\end{align}
with $\beps_t = \nabla f(\comb_t) - \frac{1}{p}\sum_i \nabla f(\bx^i_t)$ as usual. Define the auxiliary variable $\comyb_t = \comb_t - \eta \nabla f(\comb_t)$, so that
\begin{equation}
    \comb_{t+1} = \comyb_t + \eta \beps_t - \frac{\eta}{\sqrt{bp}}\bzeta_t.
    \label{eqn:comb_y_nq}
\end{equation}
Equation (\ref{eqn:comb_y_nq}) can then be used to state
\begin{align}
    \comyb_{t+1} &= \comyb_t - \eta \nabla f(\comb_{t+1}) + \eta \beps_t - \frac{\eta}{\sqrt{bp}}\bzeta_t\nonumber,\\
    &= \comyb_t - \eta \nabla f\left(\comyb_t - \frac{\eta}{\sqrt{bp}}\bzeta_t + \eta \beps_t\right) + \eta \beps_t - \frac{\eta}{\sqrt{bp}}\bzeta_t\nonumber.
\end{align}
Taylor expanding the gradient term, we find
\begin{equation}
    \nabla f\left(\comyb_t - \frac{\eta}{\sqrt{bp}}\bzeta_t + \eta \beps_t\right) = \nabla f\left(\comyb_t - \frac{\eta}{\sqrt{bp}}\bzeta_t\right) + \eta\nabla^2f\left(\comyb_t - \frac{\eta}{\sqrt{bp}}\bzeta_t\right)\beps_t + \mathcal{O}(\eta^2)\nonumber,
\end{equation}
which alters the discrete $\comy$ update to
\begin{align}
    \Delta \comyb_{t} &= -\eta \nabla f\left(\comyb_t - \frac{\eta}{\sqrt{bp}}\bzeta_t\right) + \eta\left(1 - \eta \nabla^2 f\left(\comyb_t - \frac{\eta}{\sqrt{bp}}\bzeta_t\right)\right)\beps_t - \frac{\eta}{\sqrt{bp}}\bzeta_t + \mathcal{O}(\eta^3).
    \label{eqn:comb_y_nq_disc_final}
\end{align}

Equation (\ref{eqn:comb_y_nq_disc_final}) says that, in expectation, $\comy$ performs gradient descent on a convolved loss with noise scaling reduced by a factor of $\frac{1}{\sqrt{p}}$. The reduced scaling comes at the expense of the usual disturbance term $\beps$, which decreases to zero with increasing synchronization in expectation over the noise $\zeta_s$ for $s < t$. Equation (\ref{eqn:comb_y_nq_disc_final}) differs from the non-distributed case by an additional $\mathcal{O}(\eta^2)$ factor of the Hessian.

\subsection{Discussion}
\label{ssec:klein_disc}
To better understand the interplay of synchronization and noise in SGD, we can consider several limiting cases. Consider a choice of $\eta$ corresponding to a fairly high noise level, so that the loss function is sufficiently smoothed for the iterates of SGD ($k= 0$) to avoid local minima, saddle points, and flat regions, but so that the iterates would not reliably converge to a desirable region of parameter space, such as a deep and robust minimum.

For $k\rightarrow \infty$ and $p$ sufficiently large, the quorum variable will effectively perform gradient descent on a minimally smoothed loss, and will converge to a local minimum of the true loss function close to its initialization. Due to the strong coupling, the agents will likely get pulled into this minimum, leading to convergence as if a single agent had been initialized using deterministic gradient descent at $\comb(t=0)$, despite the high value of $\eta$.

With an intermediate value of $k$ so that the agents remain in close proximity to each other, but not so strong that $\Vert \beps \Vert \rightarrow 0$, the $\bx$ variables will be concentrated around the minima of the smoothed loss (the coupling will pull the agents together, but because $\Vert \beps \Vert \neq 0$, the smoothing will not be reduced in the sense of (\ref{eqn:comb_y_nq_disc_final})). The stationary distribution of SGD is thought to be biased towards concentration around degenerate minima of high volume \citep{Poggio2019}; the coupling force should thus amplify this effect, and lead to an accumulation of agents in wider and deeper minima in which all agents can approximately fit. Eventually, if sufficiently many agents arrive in a single minimum, it will be extremely difficult for any one agent to escape, leading to a consensus solution chosen by the agents even at a high noise level.

\subsection{Numerical simulations in nonconvex optimization}
\label{ssec:klein_sims}
In this subsection, we consider simulations on a model one-dimensional nonconvex loss function, as well as one possible high-dimensional generalization. There are several goals of the discussion. The first is to show that the intuition presented in Sec.~\ref{ssec:klein_disc} is correct. The second is to provide a setting where visualization of the loss function, its analytically smoothed counterpart, and the distribution of possible convergent points is straightforward. The third is to elucidate qualitative trends in distributed nonconvex optimization as a function of $k$ in low- and high-dimensional settings, and to show to what extent properties of the low-dimensional setting translate to the high-dimensional setting. We consider the loss function
\begin{equation}
    f(x) = \frac{\left(x^4-4 x^2+ \frac{1}{5}x + \frac{2}{5} \left(3 \sin (20 x)-\frac{7}{2} \sin (2 \pi  x)+\cos \left(\frac{10 e x}{3}\right)\right)\right)}{F},
    \label{eqn:double_well}
\end{equation}
where the sinusoidal oscillations in (\ref{eqn:double_well}) introduce spurious local minima. The constant factor $F \in \mathbb{R}^+$ is used for numerical stability for a wider range of $\eta$ values, in order to reduce the large gradient magnitudes introduced by the high-frequency modes. We simulate the dynamics of QSGD using a forward Euler discretization,
\begin{equation}
    x^i_{t+1} = x^i_t - \eta \nabla f(x^i) + \eta k (\com_t - x^i_t) - \eta \zeta^i_t \hspace{20mm} i = 1 \hdots p,
    \label{eqn:euler_disc}
\end{equation}
with $f(x)$ given by (\ref{eqn:double_well}). We include $1000$ agents in each of $250$ simulations per $k$ value. Each simulation is allowed to run for $20,000$ iterations with $\eta = .15$\footnote{\normalsize We choose a relatively high value of $\eta$ so that the convolved loss will be qualitatively different from the true loss to a degree that is visible by eye. This enables us to distinguish convergence to true minima from convergence to minima of the convolved loss. An alternative and equivalent choice would be to choose $\eta$ smaller, with a correspondingly wider distribution of the noise.}. The corresponding distributions of final points, computed via a kernel density estimate, are plotted over a range of $k$ values in Fig.~\ref{fig:kleinberg}. In each subfigure, the true loss function is plotted in orange and the loss function convolved with the noise distribution is plotted in blue. The loss functions are normalized so they can appear on the same scale as the distributions, and the $y$ scale is thus omitted. The agents are initialized uniformly over the interval $[-3, 3]$, and each experiences an i.i.d. uniform noise term $\zeta^i_t \sim U(-1.5, 1.5)$ per iteration. $F$ is fixed at $150$.

\begin{figure}

    \begin{tabular}{cc}
        \begin{subfigure}{.5\textwidth}
            \centering
            \includegraphics[width=\textwidth]{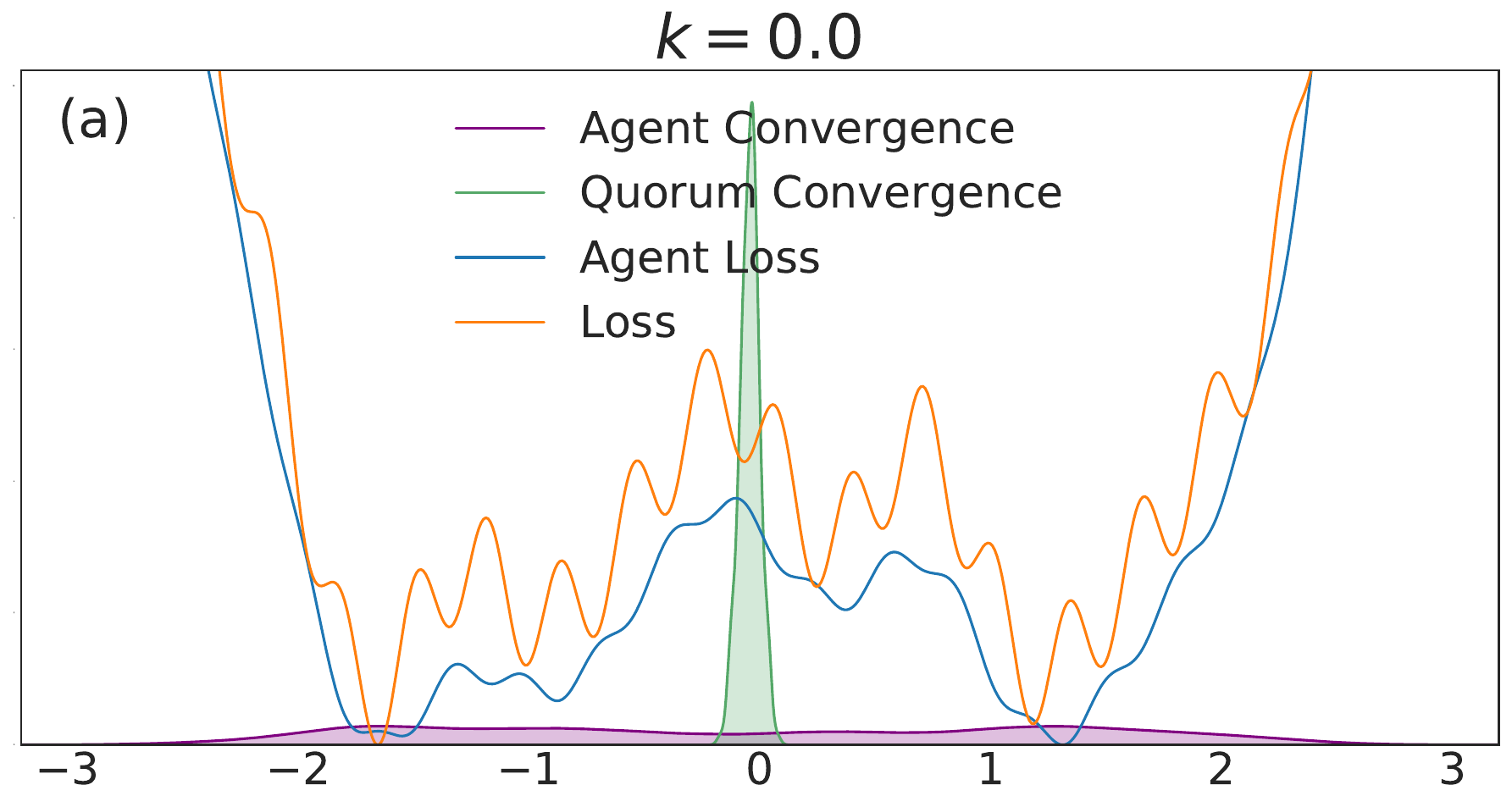}
        \end{subfigure}&
        
        \begin{subfigure}{.5\textwidth}
            \centering
            \includegraphics[width=\textwidth]{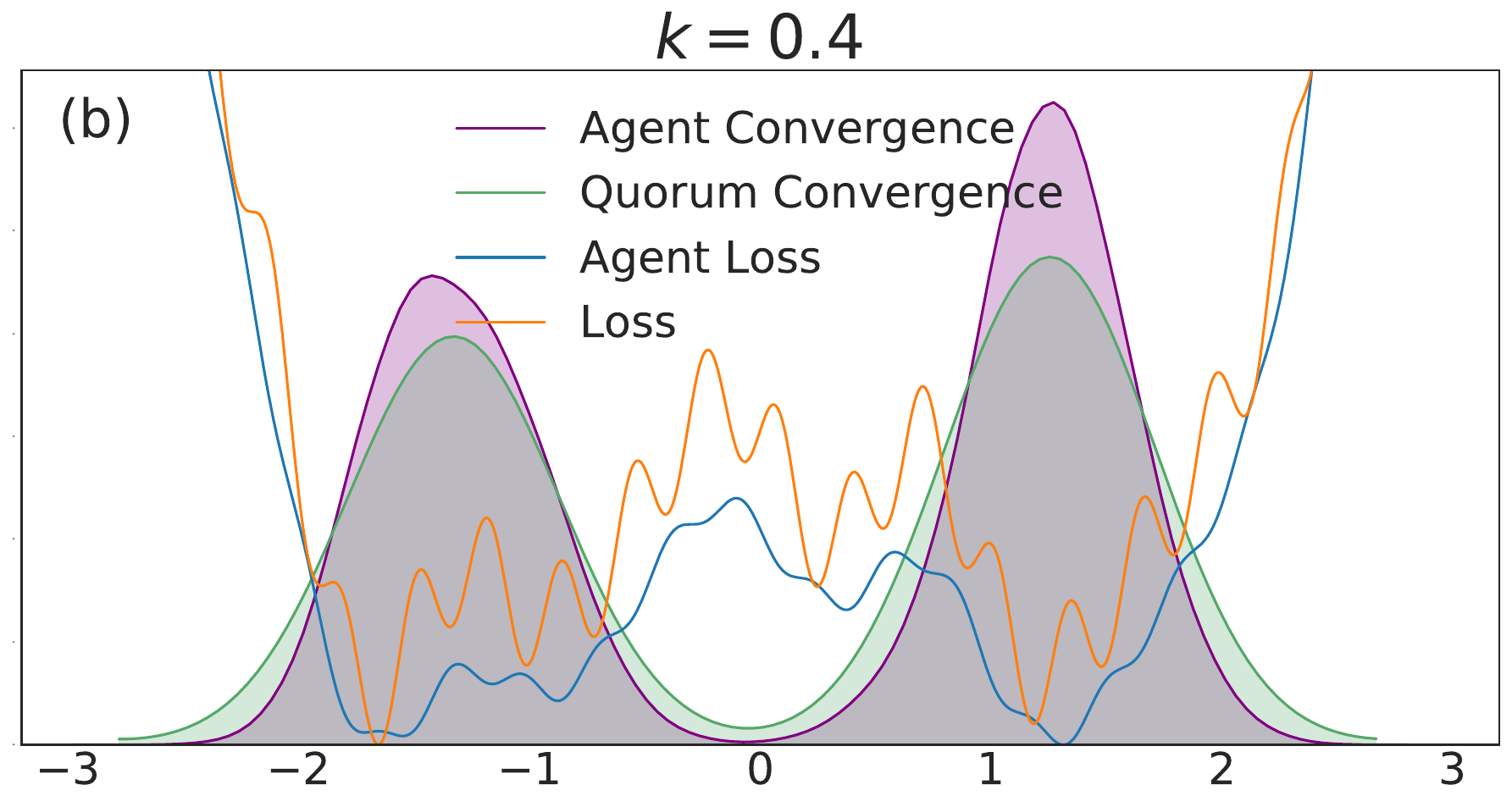}
        \end{subfigure}\\
        
        \begin{subfigure}{.5\textwidth}
            \centering
            \includegraphics[width=\textwidth]{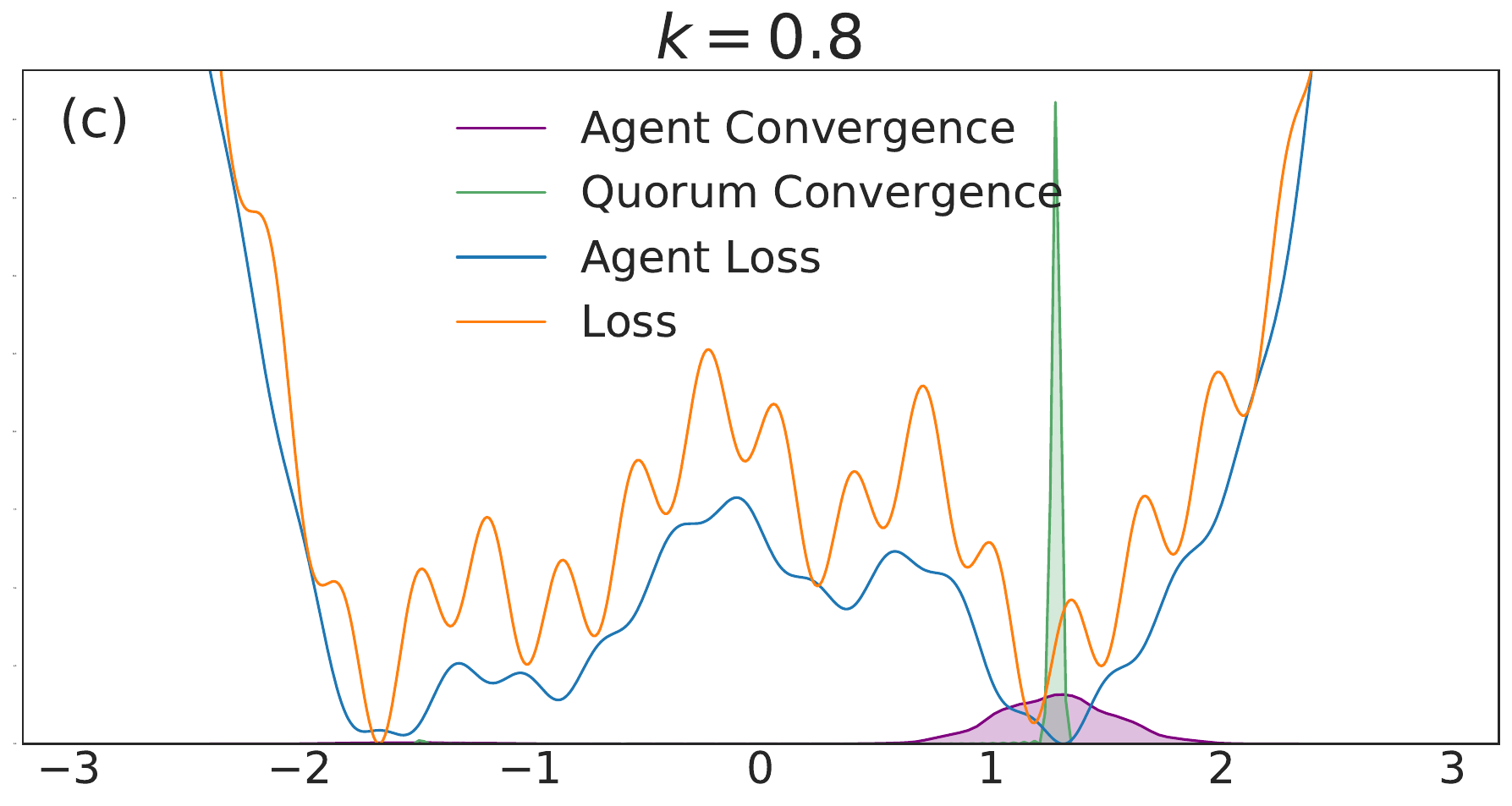}
        \end{subfigure}&
        
        \begin{subfigure}{.5\textwidth}
            \centering
            \includegraphics[width=\textwidth]{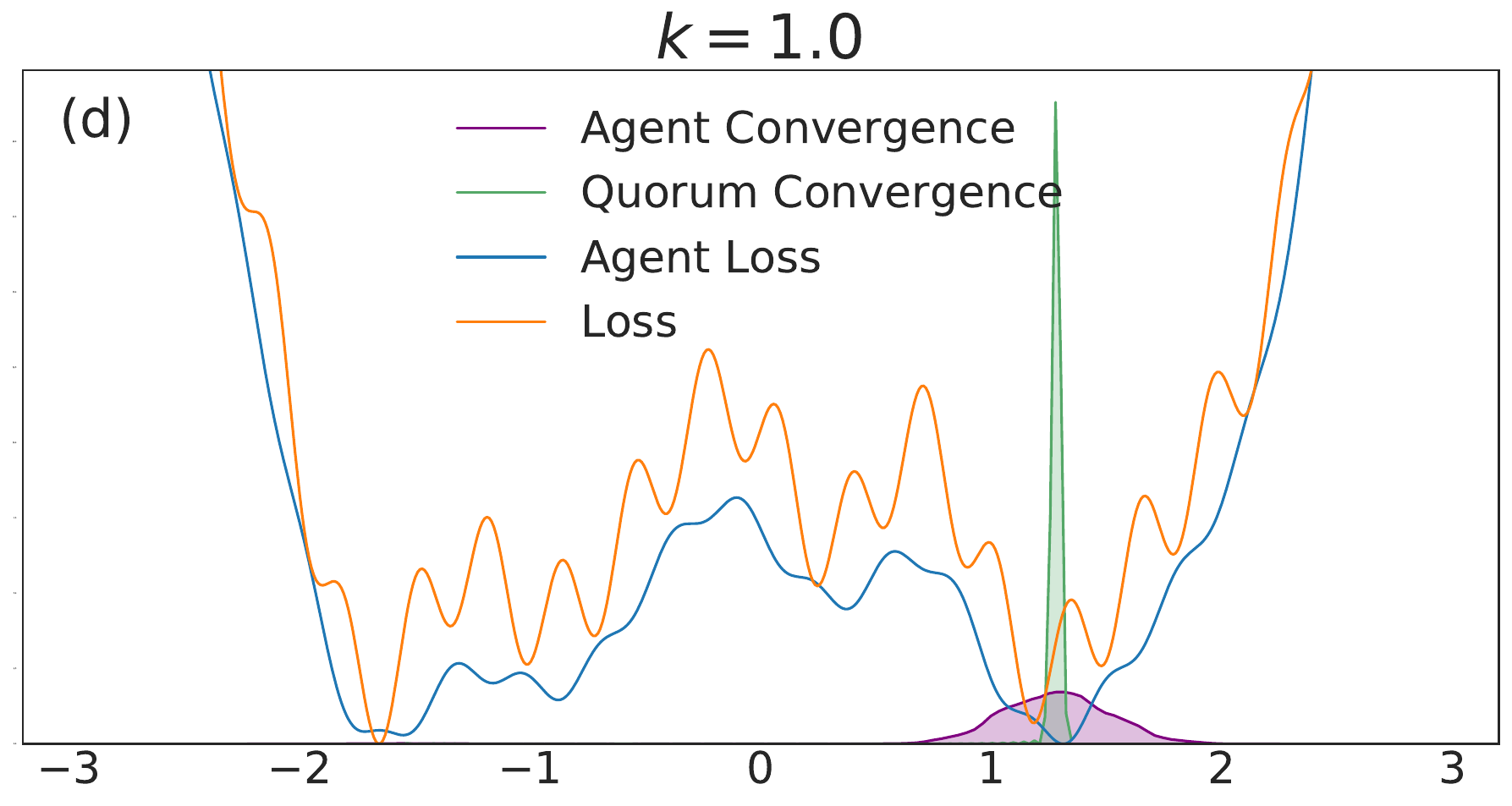}
        \end{subfigure}\\
        
        \begin{subfigure}{.5\textwidth}
            \centering
            \includegraphics[width=\textwidth]{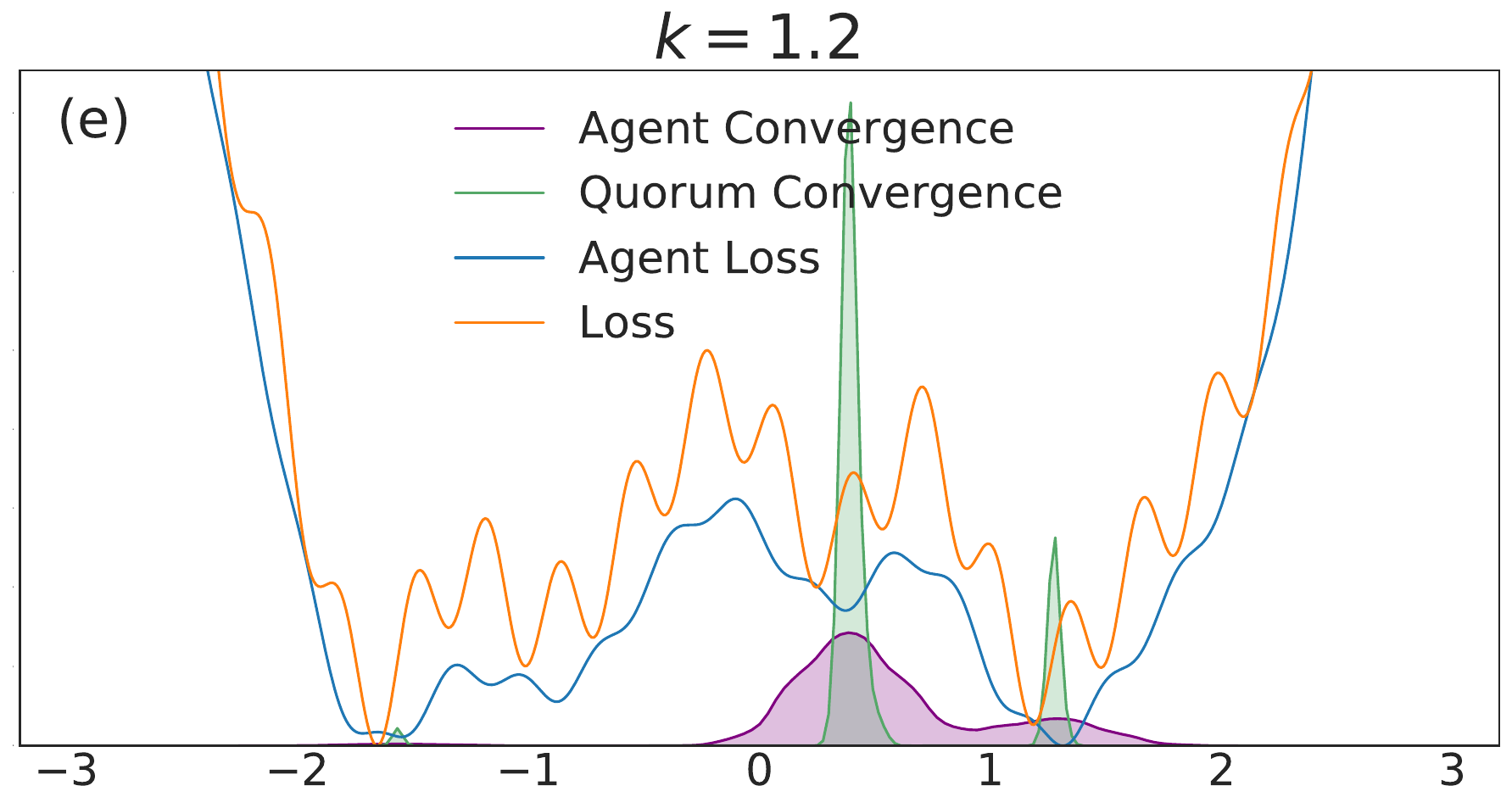}
        \end{subfigure}&
    
        \begin{subfigure}{.5\textwidth}
            \centering
            \includegraphics[width=\textwidth]{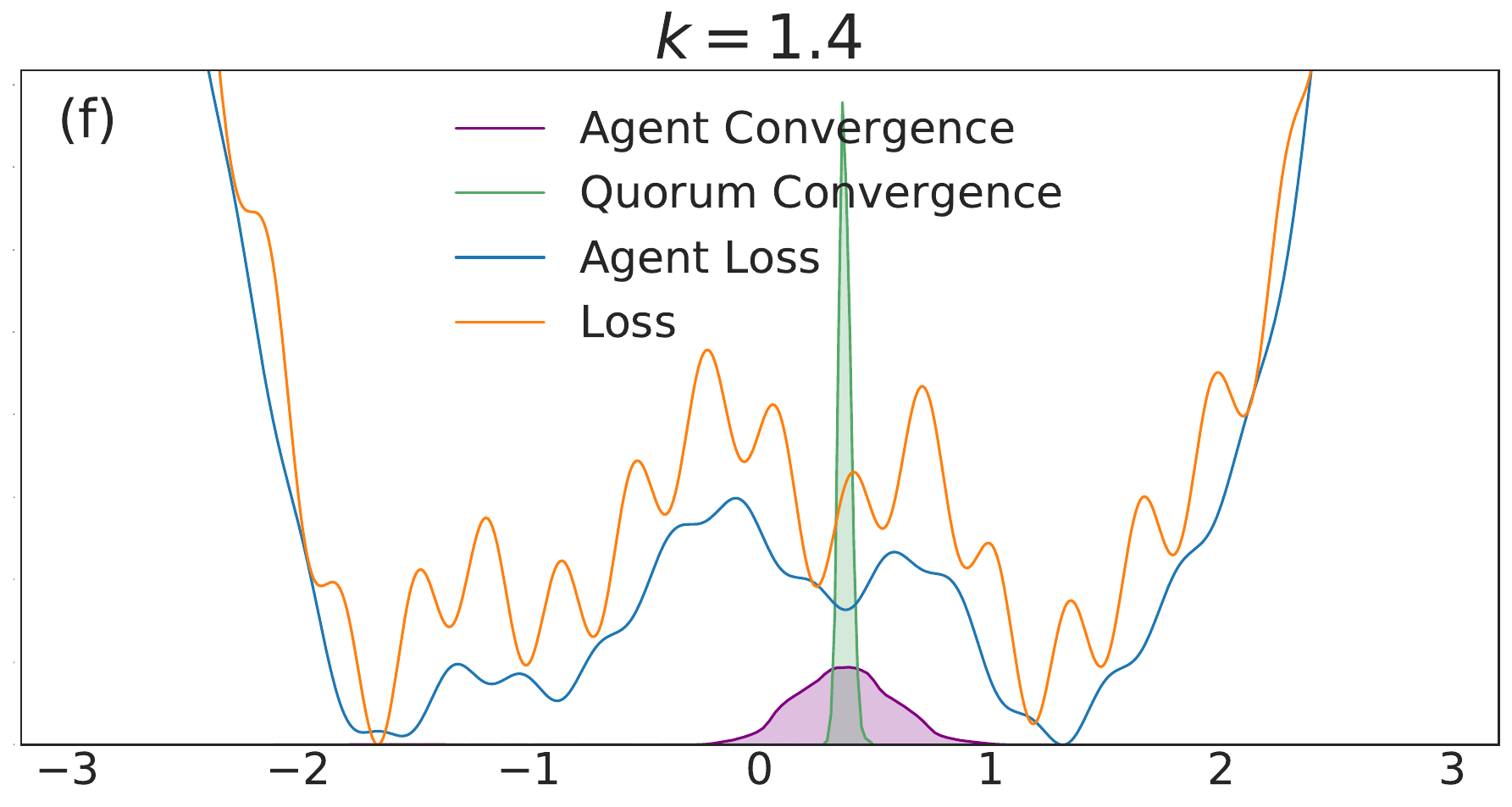}
        \end{subfigure}
    \end{tabular}
    \caption{\normalsize A demonstration of the effect of coupling in the high-noise regime. As the gain is increased, the agents transition from uniform convergence across parameter space, to sharply peaked convergence around deep minima of the smoothed loss, to convergence around minima of the smoothed loss near the initialization. The true loss is shown in orange, the smoothed loss is shown in blue, and the distributions of final iterates for the agents and the quorum variable are shown in purple and green respectively. These simulations use a value of $\eta = 0.15$. Each plot contains the final iterates over 250 simulations with 20,000 iterations each and 1000 agents per simulation. Best viewed in color.}
    \label{fig:kleinberg}
    
\end{figure}

In Fig.~\ref{fig:kleinberg}(a), there is no coupling and the distribution of final iterates for the agents is nearly uniform across the parameter space with a slightly increased probability of convergence to the two deepest regions. The distribution of the quorum variable is sharply peaked around zero\footnote{\normalsize Note that without coupling each agent performs basic SGD. Hence, the results in Fig.~\ref{fig:kleinberg}(a) are equivalent to $p\times n$ single-agent SGD simulations, where $n$ is the total number of simulations and $p$ is the number of agents per simulation}. As $k$ increases to $k=0.4$ in Fig.~\ref{fig:kleinberg}(b), the agents concentrate around the wide basins of the convolved loss function and avoid the sharp local minima of the true loss function. The distribution for the quorum variable is similar, but is too wide to imply reliable convergence to a minimum with loss near the global optimum.

As $k$ is increased further to $k=0.8$ in Fig.~\ref{fig:kleinberg}(c) and $k=1.0$ in Fig.~\ref{fig:kleinberg}(d), performance increases significantly. The distribution of the agents is centered around the global optimum of the smoothed loss, and the distribution of the quorum variable is very sharp around the same minimum; this represents the regime in which the agents have chosen a consensus solution. As demonstrated by Fig.~\ref{fig:kleinberg}(a), this improved convergence is not possible with standard SGD. As $k$ is increased again in Figs.~\ref{fig:kleinberg}(e) and (f), the coupling force becomes too great, and performance decreases -- there is no initial exploratory phase to find the deeper regions of the landscape, and convergence is simply near the initialization of $\com$.

\begin{figure}

    \begin{tabular}{cc}
        \begin{subfigure}{.5\textwidth}
            \centering
            \includegraphics[width=\textwidth]{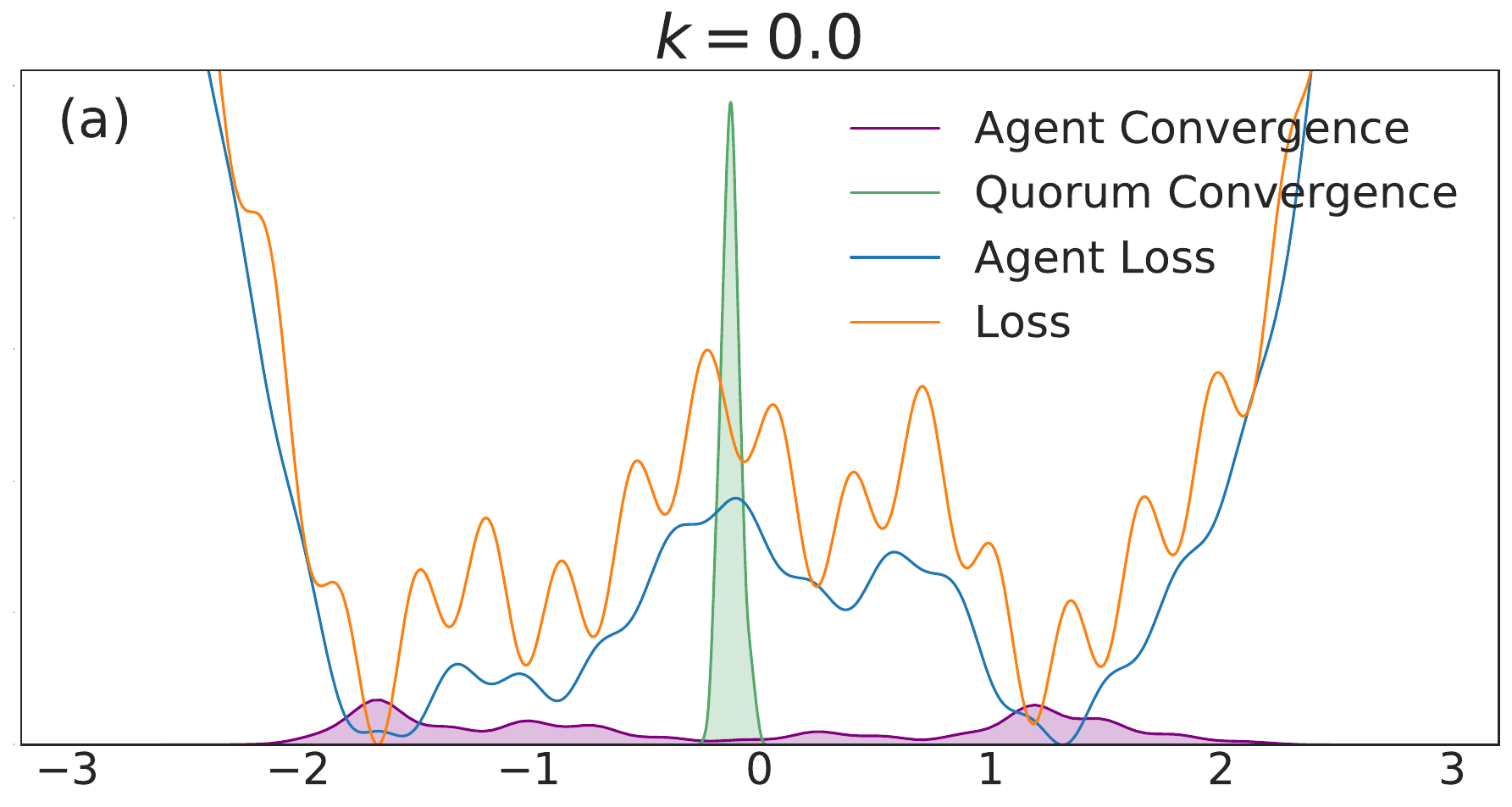}
        \end{subfigure}&
        
        \begin{subfigure}{.5\textwidth}
            \centering
            \includegraphics[width=\textwidth]{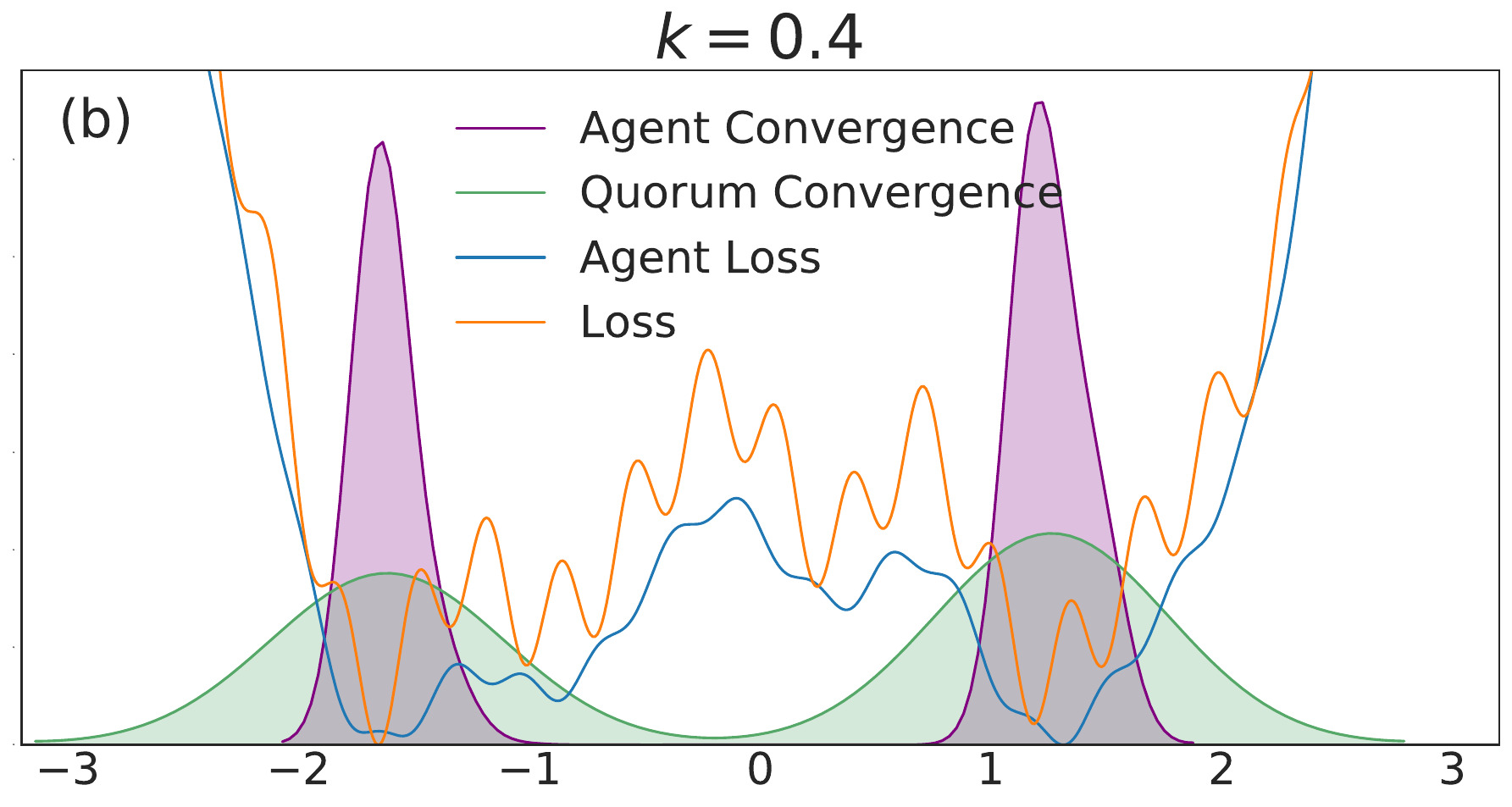}
        \end{subfigure}\\
        
        \begin{subfigure}{.5\textwidth}
            \centering
            \includegraphics[width=\textwidth]{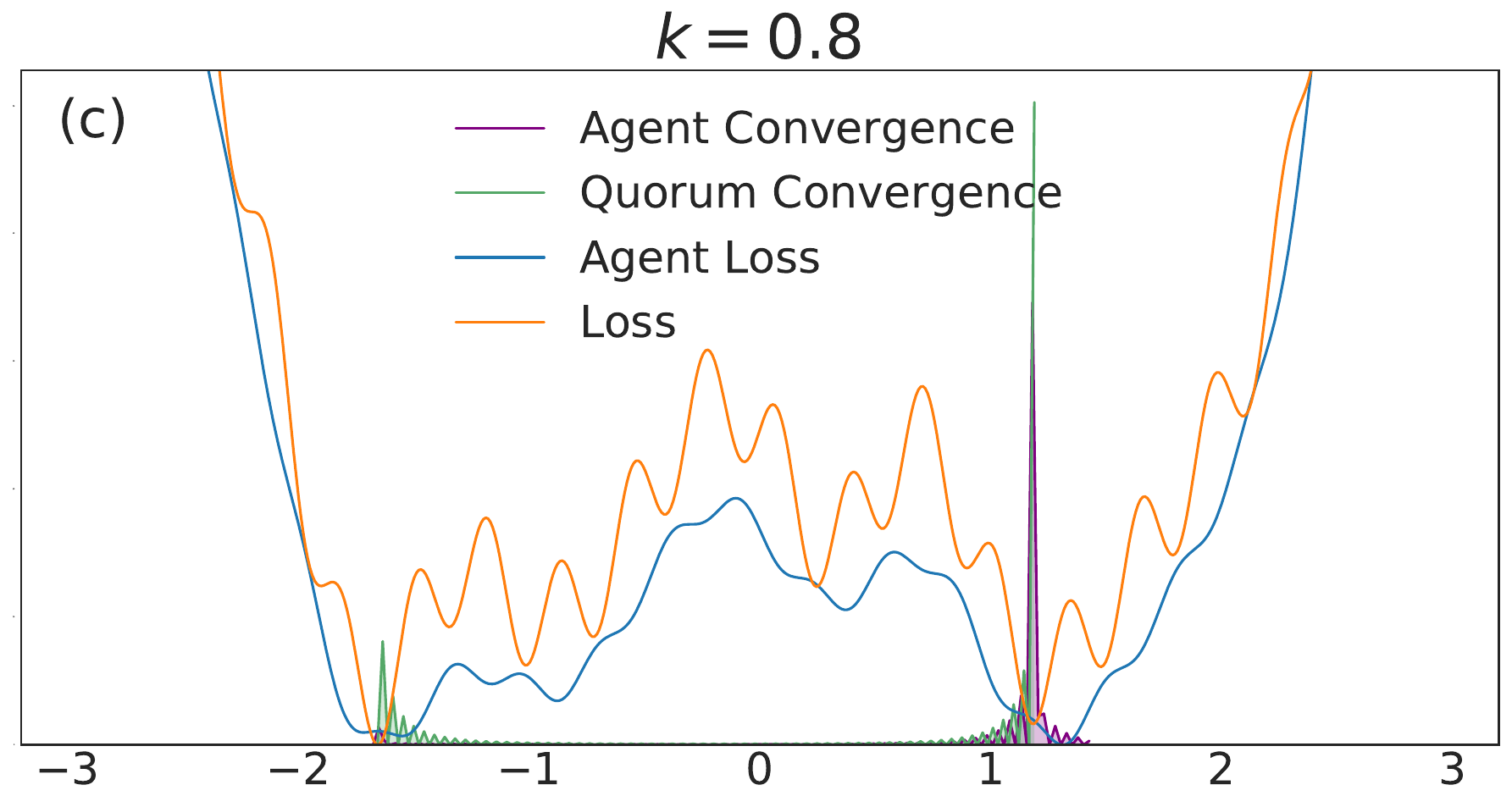}
        \end{subfigure}&
        
        \begin{subfigure}{.5\textwidth}
            \centering
            \includegraphics[width=\textwidth]{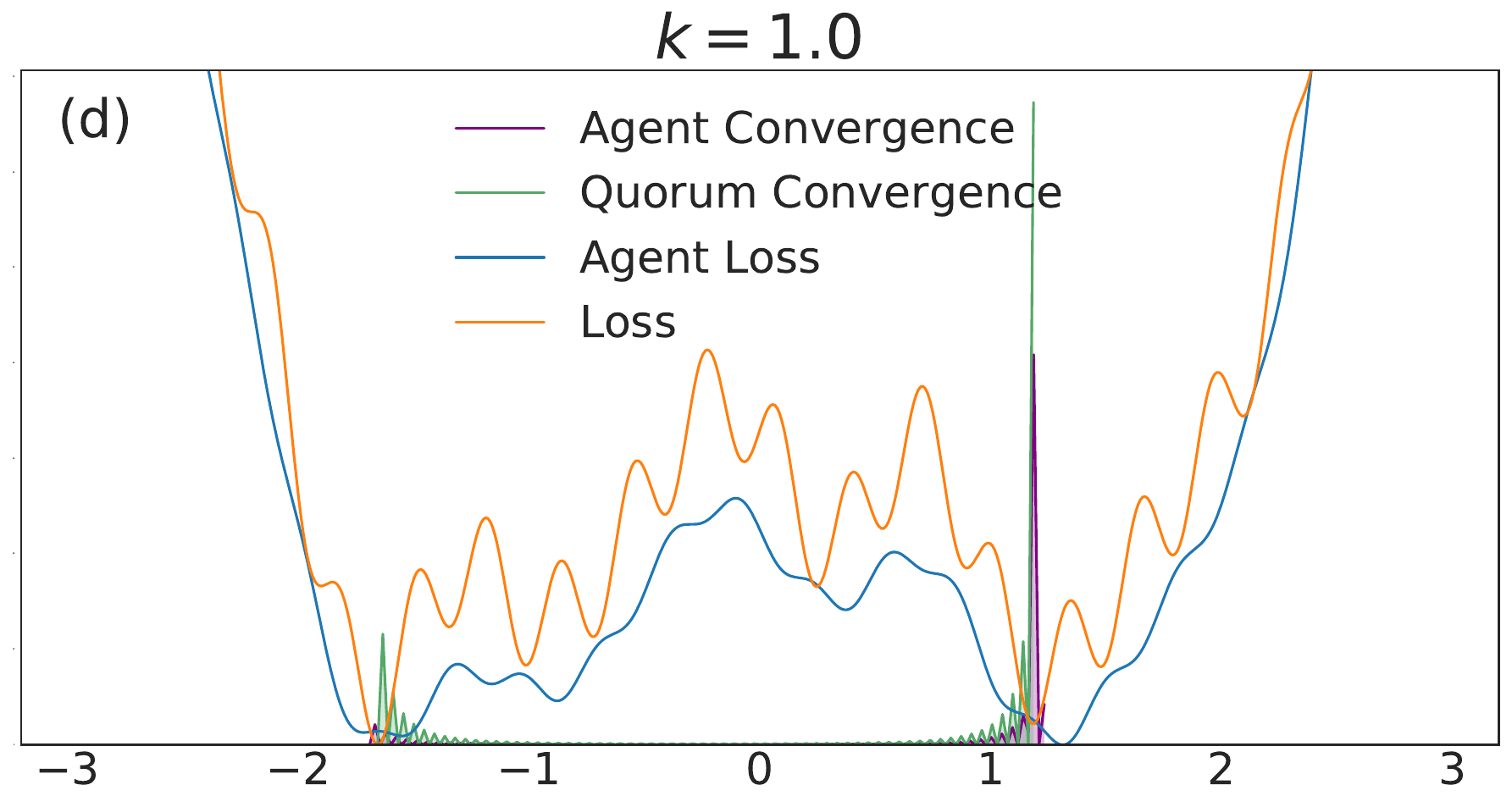}
        \end{subfigure}\\
        
        \begin{subfigure}{.5\textwidth}
            \centering
            \includegraphics[width=\textwidth]{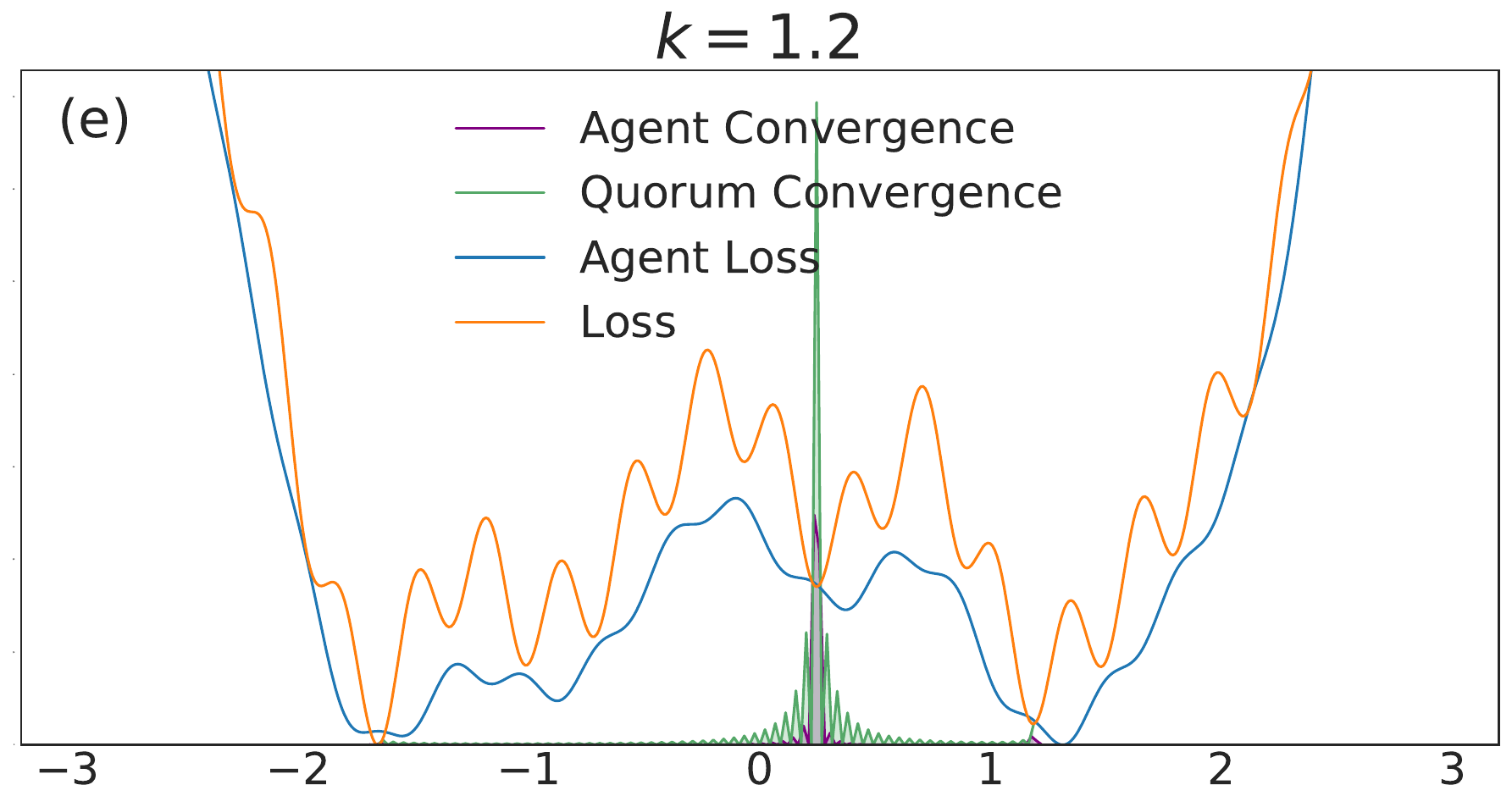}
        \end{subfigure}&
    
        \begin{subfigure}{.5\textwidth}
            \centering
            \includegraphics[width=\textwidth]{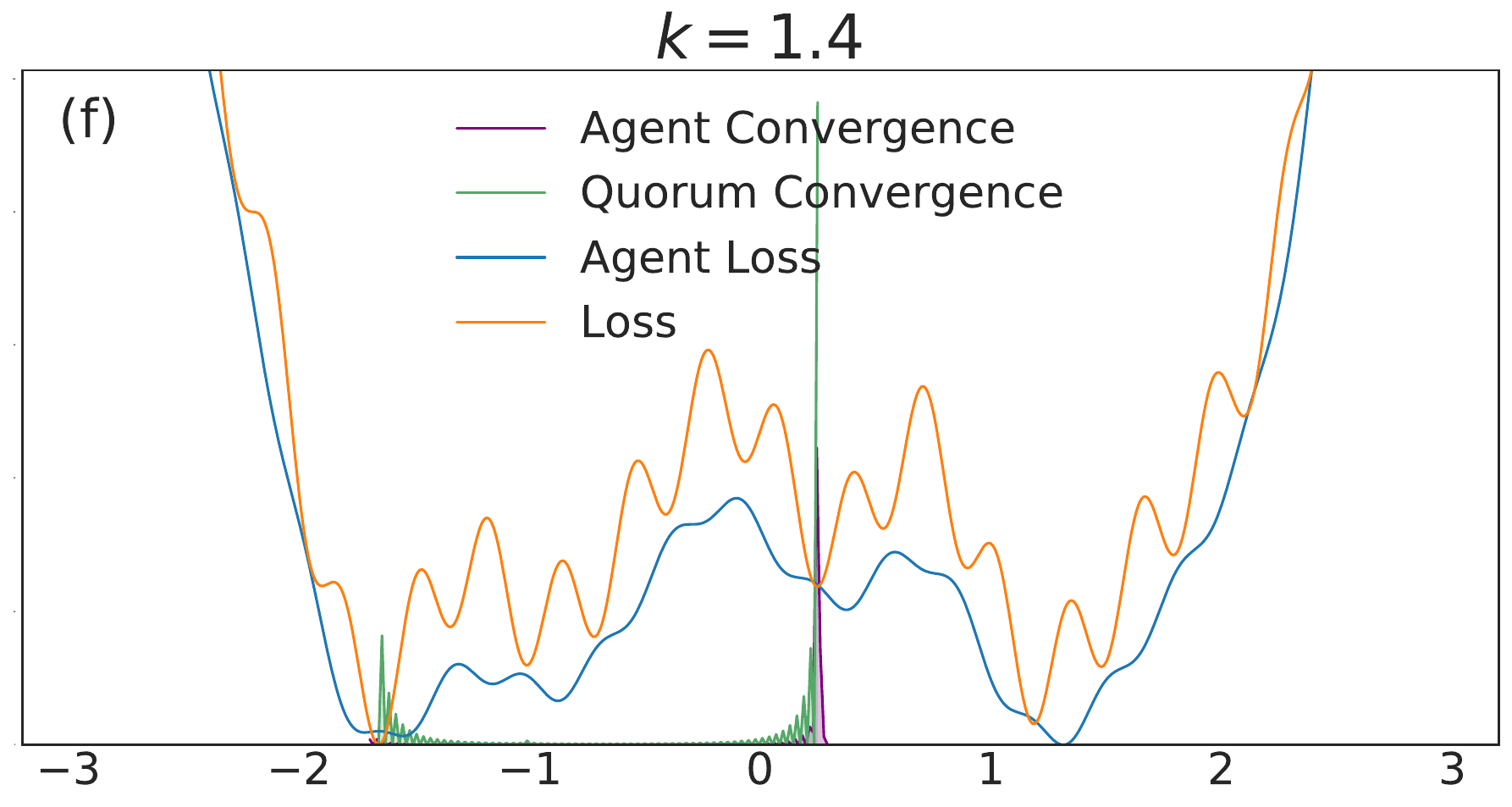}
        \end{subfigure}
    \end{tabular}
    \caption{\normalsize A demonstration of the effect of combining a learning rate schedule with coupling in the high-noise regime. The combination of coupling and learning rate scheduling significantly improves convergence for values of $k$ that concentrate around the global optimum of the smoothed loss in the non-annealed case ($k=0.8$ and $k=1.0$), and the combination leads to sharp peaks around minima of the true loss function. The true loss is shown in orange, the smoothed loss is shown in blue, and the distributions of final iterates for the agents and the quorum variable are shown in purple and green respectively. These simulations use an initial learning rate of $\eta = 0.15$. Each plot contains the final iterates over 250 simulations with 20,000 iterations each and 1000 agents per simulation. Best viewed in color.}
    \label{fig:kleinberg_lrsched}
    
\end{figure}

These simulation results suggest a useful combination of high noise, coupling, and traditional learning rate schedules. High noise levels can lead to rapid exploration and avoidance of problematic regions in parameter space -- such as local minima, saddle points, or flat regions -- while coupling can stabilize the dynamics towards a distribution around a wide and deep minimum of the convolved loss. The learning rate can then be decreased to improve convergence to minima of the true loss that lie within the spread of the distribution. In the uncoupled case, similar levels of noise would lead to a random walk.

This intuition is supported by the simulation results in Fig.~\ref{fig:kleinberg_lrsched}. The same simulation parameters are used, except the learning rate is now decreased by a factor of two every $4000$ iterations until $\eta \leq 0.001$ where it is fixed. In the uncoupled case in Fig.~\ref{fig:kleinberg_lrsched}(a), the schedule only slightly improves convergence around minima of the smoothed loss when compared to Fig.~\ref{fig:kleinberg}(a). Fig.~\ref{fig:kleinberg_lrsched}(b) again reflects a mild improvement relative to Fig.~\ref{fig:kleinberg}(b). For the two best values of $k=0.8$ and $k=1.0$ in Figs.~\ref{fig:kleinberg_lrsched}(c) and (d), convergence of the agents and the quorum variable around the deepest minimum of the true loss that lies within the distribution of the agents in Figs.~\ref{fig:kleinberg}(c) and (d) is excellent. In the very high $k$ regime in Figs.~\ref{fig:kleinberg_lrsched}(e) and (f), the coupling force is too strong to enable exploration, and convergence is again near the initialization of $\comb$, but now to the minima of the true loss.

The preceding results also qualitatively apply to momentum methods. We now turn to simulate the following iteration
\begin{align}
    \label{eqn:mom_disc_1}
    v^i_{t+1} &= \delta v^i_t - \eta \nabla f(x^i_t + \delta v^i_t) - \eta\zeta_t^i,\\
    x^i_{t+1} &= x^i_t + v^i_{t+1} + \eta k \left(\com_t - x^i_t\right),
    \label{eqn:mom_disc_2}
\end{align}
with the loss function again given by (\ref{eqn:double_well}). The distributions of final iterates after $20,000$ steps with $\eta = 0.1$, computed from $250$ simulations per $k$ value with $1000$ agents per simulation, are shown in Fig.~\ref{fig:klein_mom}.

Fig.~\ref{fig:klein_mom}(a) is identical to Fig.~\ref{fig:kleinberg}(a) except for the difference in learning rate: the agents converge uniformly across the parameter space. As $k$ is increased to $k=2$ in Fig.~\ref{fig:klein_mom}(b), the distribution of the agents becomes more localized around the center of parameter space, but not around any minima. When $k$ is increased to $k=4$ in Fig.~\ref{fig:klein_mom}(c), $k=8$ in Fig.~\ref{fig:klein_mom}(d), and $k=10$ in Fig.~\ref{fig:klein_mom}(e), the distributions of the agents and the quorum variable become localized on the two deepest minima of the convolved loss, but are still too wide for reliable convergence. The value $k=15$ in Fig.~\ref{fig:klein_mom}(f) leads to reliable convergence around the deep minimum on the right, and would combine well with a learning rate schedule as in Fig.~\ref{fig:kleinberg_lrsched}. Overall, the trend is similar to the case without momentum, though much higher values of $k$ are tolerated before degradation in performance. Despite high $k$ values rapidly pulling the agent positions close to $\comb(t=0)$, significant differences in the velocities of the agents prevents convergence to a local minimum nearby $\comb(t=0)$ in the high $k$ regime.
\begin{figure}
    \begin{tabular}{cc}
        \begin{subfigure}{.5\textwidth}
            \centering
            \includegraphics[width=\textwidth]{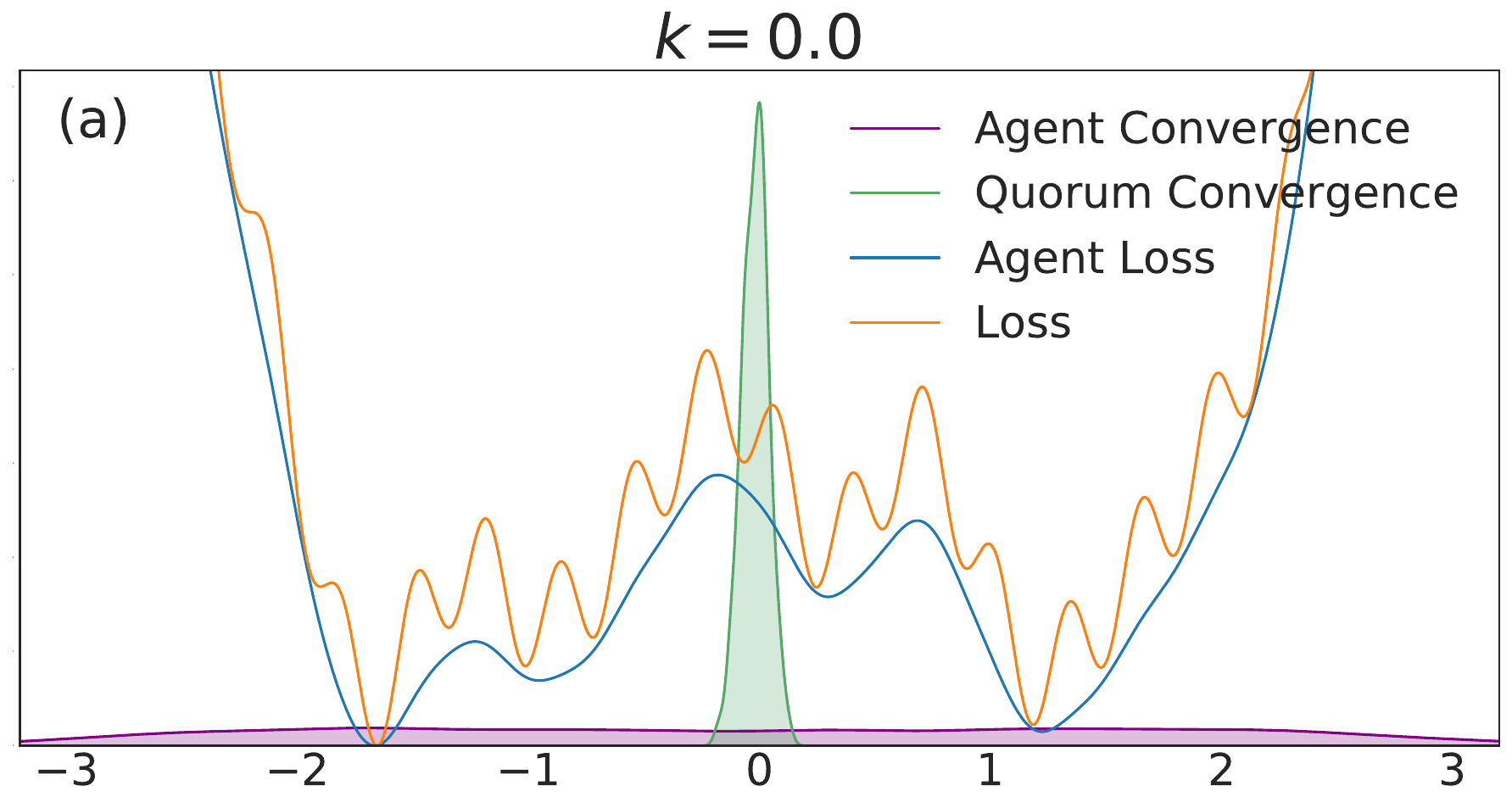}
        \end{subfigure}&
        
        \begin{subfigure}{.5\textwidth}
            \centering
            \includegraphics[width=\textwidth]{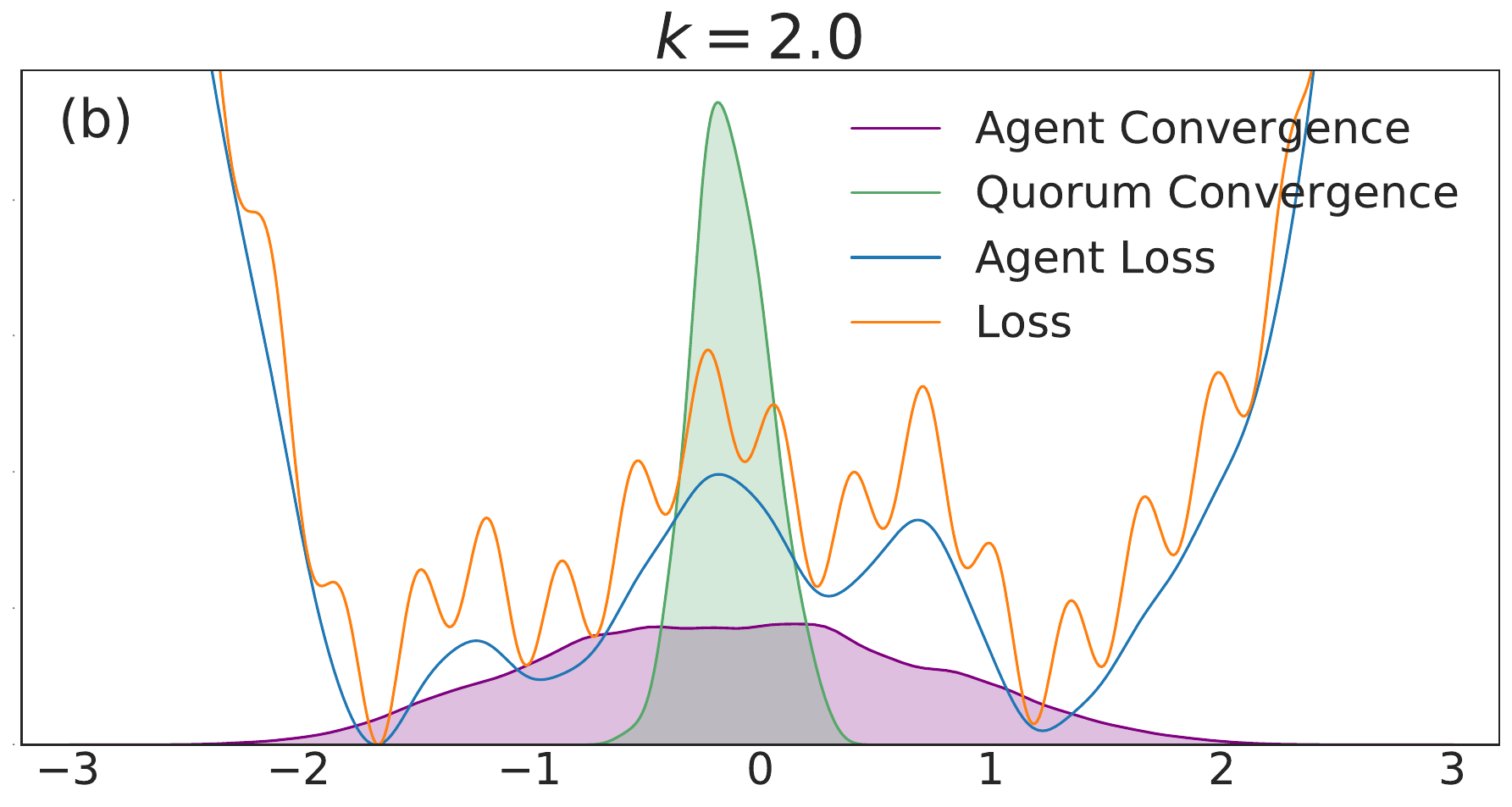}
        \end{subfigure}\\
        
        \begin{subfigure}{.5\textwidth}
            \centering
            \includegraphics[width=\textwidth]{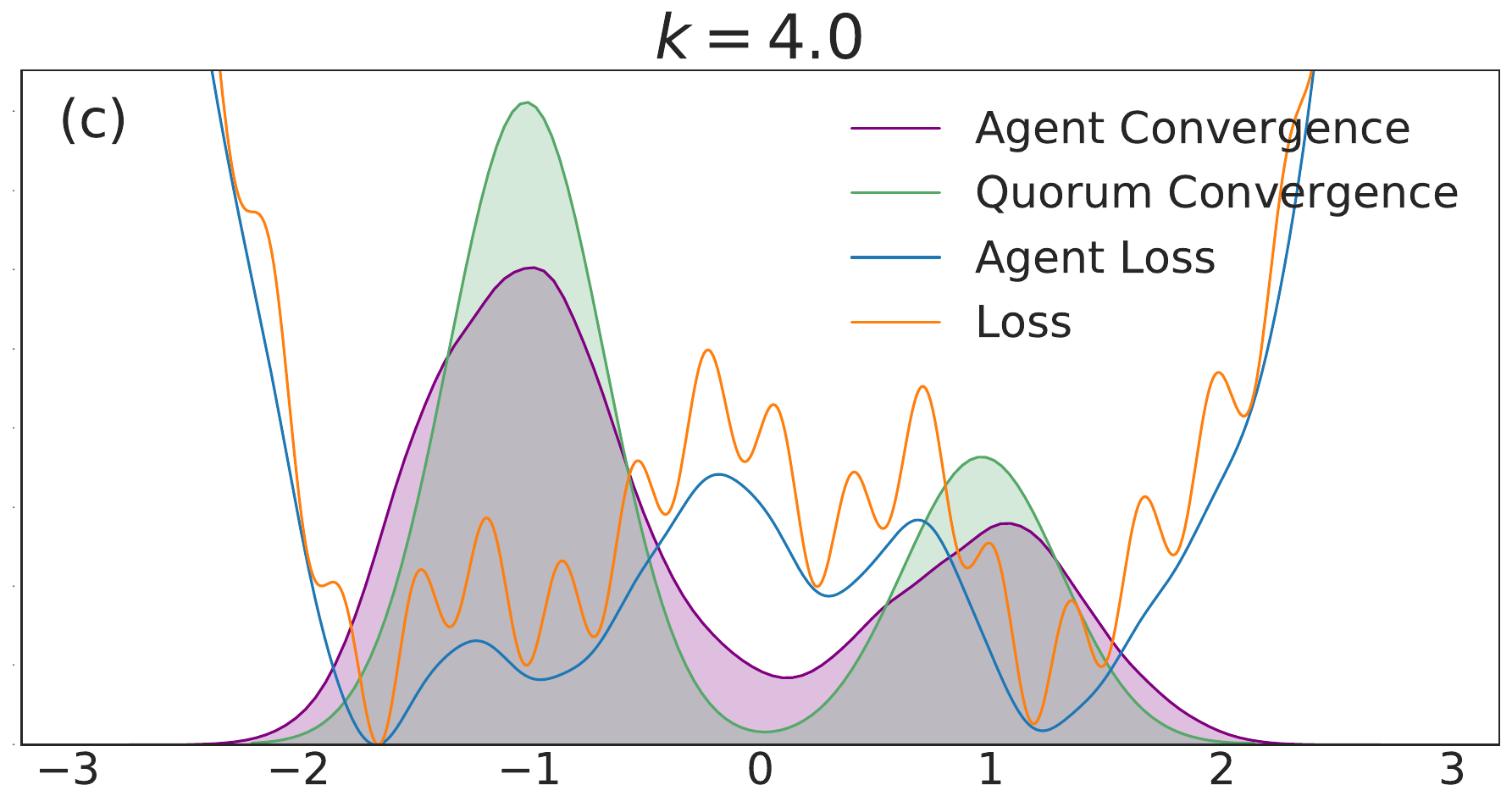}
        \end{subfigure}&
        
        \begin{subfigure}{.5\textwidth}
            \centering
            \includegraphics[width=\textwidth]{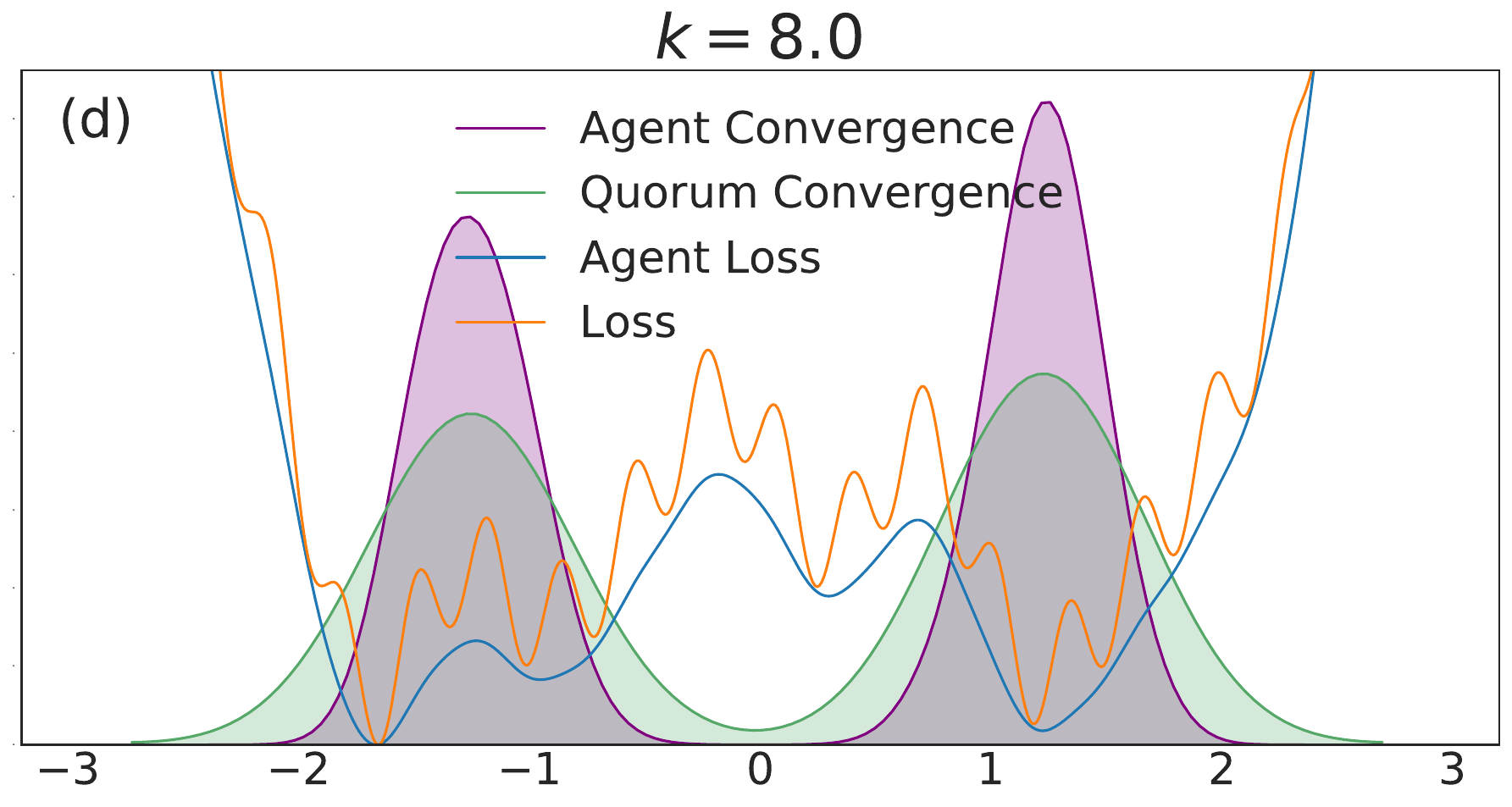}
        \end{subfigure}\\
        
        \begin{subfigure}{.5\textwidth}
            \centering
            \includegraphics[width=\textwidth]{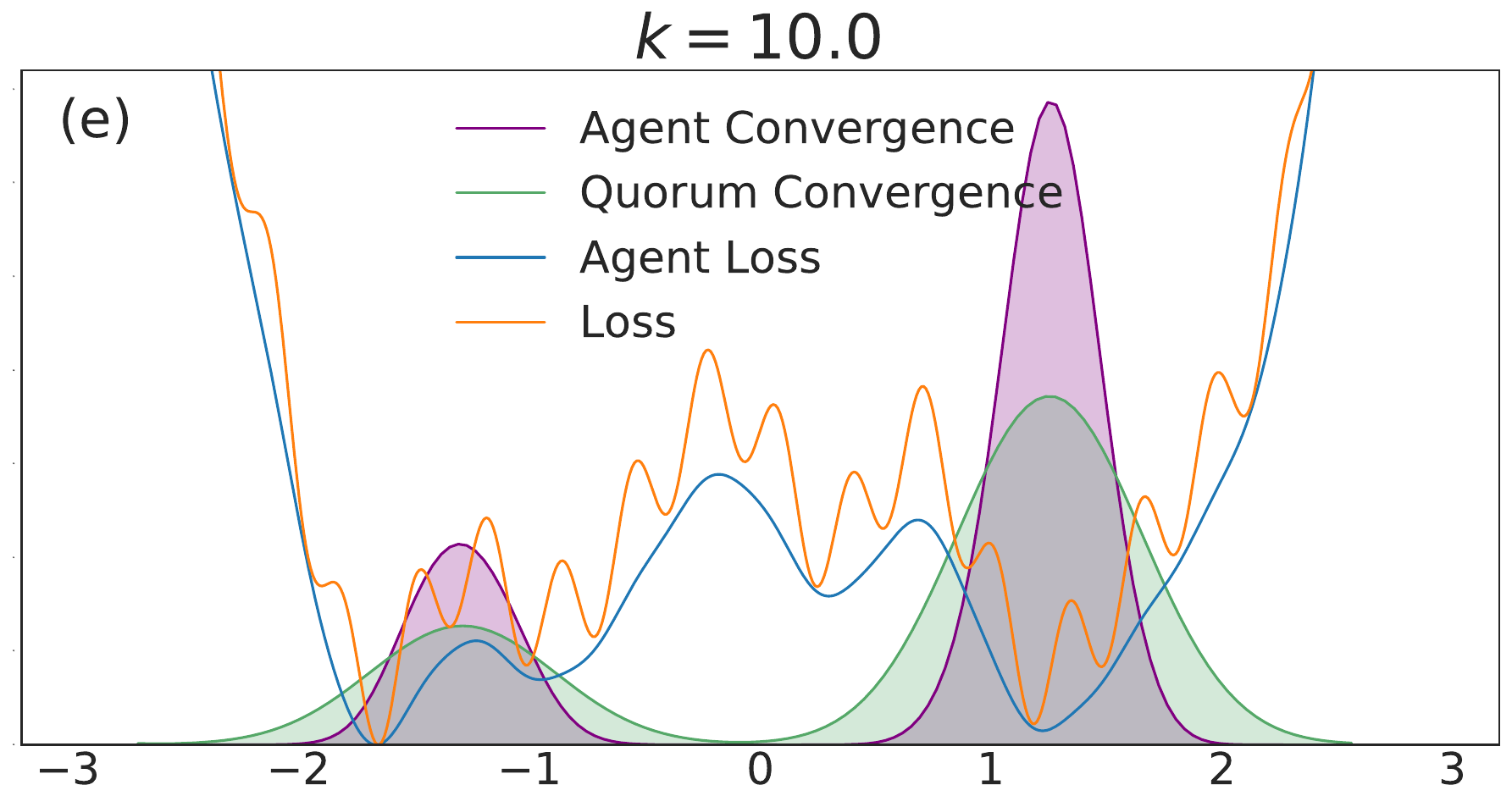}
        \end{subfigure}&
    
        \begin{subfigure}{.5\textwidth}
            \centering
            \includegraphics[width=\textwidth]{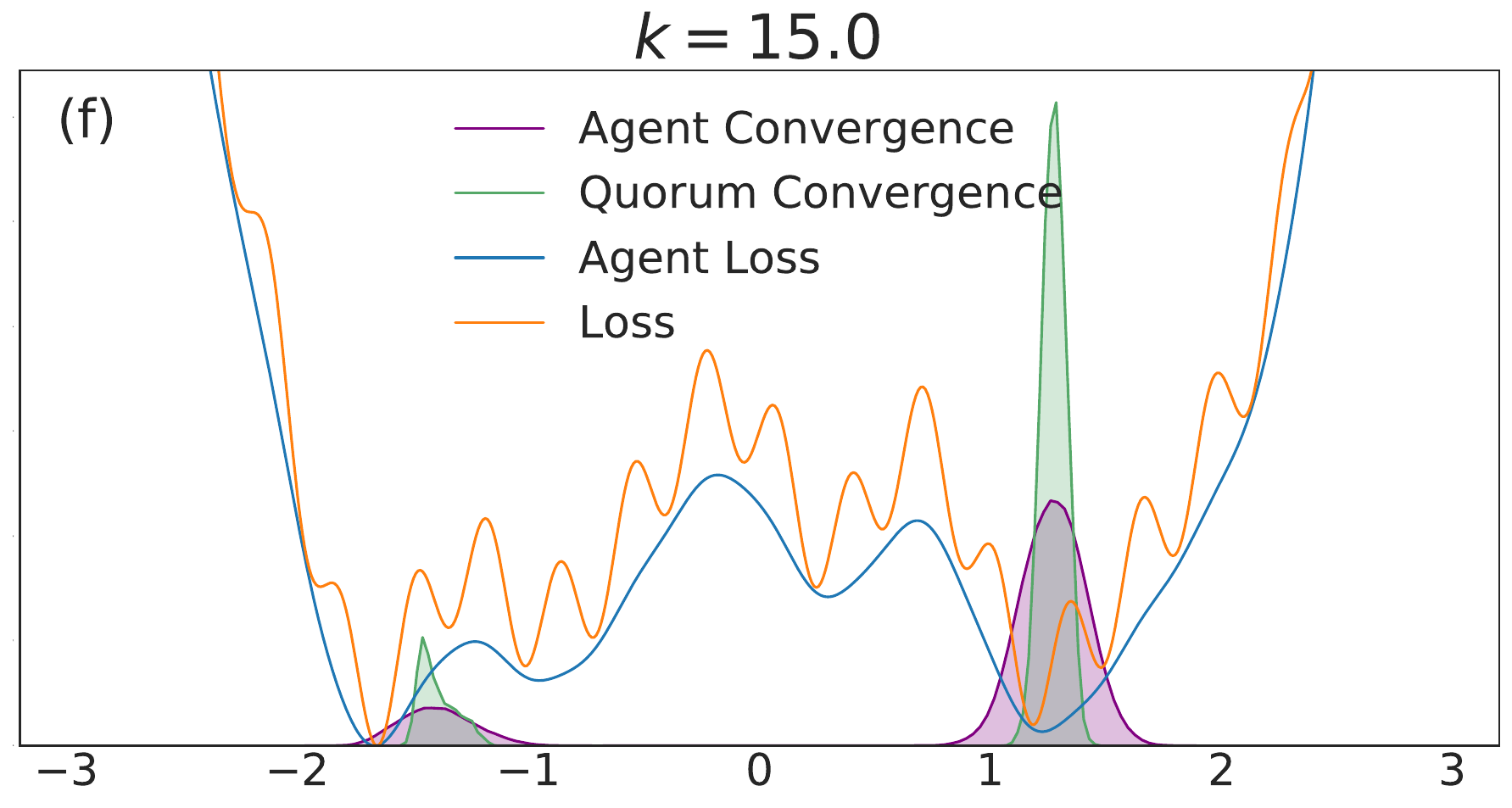}
        \end{subfigure}
    \end{tabular}
    \caption{\normalsize Simulations for the momentum method iteration given by (\ref{eqn:mom_disc_1}) and (\ref{eqn:mom_disc_2}) with $\eta = 0.1$ and $\delta = 0.9$. The true loss is shown in orange, the smoothed loss is shown in blue, and the distributions of final iterates for the agents and the quorum variable are shown in purple and green respectively. The results are qualitatively similar to QSGD without momentum, except that higher $k$ values are tolerated without degradation of performance. Each plot contains the final iterates over 250 simulations with 20,000 iterations each and 1000 agents per simulation. Best viewed in color.}
    \label{fig:klein_mom}
\end{figure}

To demonstrate that these qualitative results also hold in higher dimensions, we now consider a $d$-dimensional objective function inspired by the one-dimensional objective function (\ref{eqn:double_well}). The loss function is given by
\begin{align}
    f(\bx) &= \left(\sum_{i=1}^d x_i^4 - 4 x_i^2 + \frac{1}{5} x_i \right.\nonumber\\
    &+ \frac{2}{5}\left(\left.\sum_{i, j = 1}^d 3\sin(20x_i)\sin(20x_j) + \cos(\frac{10 e}{3}x_i)\cos(\frac{10e}{3}x_j) - \frac{7}{2}\sin(2\pi x_i)\sin(2\pi x_j)\right)\right)/F.
    \label{eqn:nd_loss}
\end{align}
Equation (\ref{eqn:nd_loss}) represents a separable sum of double well loss functions with pairwise sinusoidal coupling between all parameters. We include $1000$ agents in each of $250$ simulations per $k$ value with $d=250$. Each simulation is allowed to run for $10,000$ steps with $1000$ agents per simulation. The parameters are updated according to the vector forms of (\ref{eqn:mom_disc_1}) and (\ref{eqn:mom_disc_2}) with $\eta = .15$ and $\delta = .9$. No learning schedule is used. The agents are all randomly initialized uniformly in $[-4, 4]\times [-4, 4]$ and each experiences an i.i.d. noise term $\zeta^i_t \sim U(-.75, .75)$. $F$ is fixed at $50$.

For visualization purposes, we plot the contours of a two-dimensional cross section of the loss function by evaluating the last $d-2$ coordinates at the value $-1.2$. This value was chosen to represent the bottom-left cluster apparent in Figs.~\ref{fig:klein_agent_nd} and \ref{fig:klein_quorum_nd}; it also lies close to the global minimum of the uncorrupted loss function $(-1.426, -1.426, \hdots, -1.426)^T \in \mathbb{R}^d$. Visualization of high-dimensional loss functions is difficult, and using such a crossz section has its drawbacks; in particular, a saddle point may show up as a local minimum, correctly as a saddle point, or as a local maximum depending on the cross section taken. Nevertheless, the employed cross sections enable     qualitative visualization of the clustering of the quorum variable and the individual agents, and provide assurance that the general phenomena seen in one dimension in Figs.~\ref{fig:kleinberg}-\ref{fig:klein_mom} generalize naturally to higher dimensions.

\begin{figure}
    \centering
    \includegraphics[width=\textwidth]{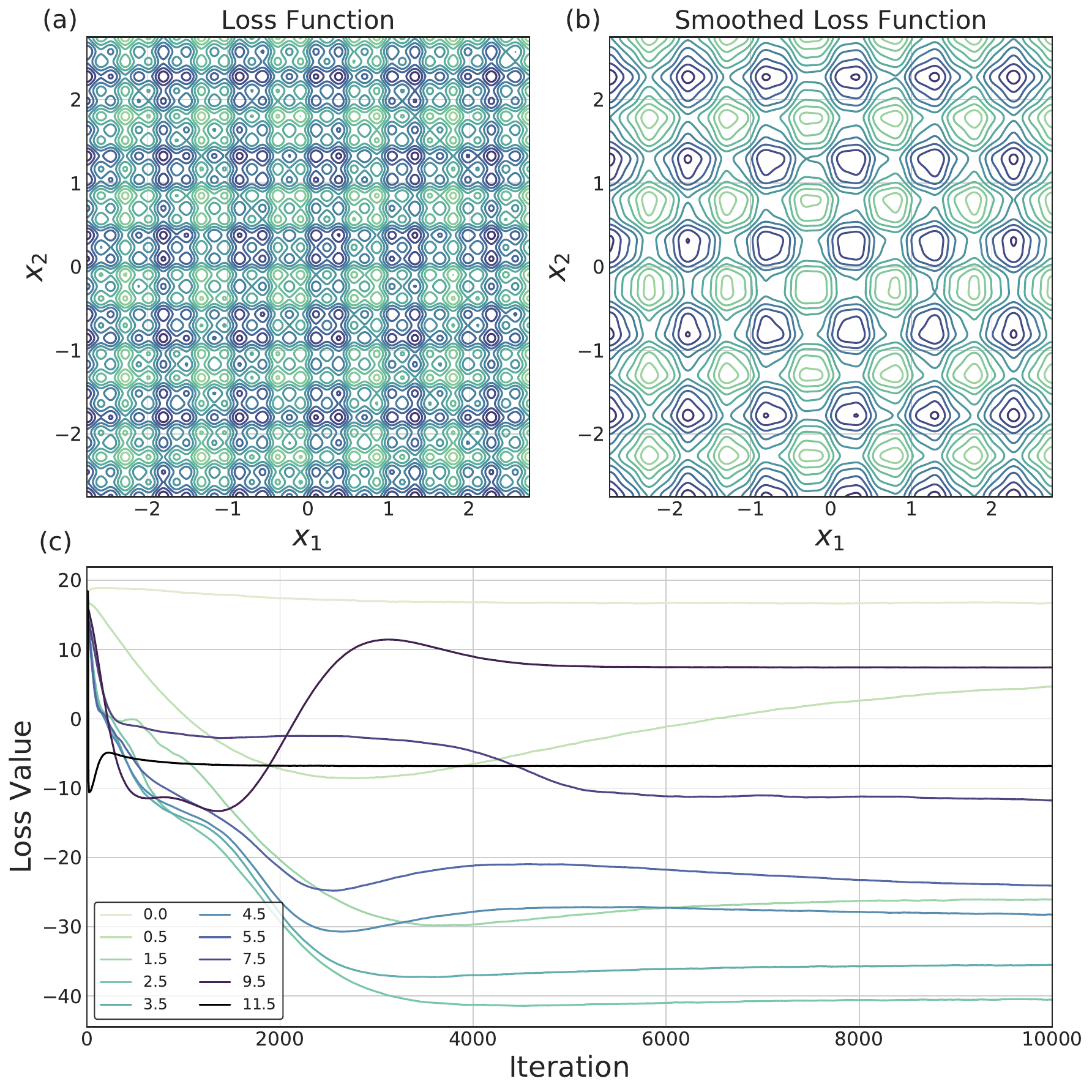}
    \caption{\normalsize (a) A cross section of the loss function (\ref{eqn:nd_loss}), evaluated at $x_i = -1.2$ for $i = 3, \hdots, d$. (b) The same cross section as (a), now of the smoothed loss function given by (\ref{eqn:nd_loss}) convolved with the $\eta$-scaled uniform noise distribution. (c) The loss function value over time for the quorum variable, averaged over all simulations (see text), for a range of $k$ values. The curves demonstrate that there is an optimum value of coupling, in this case around $k = 2.5$, for minimizing the loss function. Best viewed in color.}
    \label{fig:klein_losses_nd}
\end{figure}

The loss function itself is shown in Fig.~\ref{fig:klein_losses_nd}(a) and the smoothed loss is shown in Fig.~\ref{fig:klein_losses_nd}(b), which has significantly reduced complexity. Fig.~\ref{fig:klein_losses_nd}(c) displays the loss value of the quorum variable, averaged over all simulations, as a function of iteration number for a set of possible $k$ values. The results are much the same as was described qualitatively in one dimension. Low values of $k$ such as $k=0$ and $k=0.5$ do not successfully minimize the loss function as the agents are too spread out. Despite a significant ability to explore the loss landscape with such small coupling, the agents are not concentrated enough for $\comb$ to represent a meaningful average. As $k$ increases, the ability to optimize the loss function at first significantly improves. While better than $k=0$ and $k=0.5$, $k=1.5$ still represents the regime of too little coupling. $k=2.5$ and $k=3.5$ obtain much lower loss values than $k=0$ and $k=0.5$, with $k=2.5$ achieving the lowest loss of the displayed $k$ values. As $k$ is increased further, performance starts to degrade. $k=4.5$ performs worse than $k=2.5$, and $k=3.5$ obtains similar performance to $k=1.5$. Increasing $k$ to $k=7.5$, $k=9.5$, and $k=11.5$ continues to deteriorate the ability of the algorithm to minimize the loss. The optimum $k$ value represents, for the given noise level and loss function, the correct balance of exploration and resistance to noise.

\begin{figure}
    \centering
    \includegraphics[width=\textwidth]{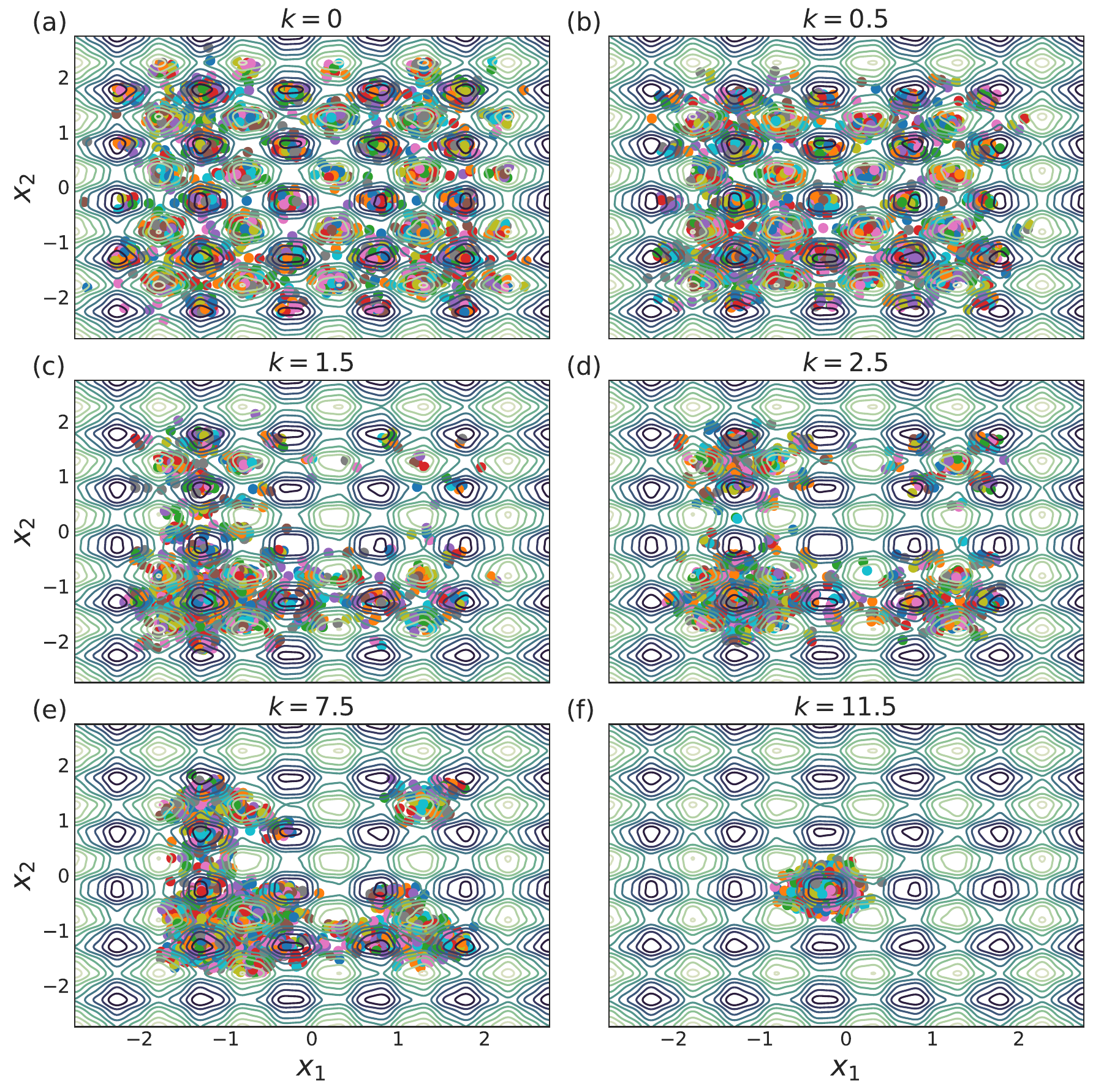}
    \caption{\normalsize Contour plots displaying the location of $25$ agents per simulation (multicolored dots) at the final iteration on top of the smoothed loss. See text for overall simulation setup. Best viewed in color.}
    \label{fig:klein_agent_nd}
\end{figure}

\begin{figure}
    \centering
    \includegraphics[width=\textwidth]{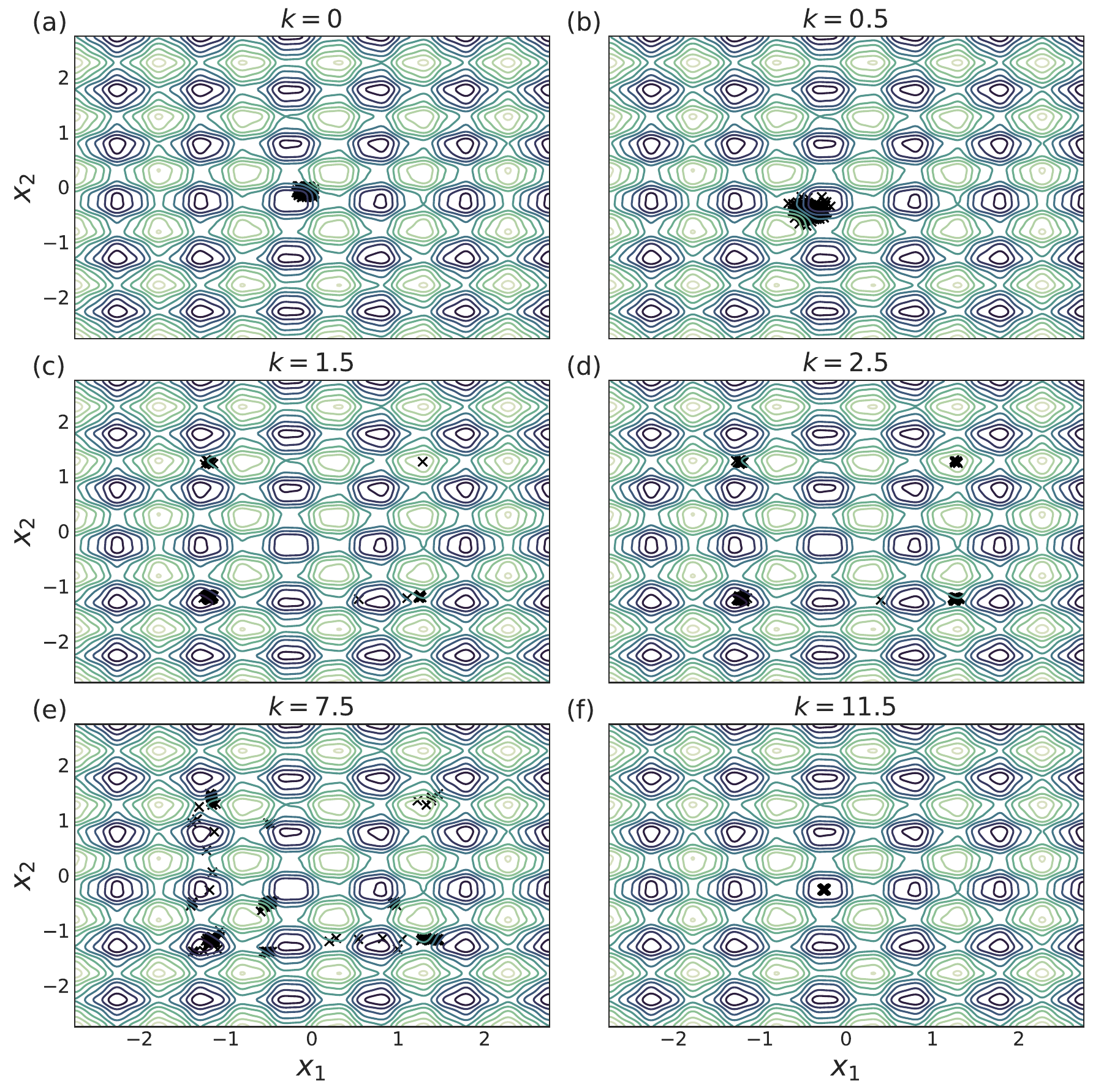}
    \caption{\normalsize Contour plots displaying the location of the quorum variable (black x) in each simulation at the final iteration on top of the smoothed loss. See text for overall simulation setup. Best viewed in color.}
    \label{fig:klein_quorum_nd}
\end{figure}

As in the case of any algorithmic hyperparameter, it is natural to expect that there will be an optimum value of $k$. To see that the manifestation of this optimum is precisely a high-dimensional analogue of the qualitative behavior observed in the one-dimensional simulations in Figs.~\ref{fig:kleinberg}-\ref{fig:klein_mom}, we visualize the final points found by the quorum variable and a random selection of $25$ agents per simulation in Figs.~\ref{fig:klein_agent_nd} and \ref{fig:klein_quorum_nd} respectively for a representative subset of the $k$ values seen in Fig.~\ref{fig:klein_losses_nd}(c).

Fig.~\ref{fig:klein_agent_nd}(a) shows that $k=0$ results in essentially uniform convergence of the agents across parameter space to local minima and saddle points, and hence the quorum variable simply converges near the origin in Fig.~\ref{fig:klein_quorum_nd}(a). The small amount of coupling $k=0.5$ in Fig.~\ref{fig:klein_agent_nd}(b) leads to increased, but still insufficient, clustering of the agents. This manifests itself in Fig.~\ref{fig:klein_quorum_nd}(b) as a shift of the ball of quorum convergence points towards the bottom left corner. $k=1.5$ and $k=2.5$ in Figs.~\ref{fig:klein_agent_nd}(c) and (d) have significantly improved convergence, with strong clustering of the agents in four balls around $(\pm1.2, \pm1.2)^T$. These clusters are located near the minima of the uncorrupted loss function, which occur at $(\pm1.426, \pm1.426, \hdots, \pm1.426)^T$. 

$k=1.5$ and $k=2.5$ have similar quorum convergence plots in Figs.~\ref{fig:klein_quorum_nd}(c) and (d), though the value of the loss in Fig.~\ref{fig:klein_losses_nd}(c) is noticeably different at iteration $10,000$. The difference in the loss function values for the quorum variables are likely hidden by the low-dimensional visualization method. Figs.~\ref{fig:klein_agent_nd}(c) and (d) show that $k=1.5$ has more ``straggler'' agents between the four corner clusters than $k=2.5$, which may shift the quorum convergence points uphill. From a qualitative perspective, both are good choices for tracking minima of the uncorrupted or the non-smoothed loss functions, and could be combined with a learning rate schedule to improve convergence from the cloud of ``starting points'' in Figs.~\ref{fig:klein_agent_nd}(c) and (d).

As $k$ is increased further to $k=7.5$, the coupling begins to grow too strong. The distinct agent clusters attempt to merge, as seen in Fig.~\ref{fig:klein_agent_nd}(e). The result of this is seen in Fig.~\ref{fig:klein_quorum_nd}(e), where there are scattered quorum convergence points between the clusters. Finally, for $k=11.5$, the coupling is too great, and convergence of both the agents and the quorum variables in Figs.~\ref{fig:klein_agent_nd}(f) and \ref{fig:klein_quorum_nd}(f) respectively are both near the origin.

Taken together, Figs.~\ref{fig:kleinberg}-\ref{fig:klein_quorum_nd} provide significant qualitative insight into the convergence of distributed SGD algorithms, both with and without momentum. In one-dimension and in high-dimensional simulations, there is an optimum level of coupling which represents an ideal balance between a) the ability of the agents to explore the loss function, and b) concentration of the distribution of final iterates. Pushing $k$ too high will lead to convergence near the initialization of $\comb$ and ultimately to reduced smoothing of the loss function, while setting $k$ too low will lead to poor convergence of the quorum variable due to a lack of clustering of the agents. Intermediate values of $k$ lead to concentration of the agents around deep and wide minima of the smoothed loss, which will generally lie close to the minima of the uncorrupted loss; convergence can be improved from here with a learning rate schedule.

The optimum value of $k$ is set by the size of the gradients in comparison to the noise level. In the simulation setup used here, this corresponds to a tradeoff between the value of $F$, which sets the gradient magnitudes, and the width of the noise distribution. By setting the width of the noise distribution very high, the optimum $k$ value can be shifted to a large value, so that numerical stability issues arise before performance begins to degrade. Similarly, with small width and small $F$, the optimum value of $k$ can be very small. In Sec.~\ref{sec:dnn}, we will see a manifestation of a similar phenomenon in deep networks for the testing loss.

\section{Convergence analysis}
\label{sec:conv}
We now provide contraction-based convergence proofs for QSGD and EASGD in the strongly convex setting. In the original work on EASGD, rigorous bounds were found for multivariate quadratic objectives in discrete-time, and the analysis for a general strongly convex objective was restricted to an inequality on the iteration for several relevant variances \citep{EASGD}. The results in this section thus extend previously available convergence results for EASGD, and contain new results for QSGD. We furthermore present convergence results for QSGD with momentum.

A significant theme of this section is that the general methodology of Thm.~\ref{thm:sync_noise} can be applied to produce bounds on the expected distance of the quorum variable from the global minimizer of a strongly convex function, again split into a sum of two terms, one based on the averaged noise and one based on bounding the distortion vector $\beps$. We also demonstrate in this section that an optimality result obtained for EASGD in discrete-time in \citet{EASGD} can be obtained through a straightforward application of stochastic calculus in continuous-time, and that the same result applies for QSGD.
\subsection{QSGD convergence analysis}
We first present a simple lemma describing convergence of deterministic distributed gradient descent with arbitrary coupling.
\begin{lem}
    \label{lem:sync}
    Consider the all-to-all coupled system of ordinary differential equations
    \begin{equation}
        \dot{\bx}^i = -\nabla f(\bx^i) + \sum_{j \neq i}\left(\bu\left(\bx^j\right) - \bu\left(\bx^i\right)\right),
        \label{eqn:gen_coup}
    \end{equation}
    with $\bx^i \in \mathbb{R}^n$ for $i = 1, \hdots, p$. Assume that $-\nabla f - p\bu$ is contracting in some metric with rate $\lambda_1$, and that $-\nabla f$ is contracting in some (not necessarily the same) metric with rate $\lambda_2$. Then all $\bx^i$ globally exponentially converge to a critical point of $f$.
\end{lem}
\begin{proof}
    Consider the virtual system
    \begin{equation}
        \dot{\by} = -\nabla f(\by) - p\bu(\by) + \sum_{j=1}^p \bu(\bx^j)\nonumber.
    \end{equation}
    This system is contracting by assumption, and each of the individual agents is a particular solution. The agents therefore globally exponentially synchronize with rate $\lambda_1$. After this exponential transient, the dynamics of each agent is described by the reduced-order virtual system
    \begin{equation}
        \dot{\by} = - \nabla f(\by)\nonumber.
    \end{equation}
    By assumption, this system is contracting in some metric with rate $\lambda_2$, and has a particular solution at any critical point $\bx^*$ such that $\nabla f(\bx^*) = 0$.
\end{proof}
\begin{rmk}
    This simple lemma demonstrates that any form of coupling can be used so long as the quantity $-\nabla f(\by) - p\bu(\by)$ is contracting to guarantee exponential convergence to a critical point. A simple choice is $\bu(\bx^j) = \frac{k}{p}\bx^j$ where $k$ is the coupling gain, corresponding to balanced and equal-strength all-to-all coupling. Then (\ref{eqn:gen_coup}) can be simplified to
    \begin{align}
        \dot{\bx}^i &= -\nabla f(\bx^i) + k \left(\comb - \bx^i\right),
        \label{eqn:com_coup}
    \end{align}
    which is QSGD without noise. Note that all-to-all coupling can thus be implemented with only $2 p$ directed connections by communicating with the center of mass variable.
\end{rmk}
\begin{rmk}
    If $f$ is $l$-strongly convex, $-\nabla f$ will be contracting in the identity metric with rate $l$.
\end{rmk}
\begin{rmk}
    If $f$ is locally $l$-strongly convex, $-\nabla f$ will be locally contracting in the identity metric with rate $l$. For example, for a non-convex objective with initializations $\bx^i(0)$ in a strongly convex region of parameter space, we can conclude exponential convergence to a local minimizer for each agent.
\end{rmk}
If $f$ is strongly convex, the coupling between agents provides no advantage in the deterministic setting, as they would individually contract towards the minimum regardless. For stochastic dynamics, however, coupling can improve convergence. We now demonstrate the ramifications of the results in Sec.~\ref{sec:sync_noise} in the context of QSGD agents with the following theorem.

\begin{thm}
    \label{thm:qsgd}
    Consider the QSGD algorithm
    \begin{equation}
        d\bx^i = \left(-\nabla f(\bx^i) + k(\comb - \bx^i)\right) dt + \sqrt{\frac{\eta}{b}}\bB(\bx^i) d\bW\nonumber,
        \label{eqn:stoch_com_coup}
    \end{equation}
    with $\bx^i \in \mathbb{R}^n$ for $i = 1, \hdots, p$. Assume that the conditions in Assumption \ref{asmpt:hess} hold, that $\bB \bB^T = \bSig$ is bounded such that $Tr(\bSig) \leq C$ uniformly, and that $f$ is $\lambda$-strongly convex. Then, after exponential transients of rate $\lambda$ and $\lambda + k$, the expected difference between the center of mass trajectory $\comb$ and the global minimizer $\bx^*$ of $f$ is given by
    \begin{equation}
        \Exp\left[\Vert \bx^* - \comb \Vert\right] \leq \frac{Q (p-1) C \sqrt{n}\eta}{4 p b\lambda(\lambda + k)} + \sqrt{\frac{\eta C}{2 b p \lambda}}.
        \label{eqn:qsgd_conv}
    \end{equation}
\end{thm}
\begin{proof}
    We first sum the dynamics of the individual agents to compute the dynamics of the center of mass variable. This leads to the SDE
    \begin{equation}
        d\comb = \left(-\nabla f(\comb) + \beps\right)dt + \sqrt{\frac{\eta}{bp^2}}\bT d\bW\nonumber,
    \end{equation}
    with $\beps = \nabla f(\comb) - \frac{1}{p}\sum_i \nabla f(\bx^i)$ and $\bT\bT^T = \sum_i \bSig(\bx^i)$ defined exactly as in Sec.~\ref{sec:sync_noise}. Consider the hierarchy of virtual systems
    \begin{align}
        \dot{\by}^1 &= - \nabla f(\by^1)\nonumber,\\
        \dot{\by}^2 &= - \nabla f(\by^2) + \beps(\bx^1, \hdots, \bx^p)\nonumber,\\
        d\by_t^3 &= \left(-\nabla f(\by^3) + \beps(\bx^1, \hdots, \bx^p)\right)dt + \sqrt{\frac{\eta}{bp^2}}\bT(\bx^1, \hdots, \bx^p) d\bW\nonumber.
    \end{align}
    The $\by^1$ system is contracting by assumption, and admits a particular solution $\by^1 = \bx^*$. As in the proof of Lemma \ref{lem:pert}, we can write with $R = \Vert \by^1 - \by^2 \Vert$,
    \begin{equation}
        \dot{R} + \lambda R \leq \Vert \beps \Vert
    \end{equation}
    which shows that $R$ is bounded. Hence, by dominated convergence,
    \begin{equation}
        \dot{\overline{\mathbb{E}[R]}} + \lambda \mathbb{E}[R] \leq \mathbb{E}[\Vert \beps \Vert]]
    \end{equation}
    As shown in Sec.~\ref{sec:sync_noise}, $\Exp[\Vert \beps \Vert] \leq \frac{Q (p-1) C \eta \sqrt{n}}{4p(\lambda + k) b}$ after exponential transients of rate $\lambda + k$\footnote{\normalsize In Sec.~\ref{sec:sync_noise}, the denominator contained the factor $k - \bal$ rather than $k + \lambda$. Strong convexity of $f$ was not assumed, so that the contraction rate of the coupled system was $k - \bal$. In this proof, strong convexity of $f$ implies that the contraction rate of the coupled system is $k + \lambda$.}. Hence by Lemma \ref{lem:pert}, the difference between the $\by^1$ and $\by^2$ systems can be bounded as
    \begin{equation}
        \Exp[\Vert \by^2 - \bx^* \Vert] \leq \frac{Q (p-1) C \eta \sqrt{n}}{4p(\lambda + k)\lambda b}\nonumber
    \end{equation}
    after exponential transients of rate $\lambda$. The $\by^2$ system is contracting for any input $\beps$, and the $\by^3$ system is identical with the addition of an additive noise term. By Cor.~\ref{cor:stoch_det}, after exponential transients of rate $\lambda$,
    \begin{equation}
        \Exp[\Vert \by^3 - \by^2 \Vert^2] \leq \frac{\eta C}{2 b p \lambda}\nonumber
    \end{equation}
    By Jensen's inequality, and noting that $\sqrt{\cdot}$ is a concave, increasing function,
    \begin{equation}
        \Exp[\Vert \by^3 - \by^2 \Vert] \leq \sqrt{\Exp[\Vert \by^3 - \by^2 \Vert^2]} \leq \sqrt{\frac{\eta C}{2 b p \lambda}}\nonumber
    \end{equation}
    Finally, note that $\comb$ is a particular solution of the $\by^3$ virtual system. From these observations and an application of the triangle inequality, after exponential transients,
    \begin{equation}
        \Exp[\Vert \comb - \bx^* \Vert] \leq \frac{Q (p-1) C \sqrt{n}\eta}{4 pb \lambda(\lambda + k)} + \sqrt{\frac{\eta C}{2 b p \lambda}}\nonumber
    \end{equation}
    This completes the proof.
\end{proof}
As in Sec.~\ref{sec:sync_noise}, the bound (\ref{eqn:qsgd_conv}) consists of two terms. The first term originates from a lack of complete synchronization and can be decreased by increasing $k$. The second term comes from the additive noise, and can be decreased by increasing the number of agents. Both terms can be decreased by decreasing $\frac{\eta}{b}$, as this ratio sets the magnitude of the noise, and hence the size of both the disturbance and the noise term.

State- and time-dependent couplings of the form $k(\comb, t)$ are also immediately applicable with the proof methodology above. For example, increasing $k$ over time can significantly decrease the influence of the first term in (\ref{eqn:qsgd_conv}), leaving only a bound essentially equivalent to linear noise averaging. For non-convex objectives, this suggests choosing low values of $k(\comb, t)$ in the early stages of training for exploration, and larger values near the end of training to reduce the variance of $\comb$ around a minimum. By the synchronization and noise argument in Sec.~\ref{sec:sync_noise} and the considerations in Sec.~\ref{sec:kleinberg}, this will also have the effect of improving convergence to a minimum of the true loss function, rather than the smoothed loss. If accessible, local curvature information could be used to determine when $\comb$ is near a local minimum, and therefore when to increase $k$. Using state- and time-dependent couplings would change the duration of exponential transients, but the result in Thm.~\ref{thm:qsgd} would still hold.

It is worth comparing Eq.~\ref{eqn:stoch_com_coup} to a bound obtained with the same methodology for standard SGD. With the stochastic dynamics
\begin{equation}
    d\bx = -\nabla f(\bx)dt + \sqrt{\frac{\eta}{b}}\bB d\bW,\nonumber
\end{equation}
and the same assumptions as in Thm.~\ref{thm:qsgd}, the expected difference after exponential transients between a critical point of $f$ and the stochastic $\bx$ is given by Cor.~\ref{cor:stoch_det} and an application of Jensen's inequality as
\begin{equation}
    \Exp\left[\Vert \bx - \bx^*\Vert\right]  \leq \sqrt{\frac{\eta C}{2 b \lambda}}.\nonumber
\end{equation}
In the distributed, synchronized case described by Thm.~\ref{thm:qsgd}, the deviation is reduced by a factor of $\frac{1}{\sqrt{p}}$ in exchange for an additional additive term. This additive term is related to the noise strength $\frac{C \eta}{b}$, the bound $Q$, the number of parameters $n$, and is divided by $\lambda(\lambda + k)$ - i.e., is smaller for more strongly convex functions and with more synchronized dynamics.

\subsection{EASGD convergence analysis}
We now incorporate the additional dynamics present in the EASGD algorithm. First, we prove a lemma demonstrating convergence to the global minimum of a strongly convex function in the deterministic setting.

\begin{lem}
    \label{lem:easgd_determ}
    Consider the deterministic continuous-time EASGD algorithm
    \begin{align}
        \dot{\bx}^i &= -\nabla f(\bx^i) + k (\txb - \bx^i)\nonumber,\\
        \dot{\txb} &= k p \left(\comb - \txb\right)\nonumber,
    \end{align}
    with $\bx^i \in \mathbb{R}^n$ for $i = 1, \hdots, p$. Assume $f$ is $\lambda$-strongly convex. Then all agents and the quorum variable $\txb$ globally exponentially converge to the unique global minimum $\bx^*$ with rate
    \begin{equation}
        \gamma \geq \frac{\lambda + k + k p}{2} - \sqrt{\left(\frac{\lambda + k - kp}{2}\right)^2 + k^2 p}.
        \label{eqn:easgd_det_conv_rate}
    \end{equation}
\end{lem}
\begin{proof}
    By Thm.~\ref{thm:sync} and strong convexity of $f$, the individual $\bx^i$ trajectories globally exponentially synchronize with rate $\lambda + k$. On the synchronized subspace, the system can be described by the reduced-order virtual system
    \begin{align}
        \dot{\by} &= - \nabla f(\by) + k (\txb - \by)\nonumber,\\
        \dot{\txb} &= k p \left(\by - \txb\right)\nonumber.
    \end{align}
    The system Jacobian is then given by
    \begin{equation}
        \mathbf{J} = \begin{pmatrix} -\nabla^2 f(\by) - k \bI & k\bI \\ k p\bI & - k p \bI \end{pmatrix}\nonumber.
    \end{equation}
    Choosing a metric transformation $\bThet = \begin{pmatrix} \sqrt{p} \bI & 0 \\ 0 & \bI \end{pmatrix}$, the generalized Jacobian becomes
    \begin{equation}
        \bThet \mathbf{J} \bThet^{-1} = \begin{pmatrix} -\nabla^2 f(\by) - k \bI & \sqrt{p} k\bI \\ \sqrt{p}k\bI & - k p \bI \end{pmatrix}\nonumber,
    \end{equation}
    which is clearly symmetric. A sufficient condition for negative definiteness of this matrix is that $\left(\lambda + k\right)k p > k^2 p$ \citep{Wang2005, HornJohn}. Rearranging leads to the condition $\lambda > 0$, which is satisfied by strong convexity of $f$. The virtual system is therefore contracting. Finally, note that $\by = \txb = \bx^*$ where $\bx^*$ is the unique global minimum is a particular solution. All trajectories thus globally exponentially converge to this minimum. The lower bound on the contraction rate in the statement of the theorem can be found by applying the result in \citet{slot-modular}, Example 3.8.
\end{proof}
Just as in Thm. \ref{thm:qsgd}, we now turn to a convergence analysis for the EASGD algorithm using the results of Lemma \ref{lem:easgd_determ}.

\begin{thm}
    \label{thm:easgd_converge}
    Consider the continuous-time EASGD algorithm
    \begin{align}
        d\bx^i &= \left(-\nabla f\left(\bx^i\right) + k \left(\txb - \bx^i\right)\right)dt + \sqrt{\frac{\eta}{b}}\bB(\bx^i) d \bW^i \nonumber\\
        d\txb &= k p \left(\comb - \txb\right)dt\nonumber
    \end{align}
    for $i = 1, \hdots, p$. Assume that $f$ is $\lambda$-strongly convex and that the conditions in Assumption \ref{asmpt:hess} are satisfied. Let $\gamma$ denote the contraction rate of the deterministic, fully synchronized EASGD system in the metric $\bM = \bThet^T \bThet$ with $\bThet$ the metric transformation from Lemma \ref{lem:easgd_determ}, as lower bounded in (\ref{eqn:easgd_det_conv_rate}). Further assume that $Tr(\bB^T \bM \bB) \leq C(p)$ with $C$ a positive constant potentially dependent on $p$ through the dependence of $\bM$ on $p$. Then, after exponential transients of rate $\gamma$ and $\lambda + k$,
    \begin{equation}
        \Exp\left[\Vert \bz - \bz^*\Vert\right] \leq \frac{Q (p-1) C(p) \sqrt{n}\eta}{4 b\sqrt{p} \gamma (\lambda + k)} + \sqrt{\frac{\eta C(p)}{2 b p \gamma}}
    \end{equation}
    where $\bz = \left(\comb, \txb\right)$ and $\bz^* = \left(\bx^*, \bx^*\right)$.
\end{thm}
\begin{proof}
Adding up the agent dynamics, the center of mass trajectory follows
    \begin{equation}
        d\comb = \left(- \nabla f(\comb) + \beps + k\left(\txb - \comb\right)\right)dt + \sqrt{\frac{\eta}{bp^2}}\bT d\bW\nonumber,
    \end{equation}
    with the usual definitions of $\beps$ and $\bT$. Consider the hierarchy of virtual systems,
    \begin{align*}
        \dot{\by}^1  &= - \nabla f(\by^1) + k \left(\tyb^1 - \by^1\right)\nonumber,\\
        \dot{\tyb}^1 &= k p \left(\by^1 - \tyb^1\right)\nonumber,\\\\
        \dot{\by}^2  &= - \nabla f(\by^2) + k \left(\tyb^2 - \by^2\right) + \beps(\bx^1, \hdots, \bx^p)\nonumber,\\
        \dot{\tyb}^2 &= k p \left(\by^2 - \tyb^2\right)\nonumber,\\\\
        d\by^3  &= \left(- \nabla f(\by^3) + k \left(\tyb^3 - \by^3\right)\right)dt + \beps(\bx^1, \hdots, \bx^p) + \sqrt{\frac{\eta}{bp^2}}\bT(\bx^1, \hdots, \bx^p) d\bW\nonumber,\\
        d\tyb^3 &= k p \left(\by^3 - \tyb^3\right)dt\nonumber.
    \end{align*}
    The first system is contracting towards the unique global minimum with rate $\gamma$ by the assumptions of the theorem and Lemma \ref{lem:easgd_determ}. The second system is contracting for any external input $\beps$, and we have already bounded $\Exp[\Vert \beps \Vert]$ in Sec.~\ref{sec:sync_noise} (the application of the bound is independent of the dynamics of the quorum variable - see App.~\ref{ssec:sync_noise_analysis_EASGD} for details). Let $\bz^i = \begin{pmatrix}\by^i, \tyb^i \end{pmatrix}$ and $\bz^* = \begin{pmatrix} \bx^*, \bx^* \end{pmatrix}$. By an identical argument as in the proof of Thm.~\ref{thm:qsgd} and noting that the condition number of $\bThet$ is $\sqrt{p}$,
    \begin{equation}
        \Exp\left[\left\Vert \bz^2 - \bz^* \right\Vert\right] \leq \frac{\sqrt{p}}{\gamma}\Exp\left[\Vert \beps \Vert\right] \leq \frac{Q(p-1)C(p)\eta\sqrt{n}}{4 \sqrt{p} (\lambda + k)\gamma b}\nonumber
    \end{equation}
    after exponential transients of rate $\gamma$ and $\lambda + k$. Note that $\lambda_{\text{min}}(\bM) = 1$. Hence we can take $\beta = 1$ in Cor.~\ref{cor:stoch_det}, and
    \begin{equation}
        \Exp\left[\Vert \bz^3 - \bz^2\Vert\right] \leq \sqrt{\frac{\eta C(p)}{2b p \gamma}}\nonumber
    \end{equation}
    after exponential transients of rate $\gamma$. Combining these results via the triangle inequality and noting that $\comb, \txb$ is a solution to the $\by^3, \tyb^3$ virtual system, we find that after exponential transients of rate $\gamma$,
    \begin{equation}
        \Exp\left[\Vert \bz - \bz^*\Vert\right] \leq \frac{Q (p-1) C(p)\sqrt{n}\eta}{4 b\sqrt{p} \gamma(\lambda + k)} + \sqrt{\frac{\eta C(p)}{2 b p \gamma}}\nonumber,
    \end{equation}
    where $\bz = \begin{pmatrix} \comb, \txb \end{pmatrix}$.
\end{proof}
Thm.~\ref{thm:easgd_converge} demonstrates an explicit bound on the expected deviation of both the center of mass variable $\comb$ and the quorum variable $\txb$ from the global minimizer of a strongly convex function. As in the discussion after Thm.~\ref{thm:qsgd}, the results will still hold with state- and time-dependent couplings of the form $k = k(\txb, t)$, and the same ideas suggested for QSGD based on increasing $k$ over time can be used to eliminate the effect of the first term in the bound.

Thm.~\ref{thm:easgd_converge} is strictly weaker than Thm.~\ref{thm:qsgd}. The metric transformation used adds a factor of $\sqrt{p}$ to the first quantity in the bound, and the assumption $Tr(\bB^T\bM\bB) \leq C(p)$ now depends on $p$ through the factor of $p$ in the top-left block of $\bM$. Indeed, writing the matrix $\bB$ in $n\times n$ block form, $Tr(\bB^T\bM \bB) = C + (p-1)Tr(\bB_{11}^T\bB_{11} + \bB_{12}^T\bB_{12})$ where $C = Tr(\bB^T\bB)$ as in Thm.~\ref{thm:qsgd}. Thus, the dependence of $C(p)$ on $p$ is in general linear. 

Because of this linear dependence on $p$, the first term in the bound scales like $p^{3/2}$, while the second is asymptotically independent of $p$. This is not the case in Thm.~\ref{thm:qsgd}, where the first term is asymptotically independent of $p$, and the second term scales like $\frac{1}{p}$. The unfavorable scaling of the bound in Thm.~\ref{thm:easgd_converge} with $p$ implies that higher values of $p$ do not improve convergence for EASGD as they do for QSGD. These issues can be avoided by reformulating Lemma \ref{lem:easgd_determ} in the Euclidean metric, but this leads to the fairly strong restriction $k < \frac{4 \lambda p}{(p-1)^2}$. 

These observations highlight potential convergence issues for EASGD with large $p$ which are not present with QSGD. In line with these theoretical conclusions, we will empirically find stricter stability conditions on $k$ for EASGD when compared to QSGD for training deep networks in Sec.~\ref{sec:dnn}. Nevertheless, in the context of nonconvex optimization, higher values of $p$ can still lead to improved performance by affording increased parallelization of the problem and exploration of the landscape

Less significantly, unlike in Thm.~\ref{thm:qsgd}, the bound in Thm.~\ref{thm:easgd_converge} is applied to the combined vector $\bz$ rather than the quorum variable $\txb$ itself, and the contraction rate $\gamma$ is used rather than $\lambda$ in the virtual system bounds\footnote{\normalsize The factor of $\lambda + k$ in the first term remains, as this factor originates in the derivation of the bound on $\mathbb{E}\left[\Vert\beps\Vert\right]$, where the synchronization rate is $\lambda + k$.}. Both of these facts weaken the result when compared to Thm.~\ref{thm:qsgd}. $\gamma$ will in general be less than $\lambda$, as exemplified by the lower bound (\ref{eqn:easgd_det_conv_rate}).

\subsection{QSGD with momentum convergence analysis}
We now present a proof of convergence for the QSGD algorithm with momentum. We first prove a lemma demonstrating convergence to the global minimum of a strongly convex, $\bal$-smooth function. We consider the case of coupling only in the position variables; coupling additionally through the momentum variables is similar. We also restrict to the case of constant momentum coefficient for simplicity.
\begin{lem}
  \label{lem:qsgd_mom_det_conv}
  Consider the deterministic continuous-time QSGD with momentum algorithm
  \begin{align*}
    \dot{\bx}^i_1 &= \bx^i_2 + k\left(\comb - \bx^i_1\right),\\
    \dot{\bx}^i_2 &= -\nabla f(\bx^i_1) - \mu \bx_2,
  \end{align*}
  with $\bx^i_j \in \mathbb{R}^n$ for $i = 1, \hdots, p$. Assume that $f$ is $\ubal$-strongly convex and $\bal$-smooth. For $\mu > 2 \sqrt{\ubal + \bal - 2 \sqrt{\ubal\bal}}$ and $k > \frac{1}{4\mu}\max\left((1 - \bal)^2, (1 - \ubal)^2\right)$, all agents globally exponentially converge to the unique minimum with zero velocity, $(\bx^i_1, \bx^i_2) \rightarrow (\bx^*, 0)$ for all $i$. The exponential convergence rate $\kappa$ can be lower bounded as
  \begin{equation}
      \label{eqn:qsgd_mom_contr_bound}
      \kappa \geq \frac{\delta \mu + (1-\delta)\mu}{2} - \sqrt{\left(\frac{\delta \mu - (1-\delta)\mu}{2}\right)^2 + \frac{1}{4}\left(\frac{(\bal-\ubal)^2}{2\left(\ubal+\bal+2(\delta-1)\delta\mu^2\right)}\right)}
  \end{equation}
  with $\delta = \delta(\mu) \in (0, 1)$.
\end{lem}
\begin{proof}
    By Lemma \ref{lem:qsgd_mom_sync} and according to the assumption on $k$, the agents will globally exponentially synchronize with rate $\xi$, where $\xi$ may be lower bounded as in (\ref{eqn:qsgd_mom_sync_bound}). On the synchronized subspace, the overall system can be described by the virtual system
    \begin{align*}
        \dot{\bx}_1 &= \bx_2,\\
        \dot{\bx}_2 &= -\nabla f(\bx_1) - \mu \bx_2,
    \end{align*}
    where the superscript has been omitted and the coupling term vanishes. Note that this system admits the particular solution $(\bx_1, \bx_2) = (\bx^*, 0)$. This system has Jacobian
    \begin{equation*}
        \bJ = \begin{pmatrix} 0 & \bI \\ -\nabla^2f & -\mu \bI \end{pmatrix},
    \end{equation*}
    which is clearly not contracting. Define the metric transformation $\bThet = \begin{pmatrix} a \bI & 0 \\ \delta \mu \bI & \bI \end{pmatrix}$ with $0 < \delta < 1$ and $a \in \mathbb{R}$. The resulting symmetric part of the generalized Jacobian is given by
    \begin{equation*}
        \left(\bThet \bJ \bThet^{-1}\right)_s = \begin{pmatrix} -\delta \mu\bI & \frac{1}{2}\left(a \bI - \frac{1}{a}\nabla^2f - \frac{(\delta - 1)\delta}{a} \mu^2 \bI\right) \\ \frac{1}{2}\left(a \bI - \frac{1}{a}\nabla^2f - \frac{(\delta - 1)\delta}{a} \mu^2 \bI\right) & (\delta-1)\mu \end{pmatrix}.
    \end{equation*}
    For contraction, we require that
    \begin{equation*}
        \delta(1-\delta)\mu^2 > \frac{1}{4}\max\left(\left(a - \frac{\delta(\delta-1)}{a}\mu^2 - \frac{\ubal}{a}\right)^2, \left(a-\frac{\delta(\delta-1)}{a}\mu^2 - \frac{\bal}{a}\right)^2\right).
    \end{equation*}
    Choosing 
    \begin{equation}
        a = \sqrt{\frac{1}{2}\left(\ubal + \bal + 2(\delta-1)\delta\mu^2\right)}
        \label{eqn:qsgd_mom_a}
    \end{equation} 
    ensures that the two arguments of the max are equal. For $a$ to be real, we require that $\mu < \sqrt{\frac{\ubal+\bal}{2(1-\delta)\delta}}$. The condition for contraction then reads that
    \begin{equation*}
        4\delta(1-\delta)\mu^2 > \left(\frac{(\bal-\ubal)^2}{2\left(\ubal+\bal+2(\delta-1)\delta\mu^2\right)}\right)
    \end{equation*}
    leading to the condition on $\mu$,
    \begin{equation*}
        \frac{1}{2}\sqrt{\frac{\ubal + \bal - 2 \sqrt{\ubal\bal}}{\delta(1-\delta)}} < \mu < \min\left(\frac{1}{2}\sqrt{\frac{\ubal + \bal + 2 \sqrt{\ubal\bal}}{\delta(1-\delta)}}, \sqrt{\frac{\ubal + \bal}{2(1-\delta)\delta}}\right).
    \end{equation*}
    The lower bound is always real and positive by the arithmetic-geometric mean inequality. There is always a gap between the lower and upper bound, regardless of which argument of the $\min$ is chosen in the upper bound. The lower bound is minimized for $\delta=\frac{1}{2}$, leading to the condition that $\mu > 2 \sqrt{\ubal + \bal - 2\sqrt{\bal\ubal}}$. With $\mu$ satisfying this minimal lower bound, the valid range of $\mu$ can be shifted arbitrarily large by choosing
    \begin{equation*}
        \delta(\mu) = \left(\frac{\sqrt{\mu ^2+ 4\left(2 \sqrt{\bar{\lambda} \underline{\lambda}} - (\bar{\lambda} + \underline{\lambda})\right)} + \mu}{2 \mu } - \alpha\right) \in (0, 1)
    \end{equation*}
    with $\alpha > 0$ an arbitrarily small positive constant, thus eliminating the upper bound. The lower bound on the contraction rate $\kappa$ of the system can be obtained by application of the result in \citet{slot-modular}, Example 3.8.
\end{proof}
Note that in general, so long as $\mu$ is chosen to satisfy the lower bound of the preceding lemma, the QSGD with momentum system will be contracting in \emph{some} metric. The given metric will depend on the value of $\delta(\mu)$, for example chosen as suggested in the proof.

With Lemma \ref{lem:qsgd_mom_det_conv} in hand, we can now state a convergence result for QSGD with momentum.
\begin{thm}
    Consider the continuous-time QSGD with momentum algorithm,
    \begin{align*}
        d\bx_1^i &= \left(\bx_2^i + k(\comb - \bx_1^i)\right)dt,\\
        d\bx_2^i &= \left(-\nabla f(\bx^i_1) - \mu \bx_2\right)dt + \sqrt{\frac{\eta}{b}}\bB(\bx_1^i)d\bW^i,
    \end{align*}
    for $i = 1, \hdots, p$. Assume that the conditions of Lemma \ref{lem:qsgd_mom_det_conv} are satisfied, and that the conditions of Assumption \ref{asmpt:hess} are met. Let $\kappa$ denote the contraction rate of the deterministic fully synchronized QSGD with momentum system as lower bounded in (\ref{eqn:qsgd_mom_contr_bound}) and let $\xi$ denote the synchronization rate of the QSGD with momentum system as lower bounded in (\ref{eqn:qsgd_mom_sync_bound}). Further assume that $Tr(\bB^T \bM \bB)\leq C$ with $C > 0$ where $\bM = \bThet^T \bThet$ and $\bThet$ is the metric transformation from Lemma \ref{lem:qsgd_mom_det_conv}. Let $\psi = \frac{a + \delta^2 \mu^2 - 1 - \sqrt{(a - 1)^2 + 2 (a + 1) \delta^2 \mu^2 + \delta^4 \mu^4}}{2 \delta \mu}$ denote the minimum eigenvalue of $\bM$ with $a$ given by (\ref{eqn:qsgd_mom_a}). Then, after exponential transients of rate $\kappa$ and $\xi$, with $\bz = \begin{pmatrix} \comb_1, \comb_2 \end{pmatrix}$ and $\bz^* = \begin{pmatrix} \bx^*, 0 \end{pmatrix}$
    \begin{equation}
        \Exp\left[\Vert \bz - \bz^* \Vert \right] \leq \frac{Q a (p-1)C \sqrt{n} \eta}{4 b p \kappa \xi} + \sqrt{\frac{\eta C}{2 b p \psi \kappa}}
        \label{eqn:qsgd_conv_bound}
    \end{equation}
\end{thm}
\begin{proof}
    Summing the agent dynamics, the center of mass trajectory follows
    \begin{align*}
        d\comb_1 &= \comb_2 dt\\
        d\comb_2 &= \left(-\nabla f(\comb_1) + \beps - \mu \comb_2\right)dt + \sqrt{\frac{\eta}{bp^2}}\bT(\bx^1, \hdots, \bx^p)d\bW
    \end{align*}
    with the usual definition of $\beps$ and $\bT$. Consider an analogous hierarchy of virtual systems as in Thms.~\ref{thm:qsgd} and \ref{thm:easgd_converge},
    \begin{align*}
        \dot{\by}_1^1 &= \by_2^1,\\
        \dot{\by}_2^1 &= -\nabla f(\by_1^1) - \mu \by_2^1,\\\\
        \dot{\by}_1^2 &= \by_2^2,\\
        \dot{\by}_2^2 &= -\nabla f(\by_1^2) - \mu \by_2^2 + \beps(\bx_1^1, \hdots, \bx_1^p),\\\\
        d\by_1^3 &= \by_2^3 dt,\\
        d\by_2^3 &= \left(-\nabla f(\by_1^3) + \beps(\bx_1^1, \hdots, \bx_1^p) - \mu \by_2^3\right)dt + \sqrt{\frac{\eta}{bp^2}}\bT(\bx^1, \hdots, \bx^p)d\bW.
    \end{align*}
    The first system is contracting towards the global minimum with zero velocity and will arrive after exponential transients of rate $\kappa$ by the assumptions of the theorem and by Lemma \ref{lem:qsgd_mom_det_conv}. The second system is contracting for any external input $\beps$, and as argued in Sec.~\ref{sec:sync_noise}, the bound on $\Exp\left[\Vert \beps \Vert\right]$ can be applied as-is to the momentum system with a suitable replacement of contraction rates. As in Thm.~\ref{thm:qsgd}, and noting that the condition number of $\bThet$ is $a$ as given in (\ref{eqn:qsgd_mom_a}),
    \begin{equation*}
        \Exp\left[\Vert \bz^2 - \bz^*\Vert\right] \leq \frac{(p-1) C \eta \sqrt{n} Q a}{4 p \xi \kappa b}
    \end{equation*}
    after exponential transients of rate $\kappa$ and $\xi$. Similarly, an application of Cor.~\ref{cor:stoch_det} gives
    \begin{equation*}
        \Exp\left[\Vert \bz^3 - \bz^2\Vert\right] \leq \sqrt{\frac{\eta C}{2 b p \psi \kappa}}
    \end{equation*}
    after exponential transients of rate $\kappa$, where we have noted that $\bx^T \bM \bx \geq \psi \Vert \bx \Vert^2$. An application of the triangle inequality leads to the result.
\end{proof}
(\ref{eqn:qsgd_conv_bound}) is similar to the results for EASGD and QSGD. The bound is closer in spirit to the bound for QSGD without momentum, in that the two terms do not have poor dependencies on $p$ as they do for EASGD. However, the statement of the theorem is complicated by the expressions for the contraction rates $\kappa$ and $\xi$, the expression for the minimum eigenvalue of the metric $\psi$, and the expression for $a$ in the metric transformation. Together, these four quantities create a more complicated dependence of the bound on hyperparameters such as $\mu$ and $k$. Nevertheless, the spirit is still the same as Thm.~\ref{thm:qsgd}, in that the first term originates from the $\beps$ disturbance and can be eliminated with synchronization, while the second term originates from the additive noise and can be eliminated by including additional agents.

\subsection{Extensions to other distributed structures}
Similar results can be derived for many other possible distributed structures in an identical manner. We present one general formalism here, involving local state- and time-dependent couplings.
\begin{lem}
The state-dependent all-to-all coupled system
\begin{equation}
    \dot{\bx}^i = - \nabla f (\bx^i) + \sum_j k_j(\bx^j, t) (\bx^j - \bx^i) \ \ \ \ \ \ \ \ \ \ \ \ \ \ \ \ \ \ \ \ \ i = 1, \hdots, p,
     \label{eqn:new_alg}
\end{equation}
will globally exponentially synchronize with rate
\begin{equation}
    \inf_{\bx^1, \hdots, \bx^p, t}\left\{\sum_j k_j(\bx^j, t)\right\} \ - \   \sup_{\by}\left\{\lambda_{\max}\left(-\nabla^2 f(\by)\right)\right\}
    \label{eqn:sd_qsgd_sync}
\end{equation}
whenever this value is positive.
\end{lem}
\begin{proof}
The weighted sum $\sum_j k_j(\bx^j, t) \bx^j$ now plays the role of the quorum variable, so that one has
\begin{equation}
    \dot{\bx}^i = - \nabla f (\bx^i) - \sum_j k_j(\bx^j, t) \ \bx^i + \sum_j k_j(\bx^j, t) \bx^j \ \ \ \ \ \ \ \ \ \ \ \ \ \ \ i = 1, \hdots, p.
    \label{eqn:new_alg_quorum}
\end{equation}
The virtual system
\begin{equation}
    \dot{\by} = -\nabla f(\by) - \sum_j k_j(\bx^j, t) \ \by + \sum_j k_j(\bx^j, t) \bx^j
    \nonumber
\end{equation}
shows that the individual $\bx^i$ trajectories globally exponentially synchronize if the conditions of the theorem are met.
\end{proof}
We note that the condition (\ref{eqn:sd_qsgd_sync}) is independent of the number of agents. With noise, the center of mass of (\ref{eqn:new_alg}) satisfies
\begin{equation}
    d\comb = (-\nabla f(\comb) + \beps)dt + \sqrt{\frac{\eta}{p b}}\bB d\bW,
    \nonumber
\end{equation}
where now $\beps = \nabla f(\comb) - \frac{1}{p}\sum_i \nabla f(\bx^i) + \sum_j k_j(\bx^j, t)\bx^j - \comb \sum_j k_j(\bx^j, t)$. As usual, $\beps \rightarrow 0$ in the fully synchronized state.

Individually state-dependent couplings of the form (\ref{eqn:new_alg}) or its quorum-mediated equivalent (\ref{eqn:new_alg_quorum}) allow for individual gain schedules that depend on local cost values or other local performance measures. This can allow each agent to broadcast its current measure of success and shape the quorum variable accordingly. For example, the classification accuracy on a validation set for each $\bx^i$ could be use to select the current best parameter vectors, and to increase the corresponding $k_i$ values to pull other agents towards them.

\subsection{Specialization to a multivariate quadratic objective}
\label{ssec:quad}
In the original discrete-time analysis of EASGD in \citet{EASGD}, it was proven that iterate averaging \citep{pol-jud} of $\txb$ leads to an optimal variance around the minimum of a quadratic objective. We now derive an identical result in continuous-time for the QSGD algorithm, demonstrating that this optimality is independent of the additional dynamics in the EASGD algorithm.

For a multivariate quadratic $f(\bx) = \bx^T \bA \bx$ with $\bA$ symmetric and positive definite, the stochastic dynamics of each agent can be written
\begin{equation}
    d\bx^i = \left(-\bA \bx^i + k \left(\comb - \bx^i\right)\right)dt + \bB d\bW^i,
    \nonumber
\end{equation}
To make the optimal result more clear, we group the factor of $\sqrt{\frac{\eta}{b}}$ into the definition of $\bB$, unlike in previous sections. We furthermore relax the state-dependence of $\bB$ in this section, and assume it to be a constant matrix; this matches the case handled in \citet{EASGD}. 

The assumption of state-independence can be justified in several ways. Theoretical analyses have demonstrated that the specific form of positive semi-definite $\bB$ does not affect the $\mathcal{O}(\eta)$ weak accuracy of the approximating SDE (\ref{eqn:sgd_cont_typical}) for SGD \citep{Li2018, Feng2018, Hu2017a}, though it does affect the constant\footnote{\normalsize The state-dependent version used earlier in this work has been empirically shown to have a lower constant \citep{Li2018}, and is closer to the $\mathcal{O}(\eta^2)$ approximating SDE, which is why it has been utilized up to this point.}. For relevance to general nonconvex optimization, we can assume that all agents have arrived sufficiently close to a minimum of the loss function that it can be approximately represented as a quadratic, and that the noise covariance is approximately constant \citep{Mandt2016, Mandt2017}. For deep networks, the noise covariance has been empirically shown to align with the Hessian of the loss \citep{Sagun2017, Zhu2018}, with theoretical justification for when this is valid provided in Appendix A of \citet{3fac}. For all agents in an approximately quadratic basin of a local minimum of a deep network, $\bB$ can then be taken to be constant such that $\bB\bB^T = \bA$, where $\bA$ is the approximately state-independent Hessian.

With this assumption, $\comb$ satisfies
\begin{equation}
    d\comb_t = -\bA \comb_t dt + \frac{1}{\sqrt{p}}\bB d\bW.
    \nonumber
\end{equation}
This is a multivariate Ornstein-Uhlenbeck process with solution
\begin{equation}
    \comb(t) = e^{-\bA t}\comb(0) + \frac{1}{\sqrt{p}}\int_0^t e^{-\bA(t-s)}\bB d\bW_s.
    \label{eqn:OU_sol}
\end{equation}
By assumption, $-\bA$ is negative definite, so that the stationary expectation $\tlim \Exp[\comb(t)] = 0$. The stationary variance $\mathbf{V}$ is given by
\begin{equation}
    \bA \mathbf{V} + \mathbf{V} \bA^T = \frac{1}{p}\bSig
    \nonumber
\end{equation}
(see, for example, \citet{Gardiner}, p.107). We now define
\begin{equation}
    \bz(t) = \frac{1}{t}\int_0^t \comb(t')dt',
    \nonumber
\end{equation}
and can immediately state the following lemma.
\begin{lem}
  The averaged variable $\bz(t)$ converges weakly to a normal distribution with mean zero and standard deviation $\frac{1}{p}\bA^{-1}\bSig\bA^{-T}$,
    \begin{equation}
        \lim_{t \rightarrow \infty }\sqrt{t}\left(\bz(t) - \bx^*\right) \rightarrow \mathcal{N}\left(0, \frac{1}{p}\bA^{-1}\bSig\bA^{-T}\right).
        \nonumber
    \end{equation}
    In particular, for the single-variable case with $\bA = h$ and $\bSig = \sigma^2$,  
    \begin{equation}
        \lim_{t \rightarrow \infty }\sqrt{t}\left(z(t) - x^*\right) \rightarrow \mathcal{N}\left(0, \frac{\sigma^2}{p h^2}\right).
        \nonumber
    \end{equation}
    \label{lem:avg}
\end{lem}
\begin{proof}
    From (\ref{eqn:OU_sol}),
    \begin{equation}
        \sqrt{t}\bz(t) = \frac{1}{\sqrt{t}}\left(\bA^{-1}\left(1 - e^{-\bA t}\right)\comb(0)\right) + \frac{1}{\sqrt{t p}}\int_0^t dt' \int_0^{t'} e^{-\bA (t'-s)}\bB d\bW_s.
        \nonumber
    \end{equation}
    The mean of which is asymptotically zero. In computing the variance, only the stochastic integral remains. Interchanging the order of integration,
    \begin{align}
        \int_0^t \int_0^{t'} e^{-\bA (t'-s)}\bB d\bW_sdt' &= \int_0^t \int_s^t e^{-\bA (t'-s)}dt' \bB d\bW_s\nonumber,\\
        &= \bA^{-1}\int_0^t \left(1 - e^{-\bA (t -s)}\right) \bB d\bW_s\nonumber.
    \end{align}
    After an application of Ito's Isometry, the variance is given by
    \begin{equation}
        \mathbb{V}\left[\sqrt{t} \bz(t)\right] = \frac{\bA^{-1}}{t p} \left(\int_0^t \left(\bSig - e^{-\bA (t-s)}\bSig - \bSig e^{-\bA^T (t-s)} + e^{-\bA (t-s)}\bSig e^{-\bA^T (t-s)}\right)ds\right)\bA^{-T}\nonumber.
    \end{equation}
    In the limit, the only nonvanishing quantity after the computation of the integral is the linear term $\bSig t$. Then,
    \begin{equation}
        \tlim \mathbb{V}\left[\sqrt{t} \bz(t)\right] = \frac{1}{p}\bA^{-1} \bSig \bA^{-T}.
        \label{eqn:opt_var}
    \end{equation}
\end{proof}

As in the discrete-time EASGD analysis, (\ref{eqn:opt_var}) is optimal in the sense of achieving the Fisher information lower bound, and is independent of the coupling strength $k$ \citep{EASGD, pol-jud}. The lack of dependence on the coupling $k$ is less surprising in this case, as it is not present in the $\comb$ dynamics. The optimality of this result, together with the comparison of Thms.~\ref{thm:qsgd} and \ref{thm:easgd_converge}, suggests that the extra $\txb$ dynamics may not provide any benefit over coupling simply through the spatial average variable $\comb$ from the perspective of convex optimization. However, in Sec.~\ref{sec:dnn}, we will show through numerical experiments on deep networks that EASGD tends to find networks which generalize better than QSGD. The benefits of EASGD must then go beyond basic optimization, and the extra dynamics may have a regularizing effect.

We can also make a slightly stronger statement about (\ref{eqn:opt_var}), as in \citet{Mandt2017}\footnote{\normalsize A similar continuous-time analysis for the averaging scheme considered here was performed in \citet{Mandt2017} for the non-distributed case; the derivation here is simpler and provides asymptotic results.}. If we precondition the stochastic gradients for each agent by the \emph{same} constant invertible matrix $\bQ$, then the stationary variance remains optimal. To see this, note that we can account for this preconditioning simply by modifying the derivation so that $\bA \rightarrow \bQ \bA$ and $\bB \rightarrow \bQ \bB$. Then,
\begin{align}
    \tlim \mathbb{V}\left[\sqrt{t} \bz(t)\right] &= \frac{1}{p}(\bQ\bA)^{-1} \left(\bQ \bB\right)\left(\bQ \bB\right)^T \left(\bQ\bA\right)^{-T}\nonumber,\\
    &= \frac{1}{p}\bA^{-1}\bQ^{-1} \bQ \bB \bB^T \bQ^T \bQ^{-T}\bA^{-T}\nonumber,\\
    &= \frac{1}{p}\bA^{-1}\bSig \bA^{-T}\nonumber.
\end{align}
If different agents are preconditioned by different matrices $\bQ^i$, this result will not hold. Using adaptive algorithms based on past iterations for each agent such as AdaGrad \citep{adagrad} thus may eliminate the optimality, as each agent would compute a different preconditioner.

\section{Deep network simulations}
\label{sec:dnn}
We now turn to evaluate EASGD, QSGD, and one possible state-dependent variant of QSGD (\ref{eqn:new_alg_quorum}) as learning algorithms for training deep neural networks on the CIFAR-10 dataset. A significant goal of the section is to understand the role of synchronization and noise in training deep neural networks. We also seek to test the extensions proposed throughout the paper -- such as multiple learning rates, synchronization bounds allowing for independent initial conditions of the agents, and state-dependent coupling. 

We obtain two primary results. The first is that less synchronization, when it still leads to reliable convergence of the quorum variable, results in the best generalization capabilities of the learned network. This is similar to the results of the model experiments performed in Sec.~\ref{ssec:klein_sims}, though those experiments revealed this to be true for general optimization rather than generalization. The observation of better generalization with reduced synchronization is in line with the comments of Sec.~\ref{ssec:sync_noise_disc} regarding noise and generalization in deep networks. 

Our second primary result is the observation of an interesting regularizing property of EASGD, even in the single-agent case. Unlike QSGD, with a single agent EASGD does not reduce to standard SGD. We find that EASGD without momentum outperforms SGD with momentum and EASGD with momentum in the non-distributed setting.

\subsection{Experimental setup}
We utilize a three-layer convolutional neural network based on the experiments in \citet{EASGD}; each layer consists of a two-dimensional convolution, a ReLU nonlinearity, $2\times 2$ max-pooling with a stride of two, and BatchNorm \citep{batchnorm} with batch statistics in both training and evaluation. The first convolutional layer has kernel size nine, the second has kernel size five, and the third has kernel size three. All convolutions use a stride of one and zero padding. Following the three convolutional layers there is a single fully-connected layer to which we apply dropout with a probability of $0.5$. The input data is normalized to have mean zero and standard deviation one in each channel in both the training and test sets. Because we are interested in qualitative trends rather than state of the art performance, we do not employ any data augmentation strategies. We use an 80/20 training/validation set split, and we use the cross-entropy loss. The stochastic gradient is computed using mini-batches of size 128. The learning rate is set to $\eta = 0.05$ initially unless otherwise specified. This value was chosen as the highest initial value of $\eta$ that remained stable throughout training for most values of $p$, and the qualitative trends presented here were robust to the choice of learning rate (further simulations demonstrating this robustness are available in the SI). We decrease the learning rate three times when the validation loss stalls\footnote{\normalsize More precisely, we keep track of the validation loss for each agent at a reference point, beginning with the validation loss at the first epoch. If the validation loss at the next epoch changes by greater than $1\%$ of the reference point, the reference loss is set to the newly computed validation loss. If the validation loss changes by less than $1\%$, the reference point is unchanged. When the reference point has been unchanged for five epochs, we decrease the learning rate.}: first by a factor of five, then a factor of two the second and third times. This is done on an agent basis, i.e., the agents are allowed to maintain different learning rates. As we are focused on the behavior of the algorithms, rather than efficiency from the standpoint of a parallel implementation, the agents communicate with the quorum variable after each update.

In all methods, we use a Nesterov-based momentum scheme unless otherwise specified \citep{nest-mom, nesterov-conv} with a momentum parameter $\delta = 0.9$ unless otherwise specified and coupling only in the position variables. For EASGD, this takes the form \citep{EASGD},
\begin{align*}
    \bv^i_{t+1} &= \delta \bv^i_t - \eta_i \bg(\bx^i_t + \delta \bv^i_t),\\
    \bx^i_{t+1} &= \bx^i_t + \bv^i_{t+1} + \eta_i k (\txb_t - \bx^i_t),\\
    \txb_{t+1} &= \txb_t + k \sum_i \eta_i (\bx^i_t - \txb_t),
\end{align*}
where $\bg$ is the stochastic gradient. The equivalent form for QSGD can be obtained by the replacement $\txb_t \rightarrow \comb_t$ and by dropping the dynamics for $\txb$. The update for SD-QSGD is similar,
\begin{align*}
    \bv^i_{t+1} &= \delta \bv^i_t - \eta_i \bg(\bx^i_t + \delta \bv^i_t),\\
    \bx^i_{t+1} &= \bx^i_t + \bv^i_{t+1} + \eta_i \left(\sum_j k_j(\bx^j, t)\bx^j - \bx^i \sum_j k_j(\bx^j, t)\right).
\end{align*}
In SD-QSGD, we use state-dependent gains $k_j = k_j(\bx^j, t)$ inspired by a spiking winner-take-all formalism \citep{slot-spike, deneve}. At the start of each epoch, we find the agent with the current minimum validation loss value. Denoting the index of this agent by $j^*$, we define
\begin{align}
    \label{eqn:sd_spike_1}
    k_{j*} &= \frac{k}{p} + (M p -1)\frac{k}{p} e^{-t/\tau},\\
    k_j &= \frac{k}{p}\left( \frac{e^{t/\tau} - 1}{e^{t_f/\tau}-1} \right) \text{\ \ \ \ for $k_j \neq k_{j^*}$},
    \label{eqn:sd_spike_2}
\end{align}
for $t < t_f$, with $k$, $\tau$, $t_f$ and $M \geq 1$ fixed constants, and where $t$ is reset to zero at the start of each epoch. (\ref{eqn:sd_spike_1}) and (\ref{eqn:sd_spike_2}) shape the quorum variable to be entirely composed of the single best agent instantaneously at the start of an epoch. The constant $M$ is a magnification factor and sets the size of the force all other agents feel in the direction of the best agent. The gains relax exponentially back to the QSGD formalism, which is obtained when $k_j = k/p$ for all $j$. The constant $\tau$ sets the speed of relaxation and $t_f$ defines the duration of the spike. At $t = t_f$, all $k_j$ will have relaxed back to the original value $\frac{k}{p}$ for all $j \neq j^*$, and with proper choice of $\tau$, $k_{j^*}$ will be very close. We introduce a small discontinuity measured by the magnitude of $(Mp-1)\frac{k}{p}e^{-t_f/\tau}$ and simply set $k_{j^*} = \frac{k}{p}$ at $t = t_f$. We use a value of $M=10$, choose $t_f = N_b/4$ where $N_b$ is the number of batches in an epoch, and choose $\tau = t_f/16$, corresponding to a rather rapid spike.\footnote{\normalsize Another option would be to set $k_j = (k - k_{j^*})/(p-1)$ when this is positive, and zero otherwise. This ensures, outside of the initial spiking period, that the total sum of the $k_j$ is constant. We found similar empirical results with both choices.}

In each of the following simulations, the fully connected weights and biases are initialized randomly and uniformly $W_{ij},\ b_i \sim \mathcal{U}(-\frac{1}{\sqrt{m}}, \frac{1}{\sqrt{m}})$ where $m$ is the number of inputs. The convolutional weights use Kaiming initialization \citep{kaim_init}. In each comparison, the methods are initialized from the same points in parameter space, but the agents are not required to be initialized at the same location. In QSGD and SD-QSGD, the quorum variable is exponentially weighted $\bar{\mathbf{x}}_{t+1} = \gamma \comb_t + (1-\gamma) \bar{\bx}_t$ with $\gamma = .1$, and we test convergence of $\bar{\bx}$. Note that because this variable is not coupled to the dynamics of the individual agents, this is still distinct from EASGD. Because we use momentum in nearly all experiments, we will refer simply to QSGD and EASGD. The non-momentum variant of EASGD, when used, will be referred to as EASGD-WM (EASGD without momentum).

\subsection{Experimental Results}

We first analyze the effect of $k$ on classification performance. We find that the best performance is obtained for the lowest possible fixed values of $k$ that still lead to convergence of the quorum variable. This is demonstrated in Fig.~\ref{fig:k_vals} for the EASGD algorithm with $\eta = 0.05$ initially and $p=8$, where we observe the general trend that test accuracy improves as the coupling gain is decreased. Note that $k=0.01$ and $k=0.02$, as well as $k=0$ (not shown) have too little synchronization for the quorum variable to reflect a meaningful average, and hence do not lead to good performance. Similar results hold for QSGD (not shown). We found not only the best performance for low, fixed $k$, but also the best scaling with the number of agents\footnote{\normalsize The improvement in test accuracy and in the minimization of the test loss with increasing number of agents is demonstrated in later plots. We found that this trend was maximized with lower values of $k$.}. 

There are several plausible explanations for the observation of improved generalization with reduced coupling. Lower values of $k$ allow for greater exploration of the optimization landscape, which intuitively should lead to better performance. As the measure of synchronization in Fig.~\ref{fig:k_vals}(d) tends to zero, the $\beps$ term in the $\comb$ dynamics will also tend to zero, and synchronization will begin to reduce the amount of noise felt by the individual agents. In neural networks, it is expected that this noise reduction will favor convergence to minima that do not generalize as well as those obtained with higher amounts of noise, as is seen in Fig.~\ref{fig:k_vals}(c).

\begin{figure}
    \centering
    \includegraphics[width=\textwidth]{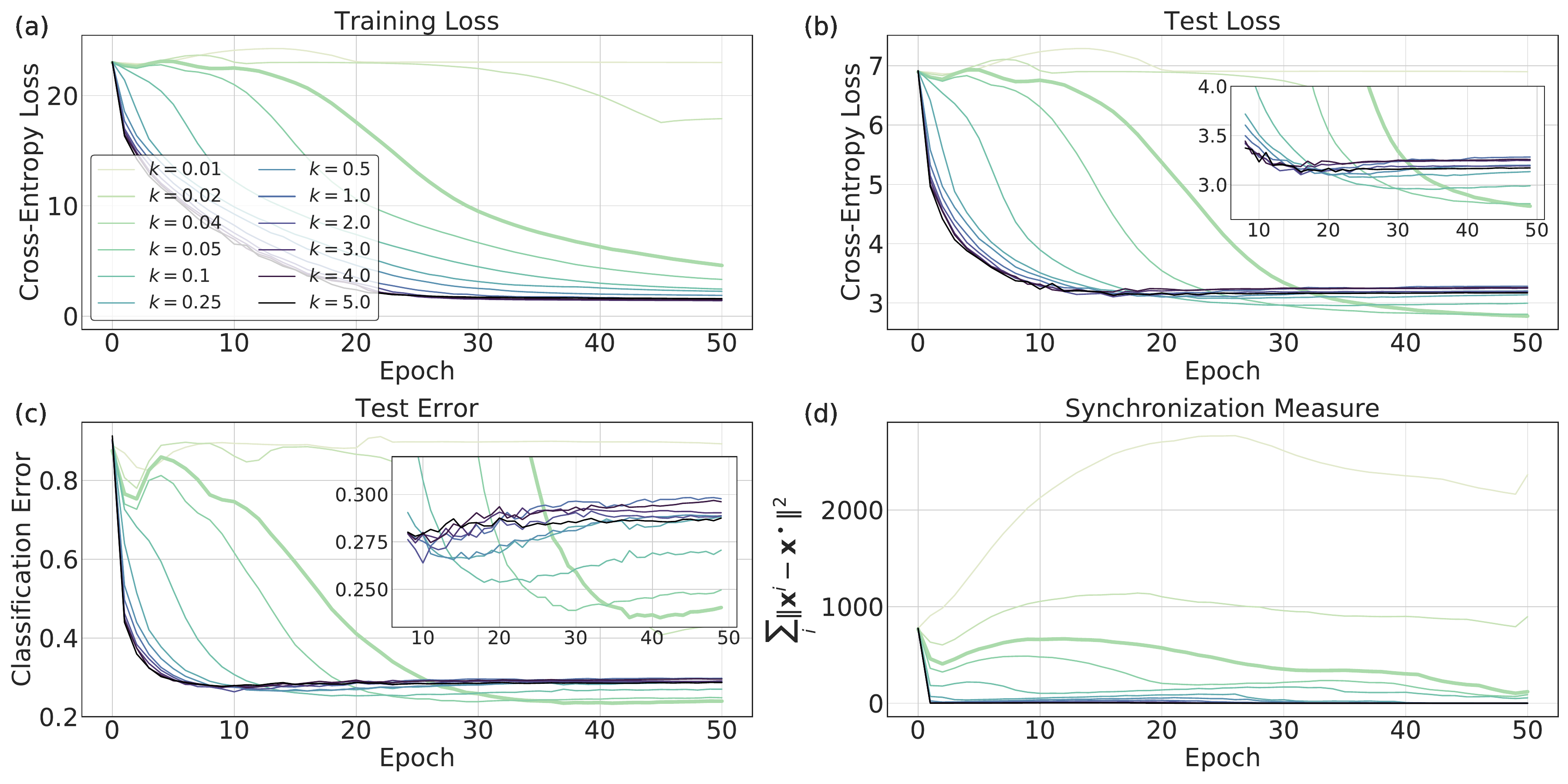}
    \caption{\normalsize The effect of varying $k$ on the learning procedure for the EASGD algorithm with $\eta = 0.05$ initially and $p=8$. In general, lower test errors and lower test loss values are seen for lower values of $k$, so long as convergence is still obtained. $k=0.01$ and $k=0.02$ have too little synchronization for the quorum variable to represent a meaningful average. Insets display a more fine-grained view near the end of learning. The best-performing curve is shown in bold. Best viewed in color}
    \label{fig:k_vals}
\end{figure}

Results for a comparison of QSGD and SD-QSGD are shown in Fig.~\ref{fig:sd} for $p=1, 4, 8, 16, 32$, and $64$ with $k=0.04$. QSGD is shown in solid lines while SD-QSGD is shown in dashed; color indicates the number of agents (see legend in Fig.~\ref{fig:sd}(a)). Note that $p=1$ simply corresponds to SGD for both SD-QSGD and QSGD, as the coupling term vanishes for a single agent. In both cases, we see significant improvement in accuracy as the number of agents increases, most likely due to an improved ability of the agents to explore the landscape along with a decrease in synchronization. The test loss and test error curves display interesting differences between the two algorithms: for $p =8$ and $p=16$, the state-dependent formalism obtains mildly improved generalization relative to QSGD, as expected by the bias towards minima with lower validation loss. QSGD performs better for $p=32$ and $p=64$; SD-QSGD does not converge for $p=64$.

\begin{figure}
    \centering
    \includegraphics[width=\textwidth]{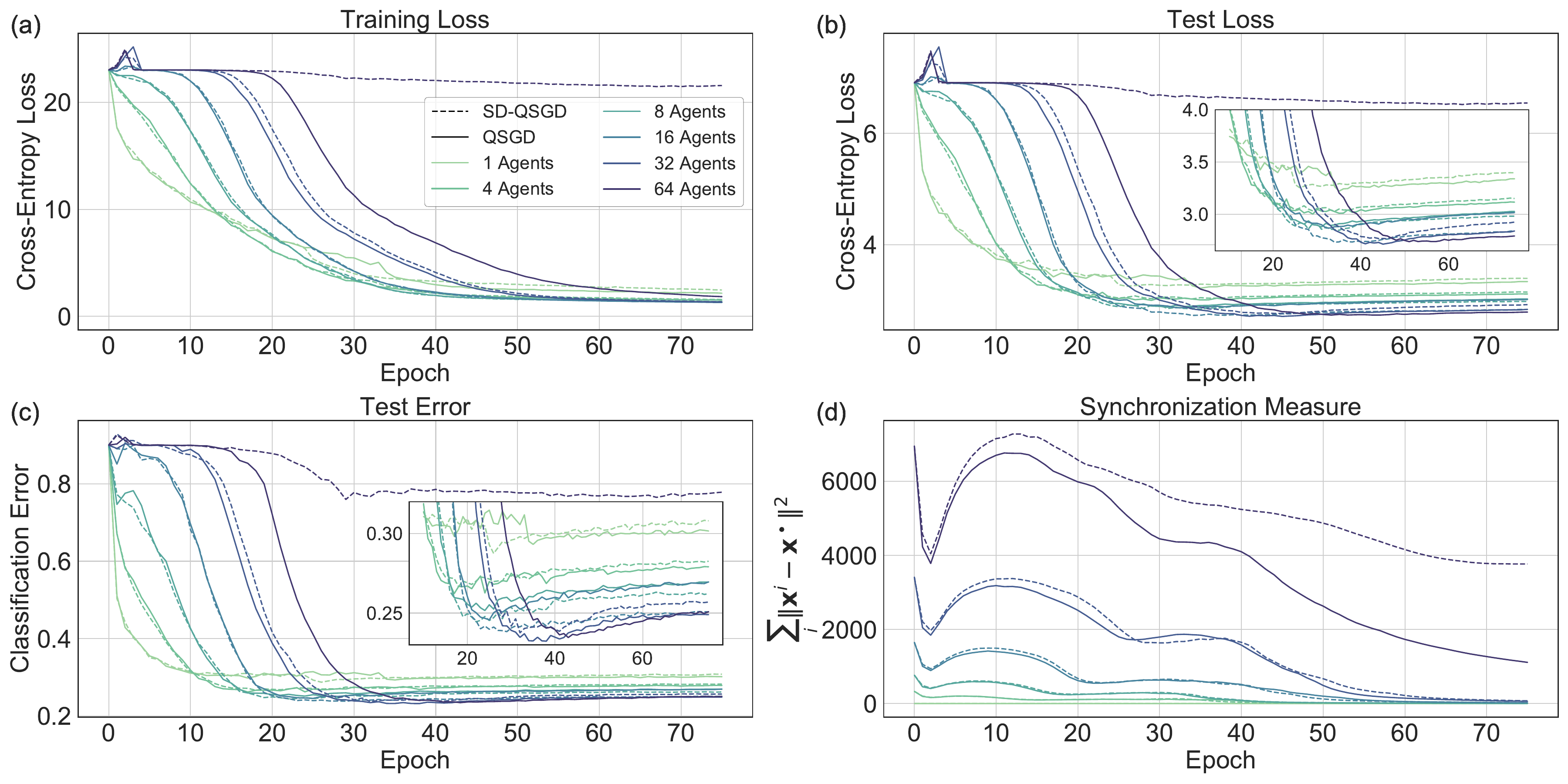}
    \caption{\normalsize A comparison of SD-QSGD using a spiking winner-take-all formalism (see text) to QSGD with a value of $k=0.04$. The state-dependent formalism obtains improved accuracy for the intermediate values of $p=8$ and $p=16$. QSGD and SD-QSGD perform similarly for $p=4$ and QSGD performs better for $p=32$. SD-QSGD does not converge for $p=64$ while QSGD does. Insets display a more fine-grained view near the end of learning. Best viewed in color.}
    \label{fig:sd}
\end{figure}

We display a comparison of QSGD and EASGD in Fig.~\ref{fig:qsgd_easgd_fac1}, again for $k=0.04$. QSGD tends to decrease the training loss further and more rapidly than EASGD; this is in line with earlier comments that, from an optimization perspective, the extra dynamics of the quorum variable offer no clear theoretical benefit. However, consistently across all experiments except for $p=16$ where it does not converge, EASGD generalizes better: the test loss is driven lower, and the test accuracy is higher than QSGD. A particularly interesting result is the single-agent case, where EASGD actually performs better than SGD with momentum\footnote{\normalsize Note that unlike QSGD with a single agent, EASGD with a single agent is a different algorithm than basic SGD. It can be seen as SGD coupled in feedback to a low-pass filter of its output.}. These observations suggest that the extra dynamics of the quorum variable may impose a form of implicit regularization which, to our knowledge, has not been observed before.

\begin{figure}
    \centering
    \includegraphics[width=\textwidth]{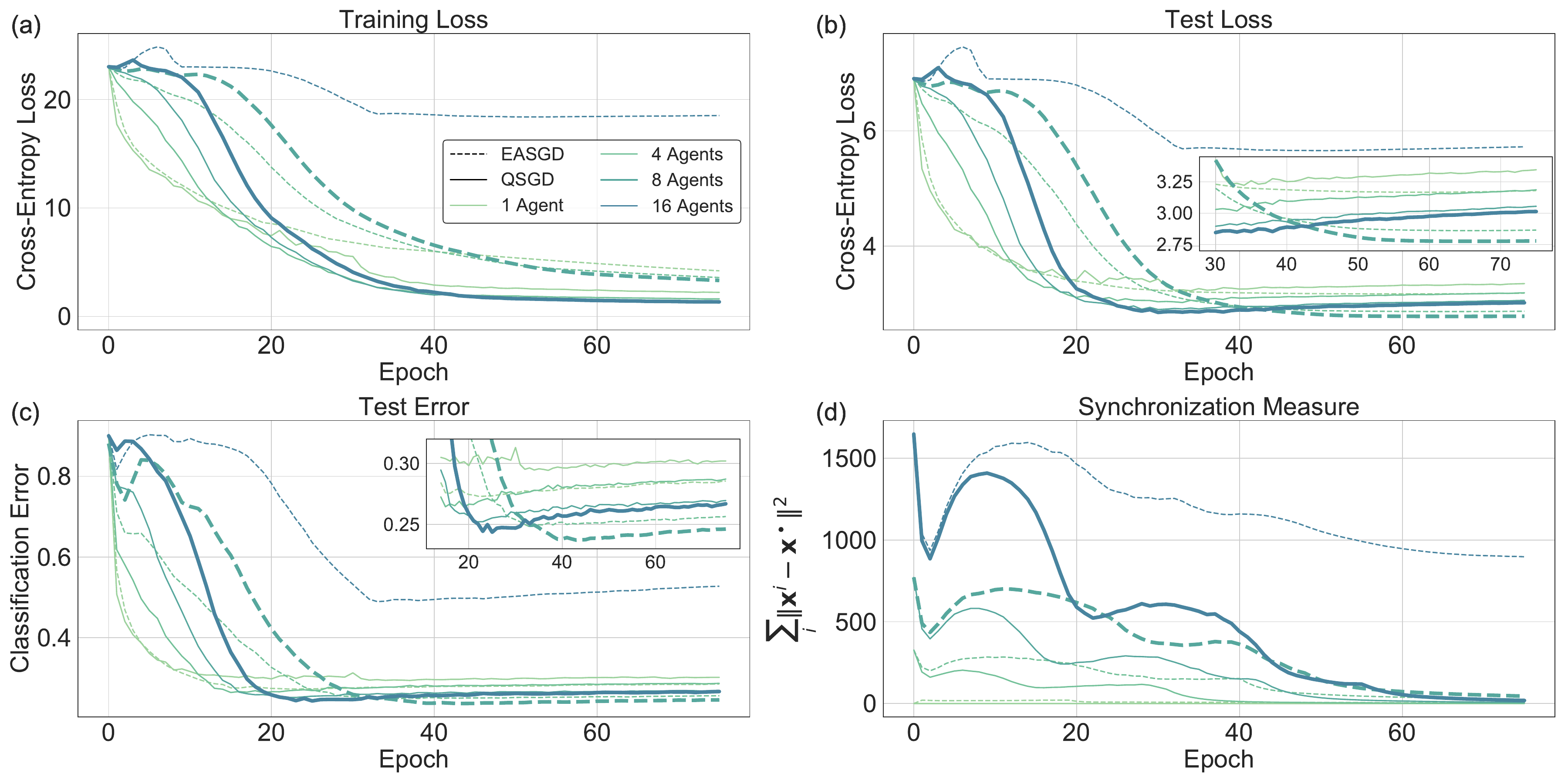}
    \caption{\normalsize A comparison of the performance of QSGD to EASGD with $k=0.04$ (see text). QSGD optimizes the training loss further and faster than EASGD, leading to overfitting. The two algorithms respond differently to fixed $k$ and have different levels of synchronization. For $p=16$, EASGD fails to converge, though QSGD continues to converge. Nevertheless, for fewer agents EASGD obtains improved performance. Insets display a more fine-grained view near the end of learning. The best-performing curves for each algorithm are shown in bold. Best viewed in color.}
    \label{fig:qsgd_easgd_fac1}
\end{figure}

Motivated by this observation, we now compare the $p=1$ EASGD algorithm with momentum, without momentum, and basic SGD with momentum in Fig.~\ref{fig:p1comp} across a range of initial learning rates. Each algorithm is initialized from the same location and each curve represents an average over three runs to eliminate stochastic variability. The momentum algorithms use $\delta = 0.9$ and the two EASGD variants use $k=0.054$. In general, EASGD with and without momentum (dashed and solid lines respectively) both achieve higher test accuracy than SGD with momentum (dotted lines). Surprisingly, EASGD without momentum often performs better than EASGD with momentum. 

To show that this trend is not an artifact of incorrectly choosing the momentum parameter, we have compiled additional data in Tab.~\ref{tab:delta_rslts} over a range of momentum parameters and learning rates. Each data point reported is again the result of an average over three independent runs, and each algorithm is initialized from the same location in each run. For simplicity, we simply report the testing loss and testing error, rather than the results on the training data. For all but one choice of $\eta$ and $\delta$, EASGD-WM outperforms both EASGD and MSGD in classification accuracy, demonstrating that the trend is robust to choice of learning rate and momentum value.

Much like SGD with momentum, single-agent EASGD-WM is a second-order system in-time. It also maintains a similar computational complexity, and only requires storing one extra set of parameters for the quorum variable.

\begin{figure}
        \begin{tabular}{cc}
        \begin{subfigure}{.46\textwidth}
            \centering
            \includegraphics[width=\textwidth]{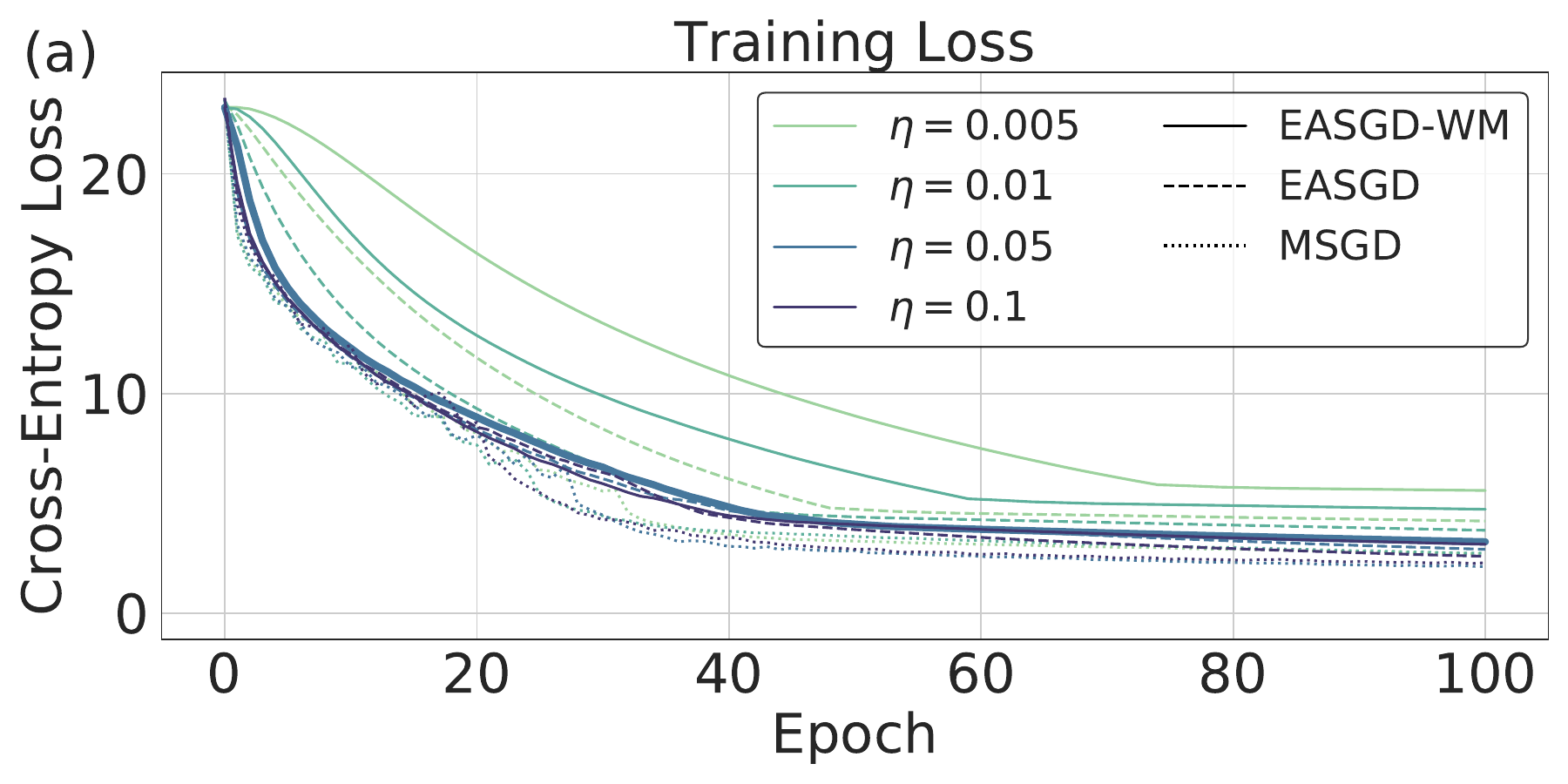}
        \end{subfigure} &
        \begin{subfigure}{.46\textwidth}
            \centering
            \includegraphics[width=\textwidth]{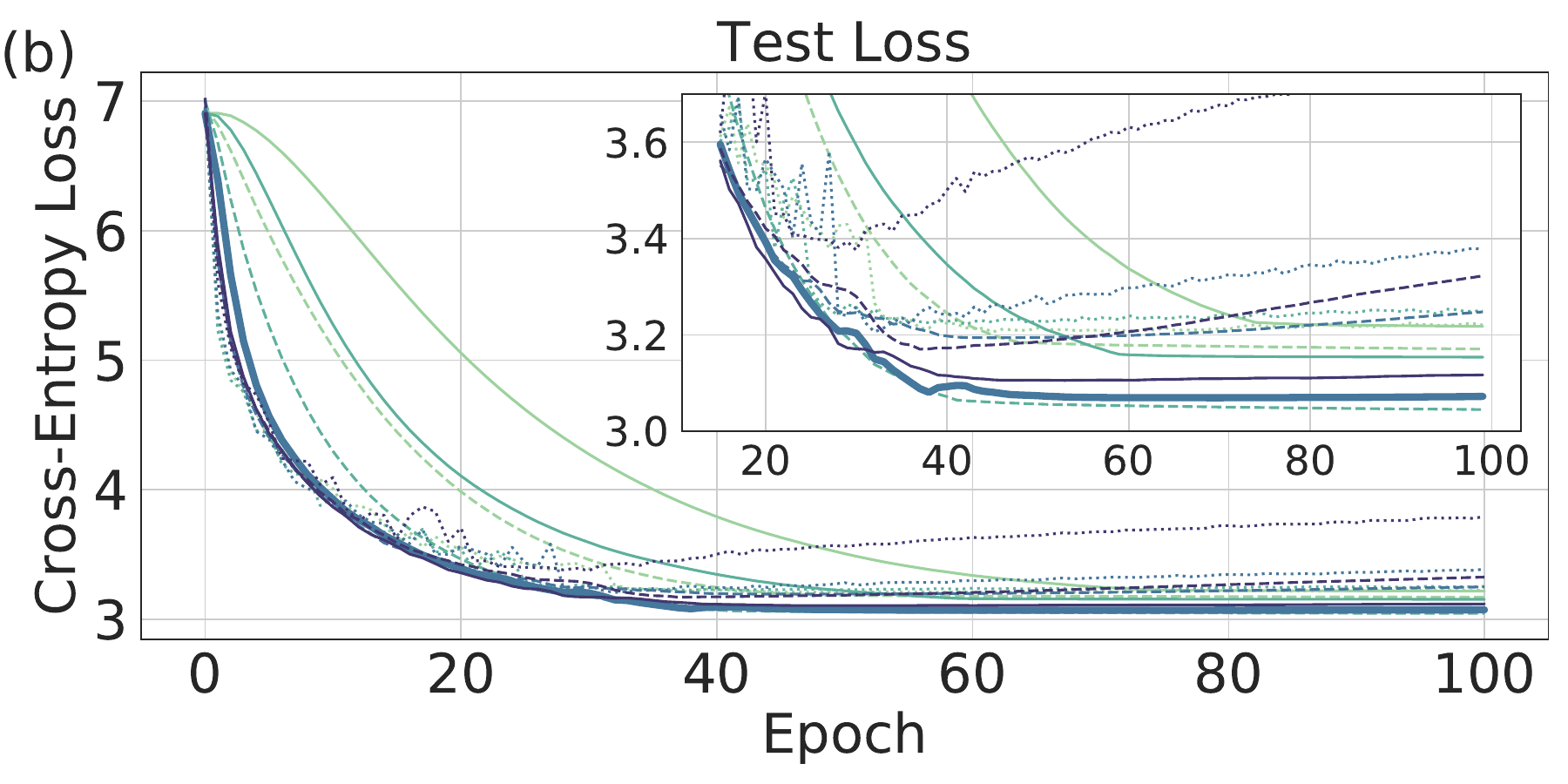}
        \end{subfigure} \\
        \multicolumn{2}{c}{
            \begin{subfigure}{.46\textwidth}
                \centering
                \includegraphics[width=\textwidth]{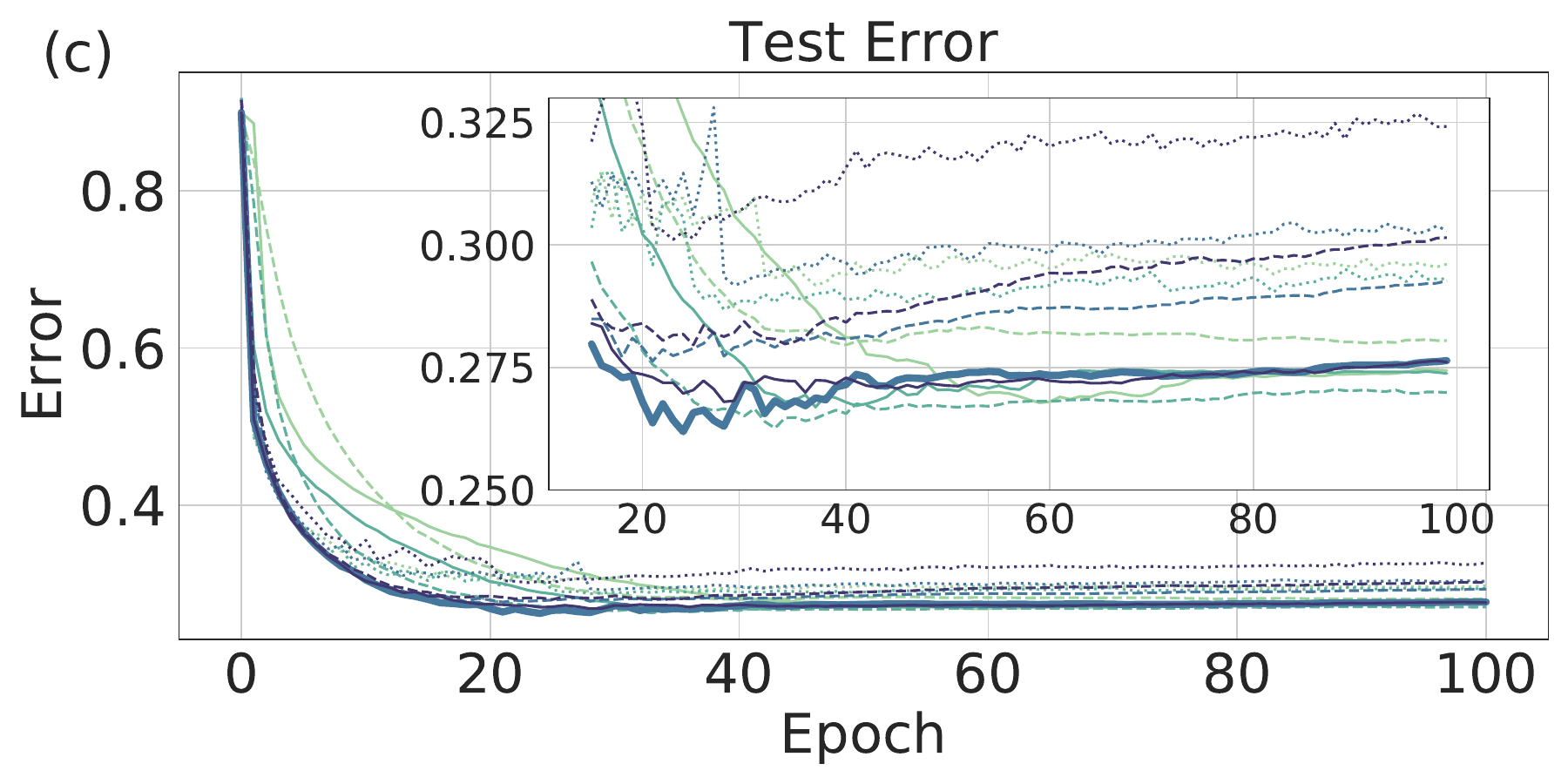}
            \end{subfigure}}\\
    \end{tabular}
    \caption{\normalsize A comparison of EASGD, EASGD without momentum (EASGD-WM) and SGD with momentum (MSGD) over a range of learning rates for momentum parameter $\delta = 0.9$ and coupling gain $k=0.054$. Surprisingly, EASGD and EASGD-WM perform better than MSGD in general, and in many cases, EASGD-WM performs better than EASGD. This motivates considering alternative dynamics for the quorum variable even for non-distributed optimization. Insets display a more fine-grained view near the end of learning. The best-performing curve is shown in bold. Best viewed in color.}
    \label{fig:p1comp}
\end{figure}

Indeed, this motivates a new class of second-order in-time algorithms for non-distributed optimization given by the feedback interconnection
\begin{align}
    \dot{\bx}  &= -\nabla f(\bx) + k (\txb - \bx),\\
    \dot{\txb} &= \mathbf{g}(\txb, \bx),
    \label{eqn:gen_filt}
\end{align}
where $\mathbf{g}$ represents arbitrary dynamics for the quorum variable \citep{russo-slot}, and in general might be chosen as a nonlinear filter. The simple linear filter $\mathbf{g}(\txb, \bx) = k(\bx - \txb)$ recovers EASGD. Fig.~\ref{fig:qsgd_easgd_fac1} shows that, while EASGD obtains better performance than QSGD, QSGD maintains better stability properties. Designing nonlinear filters $\mathbf{g}$ that can combine the regularization of EASGD with the stability of QSGD is an interesting direction of future research.

\begin{table}[]
\resizebox{\textwidth}{!}{%
\begin{tabular}{llllllllllllll}
\cline{3-14}
                                                & \multicolumn{1}{l|}{}         & \multicolumn{6}{c|}{Minimum Test Loss}                                                                                                                                       & \multicolumn{6}{c|}{Minimum Error}                                                                                                                                           \\ \cline{3-14} 
                                                & \multicolumn{1}{l|}{}         & \multicolumn{1}{l|}{$\delta=.1$} & \multicolumn{1}{l|}{$\delta=.25$} & \multicolumn{1}{l|}{$\delta=.5$} & \multicolumn{1}{l|}{$\delta=.75$} & \multicolumn{1}{l|}{$\delta=.9$} & \multicolumn{1}{l|}{$\delta=.99$} & \multicolumn{1}{l|}{$\delta=.1$} & \multicolumn{1}{l|}{$\delta=.25$} & \multicolumn{1}{l|}{$\delta=.5$} & \multicolumn{1}{l|}{$\delta=.75$} & \multicolumn{1}{l|}{$\delta=.9$} & \multicolumn{1}{l|}{$\delta=.99$} \\ \hline
\multicolumn{1}{|l|}{\multirow{3}{*}{$\eta=.005$}} & \multicolumn{1}{l|}{EASGD-WM} & \multicolumn{1}{l|}{\textbf{4.25}} & \multicolumn{1}{l|}{\textbf{4.26}}  & \multicolumn{1}{l|}{\textbf{4.28}} & \multicolumn{1}{l|}{4.29}  & \multicolumn{1}{l|}{3.22} & \multicolumn{1}{l|}{3.24}  & \multicolumn{1}{l|}{\textbf{.267}} & \multicolumn{1}{l|}{\textbf{.266}}  & \multicolumn{1}{l|}{\textbf{.264}} & \multicolumn{1}{l|}{\textbf{.269}}  & \multicolumn{1}{l|}{\textbf{.268}} & \multicolumn{1}{l|}{\textbf{.270}}  \\ \cline{2-14} 
\multicolumn{1}{|l|}{}                          & \multicolumn{1}{l|}{EASGD}    & \multicolumn{1}{l|}{4.72} & \multicolumn{1}{l|}{4.83}  & \multicolumn{1}{l|}{4.56} & \multicolumn{1}{l|}{\textbf{4.28}}  & \multicolumn{1}{l|}{\textbf{3.17}} & \multicolumn{1}{l|}{\textbf{3.11}}  & \multicolumn{1}{l|}{.304} & \multicolumn{1}{l|}{.313}  & \multicolumn{1}{l|}{.301} & \multicolumn{1}{l|}{.282}  & \multicolumn{1}{l|}{.280} & \multicolumn{1}{l|}{.277}  \\ \cline{2-14} 
\multicolumn{1}{|l|}{}                          & \multicolumn{1}{l|}{MSGD}     & \multicolumn{1}{l|}{4.75} & \multicolumn{1}{l|}{4.87}  & \multicolumn{1}{l|}{4.64} & \multicolumn{1}{l|}{4.33}  & \multicolumn{1}{l|}{3.21} & \multicolumn{1}{l|}{3.29}  & \multicolumn{1}{l|}{.310} & \multicolumn{1}{l|}{.323}  & \multicolumn{1}{l|}{.306} & \multicolumn{1}{l|}{.286}  & \multicolumn{1}{l|}{.292} & \multicolumn{1}{l|}{.295}  \\ \hline
                                                &                               &                           &                            &                           &                            &                           &                            &                           &                            &                           &                            &                           &                            \\ \hline
\multicolumn{1}{|l|}{\multirow{3}{*}{$\eta=.01$}}  & \multicolumn{1}{l|}{EASGD-WM} & \multicolumn{1}{l|}{\textbf{4.09}} & \multicolumn{1}{l|}{\textbf{5.15}}  & \multicolumn{1}{l|}{\textbf{4.03}} & \multicolumn{1}{l|}{\textbf{4.12}}  & \multicolumn{1}{l|}{3.15} & \multicolumn{1}{l|}{\textbf{3.10}}  & \multicolumn{1}{l|}{\textbf{.262}} & \multicolumn{1}{l|}{\textbf{.259}}  & \multicolumn{1}{l|}{\textbf{.253}} & \multicolumn{1}{l|}{\textbf{.261}}  & \multicolumn{1}{l|}{.267} & \multicolumn{1}{l|}{\textbf{.257}}  \\ \cline{2-14} 
\multicolumn{1}{|l|}{}                          & \multicolumn{1}{l|}{EASGD}    & \multicolumn{1}{l|}{4.57} & \multicolumn{1}{l|}{5.75}  & \multicolumn{1}{l|}{4.27} & \multicolumn{1}{l|}{4.14}  & \multicolumn{1}{l|}{\textbf{3.04}} & \multicolumn{1}{l|}{3.22}  & \multicolumn{1}{l|}{.297} & \multicolumn{1}{l|}{.300}  & \multicolumn{1}{l|}{.280} & \multicolumn{1}{l|}{.275}  & \multicolumn{1}{l|}{\textbf{.263}} & \multicolumn{1}{l|}{.283}  \\ \cline{2-14} 
\multicolumn{1}{|l|}{}                          & \multicolumn{1}{l|}{MSGD}     & \multicolumn{1}{l|}{4.59} & \multicolumn{1}{l|}{5.81}  & \multicolumn{1}{l|}{4.48} & \multicolumn{1}{l|}{4.46}  & \multicolumn{1}{l|}{3.22} & \multicolumn{1}{l|}{3.33}  & \multicolumn{1}{l|}{.300} & \multicolumn{1}{l|}{.307}  & \multicolumn{1}{l|}{.294} & \multicolumn{1}{l|}{.294}  & \multicolumn{1}{l|}{.287} & \multicolumn{1}{l|}{.301}  \\ \hline
                                                &                               &                           &                            &                           &                            &                           &                            &                           &                            &                           &                            &                           &                            \\ \hline
\multicolumn{1}{|l|}{\multirow{3}{*}{$\eta=.05$}}  & \multicolumn{1}{l|}{EASGD-WM} & \multicolumn{1}{l|}{\textbf{3.95}} & \multicolumn{1}{l|}{\textbf{3.96}}  & \multicolumn{1}{l|}{\textbf{3.86}} & \multicolumn{1}{l|}{\textbf{3.97}}  & \multicolumn{1}{l|}{\textbf{3.07}} & \multicolumn{1}{l|}{\textbf{3.00}}  & \multicolumn{1}{l|}{\textbf{.252}} & \multicolumn{1}{l|}{\textbf{.258}}  & \multicolumn{1}{l|}{\textbf{.250}} & \multicolumn{1}{l|}{\textbf{.255}}  & \multicolumn{1}{l|}{\textbf{.262}} & \multicolumn{1}{l|}{\textbf{.253}}  \\ \cline{2-14} 
\multicolumn{1}{|l|}{}                          & \multicolumn{1}{l|}{EASGD}    & \multicolumn{1}{l|}{4.41} & \multicolumn{1}{l|}{4.27}  & \multicolumn{1}{l|}{4.05} & \multicolumn{1}{l|}{4.06}  & \multicolumn{1}{l|}{3.19} & \multicolumn{1}{l|}{4.04}  & \multicolumn{1}{l|}{.286} & \multicolumn{1}{l|}{.283}  & \multicolumn{1}{l|}{.265} & \multicolumn{1}{l|}{.267}  & \multicolumn{1}{l|}{.276} & \multicolumn{1}{l|}{.417}  \\ \cline{2-14} 
\multicolumn{1}{|l|}{}                          & \multicolumn{1}{l|}{MSGD}     & \multicolumn{1}{l|}{4.46} & \multicolumn{1}{l|}{4.57}  & \multicolumn{1}{l|}{4.48} & \multicolumn{1}{l|}{4.43}  & \multicolumn{1}{l|}{3.21} & \multicolumn{1}{l|}{6.91}  & \multicolumn{1}{l|}{.294} & \multicolumn{1}{l|}{.307}  & \multicolumn{1}{l|}{.295} & \multicolumn{1}{l|}{.290}  & \multicolumn{1}{l|}{.292} & \multicolumn{1}{l|}{0.9}   \\ \hline
                                                &                               &                           &                            &                           &                            &                           &                            &                           &                            &                           &                            &                           &                            \\ \hline
\multicolumn{1}{|l|}{\multirow{3}{*}{$\eta=.1$}}   & \multicolumn{1}{l|}{EASGD-WM} & \multicolumn{1}{l|}{\textbf{4.08}} & \multicolumn{1}{l|}{\textbf{4.01}}  & \multicolumn{1}{l|}{\textbf{4.04}} & \multicolumn{1}{l|}{\textbf{4.05}}  & \multicolumn{1}{l|}{\textbf{3.11}} & \multicolumn{1}{l|}{\textbf{3.15}}  & \multicolumn{1}{l|}{\textbf{.267}} & \multicolumn{1}{l|}{\textbf{.264}}  & \multicolumn{1}{l|}{\textbf{.268}} & \multicolumn{1}{l|}{\textbf{.265}}  & \multicolumn{1}{l|}{\textbf{.268}} & \multicolumn{1}{l|}{\textbf{.269}}  \\ \cline{2-14} 
\multicolumn{1}{|l|}{}                          & \multicolumn{1}{l|}{EASGD}    & \multicolumn{1}{l|}{4.24} & \multicolumn{1}{l|}{4.23}  & \multicolumn{1}{l|}{4.14} & \multicolumn{1}{l|}{4.13}  & \multicolumn{1}{l|}{3.17} & \multicolumn{1}{l|}{6.91}  & \multicolumn{1}{l|}{.282} & \multicolumn{1}{l|}{.283}  & \multicolumn{1}{l|}{.277} & \multicolumn{1}{l|}{.272}  & \multicolumn{1}{l|}{.280} & \multicolumn{1}{l|}{0.9}   \\ \cline{2-14} 
\multicolumn{1}{|l|}{}                          & \multicolumn{1}{l|}{MSGD}     & \multicolumn{1}{l|}{4.62} & \multicolumn{1}{l|}{4.55}  & \multicolumn{1}{l|}{4.22} & \multicolumn{1}{l|}{4.47}  & \multicolumn{1}{l|}{3.38} & \multicolumn{1}{l|}{6.91}  & \multicolumn{1}{l|}{.288} & \multicolumn{1}{l|}{.307}  & \multicolumn{1}{l|}{.287} & \multicolumn{1}{l|}{.288}  & \multicolumn{1}{l|}{.301} & \multicolumn{1}{l|}{0.9}   \\ \hline
\end{tabular}}
\vspace{5mm}
\caption{\normalsize Comparison of minimum test loss achieved, minimum error achieved, and minimum training loss achieved for EASGD-WM, EASGD, and MSGD on the CIFAR-10 dataset (each with $p=1$, providing details on the effect of hyperparameter choices not seen in Fig.~\ref{fig:p1comp}). Each experiment was run three times and the minimum was taken over the average trajectory. In each run, the algorithms were initialized from the same starting location. Surprisingly, EASGD-WM consistently achieves the lowest test error (all but one setting) and the lowest test loss (all but four settings) in comparison to EASGD and MSGD. For high learning rate and high $\delta$, MSGD and EASGD eventually run into convergence issues, while EASGD-WM does not (error of $.9$ and test loss of $6.91$ indicate convergence issues).}
\label{tab:delta_rslts}
\end{table}

Returning to the distributed case, Fig.~\ref{fig:qsgd_easgd_fac1}(d) shows that EASGD and QSGD respond differently to the choice of $k$\footnote{\normalsize Fig.~\ref{fig:qsgd_easgd_fac1}(d) shows the distance from $\txb$ for EASGD. The distance from $\comb$ for EASGD is nearly identical.}. EASGD is less synchronized than QSGD in all cases. Hence, in the context of Fig.~\ref{fig:k_vals}, a possible explanation for the improved performance of EASGD when compared to QSGD is simply the observation that it tends to remain less synchronized.

To answer this question, we use a scaling factor $k_{EASGD} = r\times k_{QSGD}$ to roughly match the levels of synchronization between EASGD and QSGD. Results for $r = 1.35$ are shown in Fig.~\ref{fig:qsgd_easgd_fac1p35}, and the synchronization curves are either approximately equal or EASGD remains more synchronized across all values of $p$. Additional values of $p=32$ and $p=64$ are shown, and EASGD now converges for all attempted values of $p < 64$. QSGD continues to perform worse than EASGD on the test data due to an increased tendency to overfit. As the number of agents is increased, QSGD improves up to $p=32$; $p=64$ obtains roughly the same test performance. EASGD improves up to around $p=16$ and does not converge for $p=64$ (see Fig.~\ref{fig:qsgd_easgd_fac1p35}(a) -- the curves in (b) and (d) are covered by the insets, but EASGD obtains roughly 55\% testing accuracy). In general, EASGD with $p$ agents obtains roughly the same performance as QSGD with $2p$ agents. Interestingly, Fig.~\ref{fig:qsgd_easgd_fac1p35}(d) shows that the high $p$ stability issues for EASGD are not simply due to a lack of synchronization, as EASGD actually remains more synchronized than QSGD for $p=64$ for much of the training time. We offer a simple possible explanation for these stability issues in the SI by analyzing discrete-time optimization of a one-dimensional quadratic objective. Another explanation is afforded by Thms.~\ref{thm:qsgd} and \ref{thm:easgd_converge}, which reveal poor scaling with $p$ of both terms in the bound for EASGD when compared to QSGD. Together, these observations highlight stability issues in both continuous- and discrete-time.

\begin{figure}
    \centering
    \includegraphics[width=\textwidth]{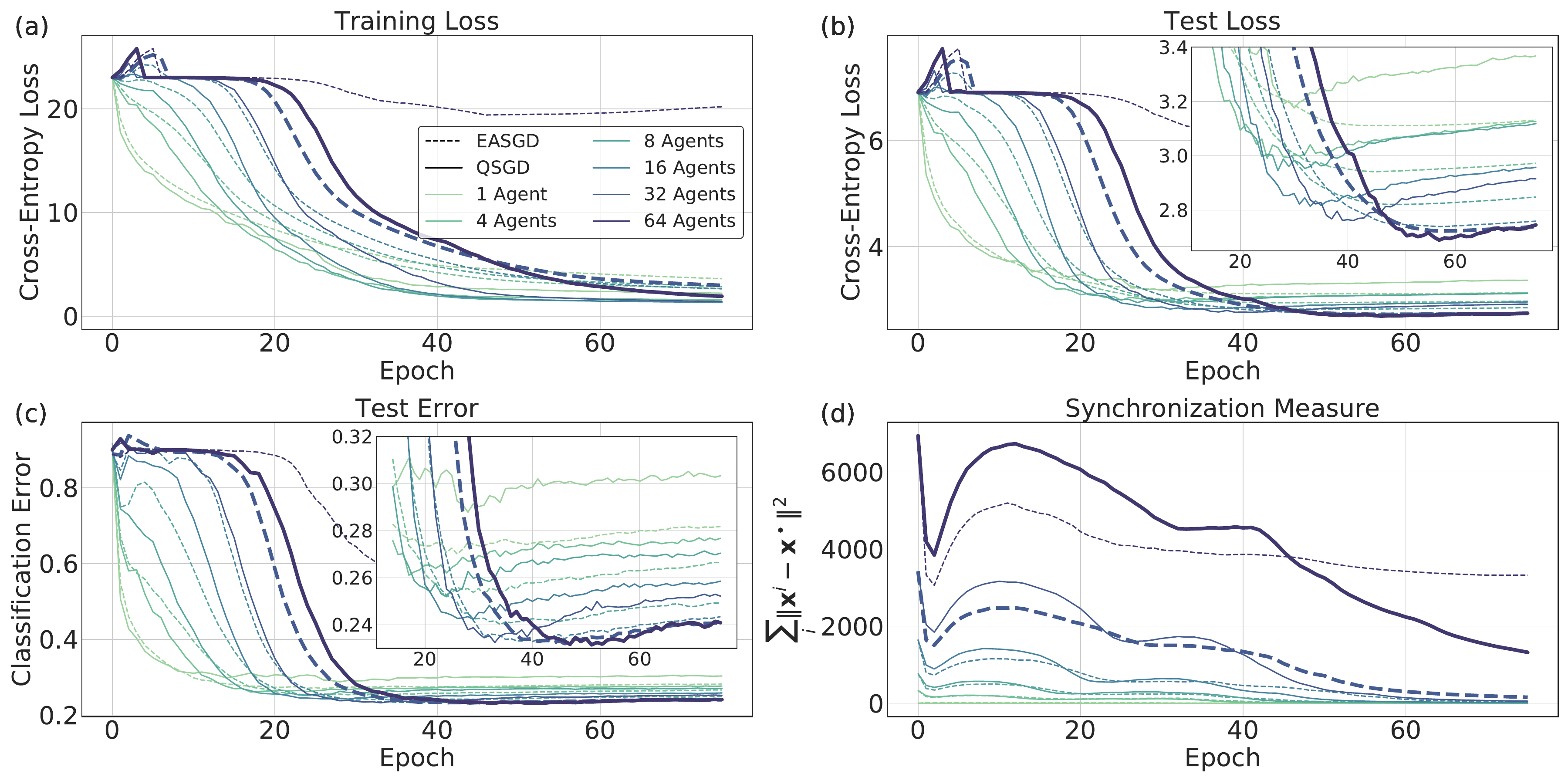}
    \caption{\normalsize A comparison of QSGD and EASGD with $k_{QSGD} = 0.04$ and $k_{EASGD} = r\times k_{QSGD}$ with $r=1.35$. In this case, EASGD converges and performs better for all values of $p$ up to $p=64$, where it again fails to converge. Nevertheless, performance for EASGD with $p=16$ and $p=32$ approximately matches that of QSGD with $p=64$. Insets display a more fine-grained view near the end of learning. The best-performing curves for each algorithm are shown in bold. Best viewed in color.}
    \label{fig:qsgd_easgd_fac1p35}
\end{figure}

As discussed in the text and the description of the experimental setup, our theory allows for the agents to be initialized in different locations, and for the agents to use distinct learning rates through individual learning rate schedules. In the original work on EASGD, it was postulated that starting the agents at different locations would break symmetry and lead to instability \citep{EASGD}. Similarly, a single learning rate was used for all agents. The above simulations demonstrate that starting from distinct locations and decreasing the learning rate on an individual basis is non-problematic. We show in Fig.~\ref{fig:one_loc} that starting from a single location leads to decreased performance. Surprisingly, Fig.~\ref{fig:one_loc} also highlights that initializing the agents from multiple locations is critical for optimal improvement as the number of agents is increased.

\begin{figure}
    \centering
    \includegraphics[width=\textwidth]{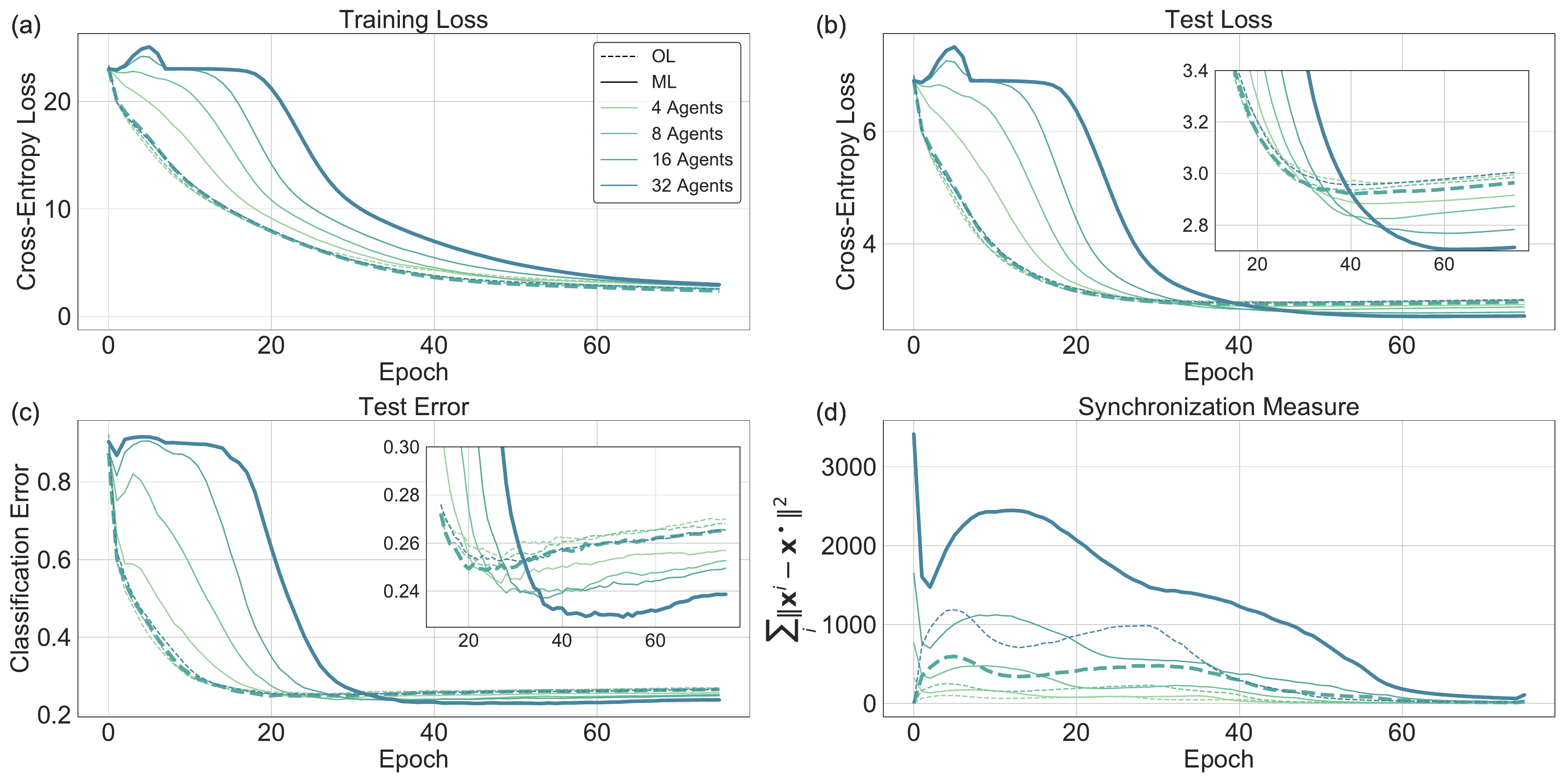}
    \caption{\normalsize A comparison between starting the agents from multiple locations (ML) and one location (OL) for EASGD with a value of $k=0.054$. Starting from multiple locations exhibits better test accuracy, lower test loss, and greater improvement as the number of agents is increased. Insets display a more fine-grained view near the end of learning. The best-performing curves for each setting are shown in bold. Best viewed in color.}
    \label{fig:one_loc}
\end{figure}

\section{Conclusion}
\label{sec:conc}
In this paper, we presented a continuous-time analysis of distributed stochastic gradient algorithms within the framework of stochastic nonlinear contraction theory. Through analogy with quorum sensing mechanisms, we analyzed the effect of synchronization of the individual SGD agents on the noise generated by their stochastic gradient approximations. We demonstrated that synchronization can effectively reduce the noise felt by each of the individual agents and by their spatial mean. We further demonstrated that synchronization can be seen to reduce the amount of smoothing imposed by SGD on the loss function. Through simulations on model non-convex optimization problems, we provided insight into how the distributed and coupled setting affects convergence to minima of the smoothed loss and the true loss. We introduced a new distributed algorithm, QSGD, and proved convergence results for a strongly convex objective for QSGD, QSGD with momentum, and EASGD. We further introduced a state-dependent variant of QSGD and constructed one specific example of the algorithm to show how the formalism can be used to bias exploration. We presented experiments on deep neural networks, and compared the properties of QSGD, SD-QSGD, and EASGD for generalization performance. We noted an interesting regularizing property of EASGD even in the single-agent case, and compared it to basic SGD with momentum, showing that it can lead to improved generalization. Research into similar higher-order in time optimization algorithms formed as coupled dynamical systems is an interesting direction of future work.

\section*{Acknowledgments}
N. M. Boffi was supported by a Department of Energy Computational Science Graduate Fellowship. We graciously thank the reviewers for helpful feedback and for suggestions to improve the work.

\appendix
\section{Interaction between synchronization and noise: extra quorum dynamics}
\label{ssec:sync_noise_analysis_EASGD}
We now provide a mathematical characterization of how synchronization reduces the noise felt by the agents with arbitrary quorum dynamics. This is a generalization of what was shown in Sec.~\ref{ssec:sync_noise_analysis}, and does not depend on the dynamics of the quorum variable. In addition to the assumptions stated in Sec.~\ref{ssec:assump}, we require that the gradient workers are stochastically contracting with rate $\lambda = k - \bal$ and bound $\frac{\eta}{b}C$, so that the synchronization condition (\ref{eqn:sync_com}) derived in Sec.~\ref{ssec:sync_measure} can be applied. For completeness, we consider,
\begin{align}
    d\bx^i &= \left(-\nabla f(\bx^i) + k \left(\txb - \bx^i\right)\right)dt + \sqrt{\frac{\eta}{b}}\bB(\bx^i) d\bW^i,\\
    \label{eqn:EASGD_cont_app_1}
    d\txb &= \mathbf{g}(\comb, \txb)dt.
\end{align}
As in the main text, we define $x^\bullet = \frac{1}{p}\sum_i x_i$. Adding up the stochastic dynamics in (\ref{eqn:EASGD_cont_app_1}), we find
\begin{equation}
    d\comb = \left[-\frac{1}{p}\sum_i \nabla f(\bx^i) + k (\txb - \comb)\right]dt + \sqrt{\frac{\eta}{bp^2}}\sum_i \bB(\bx^i) d\bW^i\nonumber.
    \label{eqn_app:com_ito_1}
\end{equation}
We then define
\begin{equation}
    \beps = -\frac{1}{p}\sum_i \nabla f(\bx^i) + \nabla f(\comb)\nonumber,
\end{equation}
so that we can rewrite
\begin{equation}
    d\comb = \left[-\nabla f(\comb) + \beps + k (\txb - \comb)\right]dt + \sqrt{\frac{\eta}{bp^2}}\sum_i \bB(\bx^i) d\bW^i\nonumber.
    \label{eqn_app:com_ito_2}
\end{equation}
Applying the Taylor formula with integral remainder to the components of the gradient $\left(-\nabla f(x)\right)_j$, we have, with $\bF_j$ denoting the gradient of $\left(-\nabla f(x)\right)_j$, and $\bH_j$ denoting its Hessian,
\begin{align}
    \left(-\nabla f(\bx^i)\right)_j &+ \left(\nabla f(\comb)\right)_j - \bF_j^T(\comb)(\bx^i - \comb) \nonumber\\
    &=\int_0^1 (1-s)\left(\bx^i - \comb\right)^T\bH_j\left((1-s)\bx^i + s \comb\right)\left(\bx^i - \comb\right)\nonumber.
    \label{eqn_app:bound}
\end{align}
Summing over $i$ and applying the assumed bound $\bH_j \leq Q\bI$ leads to the inequality
\begin{equation}
    \left|\sum_i \left[\left(-\nabla f(\bx^i)\right)_j + \left(\nabla f(\comb)\right)_j\right]\right| \leq \frac{Q}{2}\sum_i\Vert \bx^i - \comb\Vert^2\nonumber.
\end{equation}
The left-hand side of the above inequality is $p|\beps_j|$. Squaring both sides and summing over $j$ provides a bound on $p^2\Vert \beps\Vert^2$. Squaring both sides, performing this sum, noting that $j$ runs from $1$ to $n$, taking a square root, taking an expectation over the noise, and using the synchronization condition in (\ref{eqn:sync_com}), 
\begin{equation}
    \mathbb{E}\left[\Vert\beps\Vert\right] \leq \frac{(p-1) \eta Q C \sqrt{n}}{4 p b\left(k - \bal\right)}\nonumber.
    \label{eqn_app:eps_bound}
\end{equation}
As a sum of $p$ independent Gaussian random variables with mean zero and standard deviations $\frac{\eta}{b p^2}\bSig(\bx^i)$, the quantity
\begin{equation}
    \sqrt{\frac{\eta}{bp^2}}\sum_i \bB(\bx^i) d\bW^i = \sqrt{\frac{\eta}{bp^2}}\bT d\bW\nonumber
\end{equation}
can be rewritten as a single Gaussian random variable with $\bT\bT^T = \sum_i \bSig(\bx^i)$ as in the main text. Thus, for a given noise covariance $\bSig$ and corresponding bound $C$, the difference between the dynamics followed by $\comb$ and the noise-free dynamics 
\begin{align}
    \dot{\bx}^i_{nf} &= -\nabla f(\bx^i_{nf}) + k\left(\txb - \bx^i_{nf}\right)\nonumber,\\
    \dot{\txb} &= \mathbf{g}(\txb, \comb_{nf})\nonumber,
\end{align}
tends to zero almost surely as $k \rightarrow \infty$ and $p \rightarrow \infty$. The limit $k \rightarrow \infty$ is needed to increase the degree of synchronization to eliminate the effect of $\beps$ on $\comb$, while the limit $p \rightarrow \infty$ is needed to eliminate the effect of the additive noise.
\bibliography{refs}

\end{document}